\documentclass[10pt]{amsart}
\usepackage{mathrsfs}
\usepackage{amssymb}
\usepackage[all,pdf]{xy}
\usepackage{titletoc}
\usepackage{amscd}
\usepackage{amsmath, amssymb}
\usepackage{amsfonts}
\usepackage[colorlinks,linkcolor=blue,citecolor=blue, pdfstartview=FitH]{hyperref}
\usepackage{arydshln}

\numberwithin{equation}{section}

\theoremstyle{plain}
\newtheorem{thm}{Theorem}[section]
\newtheorem{theorem}[thm]{Theorem}
\newtheorem{thmx}{Theorem}

\newtheorem{lemma}[thm]{Lemma}
\newtheorem{corollary}[thm]{Corollary}
\newtheorem{proposition}[thm]{Proposition}

\newtheorem{conjecture}[thm]{Conjecture}
\newtheorem*{conj}{Conjecture}
\newtheorem*{cor}{Corollary}

\theoremstyle{definition}

\newtheorem{remark}[thm]{Remark}

\newtheorem{definition}[thm]{Definition}

\newtheorem{example}[thm]{Example}

\newcommand{\sE}{{\mathcal E}}
\newcommand{\sF}{{\mathcal F}}
\newcommand{\sG}{{\mathcal G}}
\newcommand{\sH}{{\mathcal H}}
\newcommand{\sI}{{\mathcal I}}

\newcommand{\sO}{{\mathcal O}}
\newcommand{\sP}{{\mathcal P}}

\newcommand{\sR}{{\mathcal R}}

\newcommand{\sT}{{\mathcal T}}

\newcommand{\sV}{{\mathcal V}}

\newcommand{\A}{{\mathbb A}}
\newcommand{\B}{{\mathbb B}}
\newcommand{\C}{{\mathbb C}}
\newcommand{\D}{{\mathbb D}}
\newcommand{\E}{{\mathbb E}}

\newcommand{\G}{{\mathbb G}}
\renewcommand{\H}{{\mathbb H}}

\renewcommand{\L}{{\mathbb L}}

\newcommand{\N}{{\mathbb N}}
\renewcommand{\P}{{\mathbb P}}
\newcommand{\Q}{{\mathbb Q}}
\newcommand{\R}{{\mathbb R}}

\newcommand{\V}{{\mathbb V}}

\renewcommand{\S}{{\mathbb S}}

\newcommand{\Z}{{\mathbb Z}}

\newcommand{\Hom}{{\mathrm{Hom}}}

\newcommand{\Ker}{{\mathrm{Ker}}}

\newcommand{\Coh}{\mathrm{Coh}}

\newcommand{\ua}{\underline{\alpha}}

\newcommand{\bdot}{\boldsymbol{\cdot}}
\newcommand{\lrw}{\longrightarrow}

\newcommand{\CQ}{\mathbb{C}Q}
\newcommand{\Seq}{\mathrm{Seq}}
\newcommand{\LHS}{\mathrm{LHS}}

\usepackage{tikz-cd}
\usepackage{tikz}
\usetikzlibrary{decorations.pathmorphing, decorations.pathreplacing, calc}
\usepackage[all]{xy}
\usepackage{fancyhdr}
\begin{document}

\title{KLR-Schur algebra of coherent sheaves on $\mathbb{P}^1$:\\ Tilting and PBW bases }
\author{Olivier Schiffmann}
\address[Olivier Schiffmann]{Laboratoire de Math\'ematiques d'Orsay, Universit\'e de Paris-Saclay, B\^at.425, 91405 Orsay Cedex, France, UMR8628 (CNRS), And Simion Stoilow Institute of Mathematics, Bucharest, Romania}
\email{olivier.schiffmann@universite-paris-saclay.fr}

\author{Fang Yang}
\address[Fang Yang]{Max Planck Institute for Mathematics, Bonn, Germany}
\email{f.yang@mpim-bonn.mpg.de}

\begin{abstract}
We begin the study of Khovanov-Lauda-Rouquier type algebras associated to moduli stacks of coherent sheaves on smooth projective curves. We consider the case of $\mathbb{P}^1$ and define, for any pair $(r,d)$ of a rank and a degree, the KLR and Schur algebras $A_{r,d}, \mathcal{R}_{r,d}$ as suitable convolution algebras in the Borel-Moore homology of an analog of the Steinberg stack built from the stacks $Coh_{r,d}(\mathbb{P}^1)$. We use the tilting equivalence and Bridgeland stability conditions to construct an interpolation between the KLR or Schur algebras of the categories of coherent sheaves on $\P^1$ and the KLR or Schur algebras of the categories of representations of the Kronecker quiver. We also introduce a stratification of the Steinberg stacks into cohomologically pure pieces and use this to construct a PBW basis of the corresponding algebras.
\end{abstract}

\maketitle
\tableofcontents

\section*{Introduction}
\subsection{Motivation}
Introduced by Khovanov-Lauda \cite{Khovanov2009,Khovanov2011} and Rouquier \cite{Rouquier2008} around 15 years ago, the now-called Khovanov-Lauda-Rouquier (KLR) algebras have since played a central role in representation theory by providing a uniform construction of categorifications of highest weight integrable representations of Kac-Moody algebras (see \cite{Kang2018a}), with applications to link homology (\cite{Webster}) and cluster algebras (see \cite{Kashiwara}). KLR algebras are associated to a quiver $Q=(I,E)$ and a dimension vector $d \in \mathbb{N}^I$, and admit a diagrammatic presentation. By a theorem of Rouquier, and Varagnolo-Vasserot (\cite{Varagnolo2011}), they may be realized as convolution algebras in the Borel-Moore homology of quiver analogs of Steinberg varieties. More precisely, let $\widetilde{Rep}_{\underline{d}}$ denote the stack parametrizing filtrations $M_1 \subset \cdots \subset M_s$ of $d$-dimensional representations of $Q$ with factors of dimension $\underline{d}$ and let $\pi_{\underline{d}}: \widetilde{Rep}_{\underline{d}} \to Rep_d$
be the forgetful morphism. Then
$$A_{Q,\mathbf{d}}\simeq \bigoplus_{\underline{d},\underline{d'}}\text{H}_*\left(\widetilde{Rep}_{\underline{d}} \underset{Rep_{d}}{\times} \widetilde{Rep}_{\underline{d'}},\Q\right) \simeq \text{H}^*\left( R\text{Hom}(\pi_{\underline{d}!}\Q, \pi_{\underline{d'}!}\Q)\right)$$
where $\underline{d},\underline{d'}$ run through the set of \textit{discrete} compositions of $d$, i.e. tuples $(a_1, a_2, \ldots, a_s)\in (\N^I)^s$ such that $\sum_i a_i=d$ and each $a_i$ is of dimension $1$.
Note that the simple perverse sheaves occurring in $\pi_{\underline{d}!}\Q$ for discrete compositions $\underline{d}$ of $d$ form Lusztig's canonical basis of the quantum enveloping algebra $\mathbf{U}^-_v(\mathfrak{g})$ associated to $Q$. Allowing $\underline{d},\underline{d'}$ to vary among \textit{all} compositions of $d$, we obtain a larger algebra $\mathcal{R}_{Q,d}$ called the \textit{Schur}, or \textit{quiver-Schur} algebra associated to $Q$ and $d$. For acyclic quivers, $A_{Q,d}$ and $\mathcal{R}_{Q,d}$ are Morita equivalent, but $\mathcal{R}_{Q,d}$ is harder to study in general (for instance, no presentation of $\mathcal{R}_{Q,d}$ is known beyond finite type quivers).

\medskip
This paper is the first of a series devoted to defining and studying the analog of quiver-Schur and KLR algebras when the category of representations of a quiver is replaced by the category of coherent sheaves on a smooth projective curve. In the case of zero-dimensional sheaves, this is the topic of \cite{Maksimau2022}, where an explicit relation to some symmetric algebras is given. In the present paper, we consider the KLR and Schur algebras $A_{\P^1,\alpha}, \mathcal{R}_{\P^1,\alpha}$ associated to $\mathbb{P}^1$. There are several reasons for wanting to extend the theory of Schur and KLR algebras to this setting. First of all, thanks to the well-known analogy between quivers and curves and in particular \cite{Baumann2001, Kapranov1997}, it is natural to expect this analog of KLR algebra to be related to the categorification of the quantum \textit{loop} algebra $\mathbf{U}^+_v(\widehat{sl}_2)$ (see \cite{Shan2022}) for some results in that direction). Secondly, the stack $Bun_r$ classifying vector bundles of rank $r$ on $\P^1$ is isomorphic to the \textit{thick} affine Grassmanian of $GL_r$, and the construction of the KLR algebra of $\P^1$ is a direct analog of the construction of Soergel bimodules in that context. In fact, any simple perverse sheaf on $Bun_r$ appears as a direct summand of some $\pi_{\underline{d}!}\Q$ hence $A_{\P^1,\alpha}$ in principle determines the structure of the category of all\footnote{Of course, this stops being true when $\P^1$ is replaced with a curve $X$ of higher genus; in that case, $A_{X,\alpha}$ only captures a small --yet interesting!-- subcategory of the category of perverse sheaves on $Bun_r$.} perverse sheaves on $Bun_{r}$. Finally, although this is still very speculative at this point, one can hope that a diagrammatic presentation of $A_{\P^1,\alpha}$ would lead to new invariants of knots and links.
In the next paper in this series, we will consider the case of an arbitrary curve $X$ but of KLR algebras associated to \textit{rank one bundles} and observe a relation to Carlsson and Mellit's $A_{q,t}$-algebras \cite{Carlsson2018}.

\medskip
\subsection{Main results}
Let us now describe in more details our main constructions and results.

\medskip

Denote by $Coh_{\alpha}$ the stack of coherent sheaves on $\P^1$ of class $\alpha=(r,d)\in (\Z^2)^+$. It is a smooth stack, locally of finite type. For any sequence $\ua$ whose total class is $\alpha$, let $\widetilde{Coh}_{\ua}$ be the stack parametrizing filtrations of coherent sheaves of class $\alpha$, whose i-th factor  is of class $\alpha_i$.  The Steinberg stack is defined as
\[Z_{\ua,\ua'}:=\widetilde{Coh}_{\ua}\underset{Coh_{\alpha}}{\times}\widetilde{Coh}_{\ua'}\ .\]
It parametrizes triples $(F,F_{\bdot},F'_{\bdot})$ where $F_{\bdot}$ and $F'_{\bdot}$ are filtrations, of type $\ua$ and $\ua'$ respectively, of a coherent sheaf $F$. 
Following \cite{Chriss2010}, we equip with an associative convolution product the direct sums
\[A_{\alpha}:=\bigoplus_{\ua,\ua'\in \mathrm{Seq}(\alpha)}A_{\ua,\ua'}:=\bigoplus_{\ua,\ua'\in \mathrm{Seq}(\alpha)} \text{H}_*(Z_{\ua,\ua'},\Q),\qquad\sR_{\alpha}:=\bigoplus_{\ua,\ua'\vDash \alpha} \text{H}_*(Z_{\ua,\ua'},\Q)\]
 and call $A_{\alpha}$ and $\sR_{\alpha}$ the Schur algebra and the KLR algebra of $\P^1$ respectively. Here $\Seq(\alpha)$ is the set of compositions of $\alpha$ with components in $(\Z^2)^+$ and the notation $\ua \vDash \alpha$ means that $\ua=(\alpha_1,\ldots,\alpha_s) \in \Seq(\alpha)$ and $\text{rank}(\alpha_i) \leq 1$ for all $i$.
 In stark contrast with the case of quivers, the stack $Z_{\ua,\ua'}$ is of infinite type as soon as $\alpha$ is of positive rank, and contains in general infinitely many irreducible components of the same dimension; as a result, $A_\alpha$ and $\mathcal{R}_\alpha$ are naturally pro-vector spaces (even for a fixed cohomological degree).

\medskip

There is a well-known equivalence between the derived categories of coherent sheaves on $\P^1$ and of representations of the Kronecker quiver $Q$. Our first main result concerns the relationship between the corresponding Schur algebras. The basic idea, already present in \cite{Diaconescu2025}, is to approximate the heart $\Coh\P^1$ by translates of the heart $\text{Rep}Q$ using the affine braid group action on $D^b(\text{Rep}Q)$. For this, we consider the natural Bridgeland stability condition on $D^b(\Coh\P^1)$, which allows us to associate to any interval length one interval $I \subset \R$ an abelian heart $\Coh^I$ such that
$$\Coh^{]-\frac12,\frac12]}=\Coh\P^1, \qquad \Coh^{]0,1]}=\text{Rep}\Q,$$
One can then consider the associated moduli stacks of objects $Coh^I_\alpha$ as well as the stacks of filtrations $\widetilde{Coh}^I_{\underline{\alpha}}$.
For length one intervals $I, I'$ which overlap along $J=I \cap I'$ there are canonical open immersions of stacks
$$\begin{tikzcd} Coh^I_\alpha & 
  Coh^J_\alpha\arrow[r,hook] \arrow[l,hook'] & Coh^{I'}_\alpha
\end{tikzcd}.$$
However, this is not the case of $\widetilde{Coh}_{\ua}^I$ and  $\widetilde{Coh}^{I'}_{\ua}$,  since the structure of the quot scheme of a given coherent sheaf $F\in Coh^J$ typically depends on the choice of abelian heart with respect to which the quot scheme is defined.
Nevertheless, we prove that for any $\alpha, \ua,\ua'$ and any quasi-compact open substack $\mathcal{U} \subset Coh_\alpha$ there is a canonical identification of stacks $$Z_{\ua,\ua'}^I \underset{Coh_\alpha}{\times} \mathcal{U} \simeq Z_{\ua,\ua'} \underset{Coh_\alpha}{\times} \mathcal{U}$$ for $I=]\sigma-1,\sigma]$ and $0 < \sigma-\frac12 \ll 1$ (see Theorem~\ref{qcthm}). As a result, we are able to realize the Schur algebra $A_{\alpha}$ as a suitable limit of quotients of Schur algebras $A_{\alpha}^{I}$ associated to intervals $I=]-\frac12+\epsilon, \frac12 + \epsilon]$ as $\epsilon\to 0^+$. We obtain the following comparison result (see Section~\ref{sec:limit}). Set $$A^{\nu,\sigma}_\alpha=\bigoplus_{\ua,\ua'\in \mathrm{Seq}(\alpha)} \text{H}_*\left(Z_{\ua,\ua'}^{]\nu-1,\nu]} \underset{Coh_\alpha}{\times} Coh^{]\sigma-1,\frac12]},\Q\right).$$

\begin{thmx}[Corollary \ref{cor:LimitConstruction}]\label{thmintro1} For any $\alpha \in (\Z^2)^+$, there is a canonical isomorphism of associative algebras
\[A_{\alpha}=\lim_{\sigma \to \frac12^+}\underset{\nu \to \frac12^+}{\mathrm{colim}}\; A^{\nu,\sigma}_{\alpha} \ .\]
\end{thmx}

In particular, choosing a sequence of abelian hearts $\Coh^{]\nu_n-1,\nu_n]}$ equivalent to $\text{Rep}Q$, we obtain an explicit realization of $A_\alpha$ in terms of quotients of the KLR-Schur algebra of the Kronecker quiver. We hope to exploit this in future work to get a diagrammatic presentation of $A_\alpha$. Our construction also naturally produces a natural 'polynomial' representation of $A_\alpha$ on $P_\alpha:=\bigoplus_{\ua} \left(\bigotimes_i \text{H}_*(Coh_{\alpha_i},\Q)\right)$. This representation is known to be faithful in the context of quivers; the proof does not transpose to our setup but we conjecture that the same holds here:

\begin{conj}[Conjecture \ref{conj:PolyRep}]
  For any $\alpha\in (\Z^2)^+$, the polynomial representation $P_{\alpha}$ is faithful.
\end{conj}

\medskip

Our second main result provides PBW-type bases for $A_\alpha$ and $\mathcal{R}_\alpha$ along with a family of generators. For this, we consider a locally finite stratification 
$$Z_{\ua,\ua'}=\bigsqcup_{w} Z_{\ua,\ua'}(w), \qquad w \in W(\ua,\ua')$$
into smooth and cohomologically pure locally closed strata. Here $W(\ua,\ua')$ is a (typically infinite) set of contingency matrices with entries in $(\Z^2)^+$ whose row, resp. column sums are given by the components of $\ua$, resp. $\ua'$. Furthermore, each $Z(w)$ is an iterated vector bundle stack $Z_{\ua,\ua'}(w) \to \prod_{ij} Coh_{w_{ij}}$ hence by Heinloth's extension of the Atiyah-Bott theorem to $Coh_\alpha$, the cohomology of $Z(w)$ is freely generated by the natural tautological classes it carries. 

\medskip

One main difficulty, coming from the non-quasi compactness of $Coh_\alpha$, is the nature of the poset of strata (induced by the closure relation): it typically simultaneously contains intervals isomorphic to $(\N,\geq)$ (as, say an affine Grassmannian), to $(-\N,\geq)$ (as, say a thick affine Grassmannian) and $(\Z,\geq)$ (as, say, a bi-infinite chain of $\P^1$s). This makes the stratification difficult to use to study the cohomology of $Z_{\ua,\ua'}$. We get around this by constructing an exhaustion by quasi-compact opens whose intersection with each $Z_{\ua,\ua'}(w)$ are cohomologically pure, see Proposition~\ref{prop:purity local}.

\medskip

To define a basis from there, we may now follow the classical approach (see e.g. \cite{Przezdziecki2019}, \cite{Varagnolo2011}): we introduce a family of 'elementary' elements of $A_\alpha$ called \textit{merge} and \textit{split} defined as the fundamental classes of closed strata $Z_{\ua,\ua'}(w)$ corresponding to incidence varieties (i.e. to pairs of filtrations $(F_\cdot, G_\cdot)$ in which one filtration refines the other). To each pair $(w,c)$ consisting of a strata $w \in W(\ua,\ua')$ and a cohomology class $c \in \text{H}^*(Z_{\ua,\ua'}(w),\Q)$ we associate a specific ordered product $\mathbf{b}_{c,w}\in \text{H}_*(Z_{\ua,\ua'},\Q)$ of split, merge and operators of multiplication by tautological classes, which is best depicted diagrammatically, see Section~\ref{sec:PBW bases}, in particular Figure~\ref{fig:split cross and merge}. The element $\mathbf{b}_{c,w}$ is referred to as the \textit{PBW-basis element}, for the following reason which is our second main result. For each $w=(w_{ij})_{ij} \in W(\ua,\ua')$, let us fix  a basis $\mathbb{B}_w$ of $\text{H}^*(Z_{\ua,\ua'}(w),\Q)=\bigotimes_{ij} \text{H}^*(Coh_{w_{ij}},\Q)$.

\smallskip

\begin{thmx}[{Theorem \ref{thm:PBWbases}}]\label{thmintro2}
The set $\{(\mathbf{b}_{w,c}|~\ua,\ua'\in \mathrm{Seq}(\alpha),~w\in W(\ua,\ua'),~c\in \B_w\}$ is a topological basis of $A_{\alpha}$. Restricting to discrete $(\ua,\ua')$ yields a topological basis of $\mathcal{R}_\alpha$.
\end{thmx}

In the above, by topological basis we mean a linearly independent set whose span is a dense subspace. 

\begin{cor}
The Schur algebra $A_{\alpha}$ is generated, as an associative algebra,  by elementary splits, merges and multiplication by tautological classes. 
\end{cor}

\medskip
\subsection{Organization of paper} section~\ref{sec:Preliminary} serves as a reminder on the stacks of coherent sheaves on $\P^1$ and of representations of the Kronecker quiver $Q$, as well as on the relations between their cohomology rings induced by the tilting equivalence. In Section~2  we introduce the Steinberg stacks $Z_{\ua,\ua'}$ of coherent sheaves on $\P^1$, and study several partitions of it (see Section \ref{sec:Steinberg stacks}); in particular, we prove some key purity results for these strata and their intersections. Besides, we compare $Z_{\ua,\ua'}$ with the Steinberg variety of representations of $Q$.
In Section~\ref{sec:Schur algebras}, we define the Schur and KLR algebras $A_{\alpha}$ and $\mathcal{R}_\alpha$ and prove Theorem~\ref{thmintro1} identifying $A_\alpha$ with a limit of quotients of the Schur algebra of the Kronecker quiver. In the last section~\ref{sec:PBW basis} we construct the family of PBW basis elements of $A_{\alpha}$ and prove Theorem~\ref{thmintro2}.
  
\medskip
 
\section{Preliminaries}\label{sec:Preliminary}

In this section, we recall some basic properties of the cohomology and Borel-Moore homology of stacks and introduce the main characters: the stacks of coherent sheaves on $\P^1$ and of representations of the Kronecker quiver $Q$. We also review the derived equivalence between $\Coh\P^1$ and $\mathrm{mod}\CQ$, and the consequences for the corresponding stacks.

\subsection{Stacks and Borel-Moore homology} \label{sec:BMhomo}
Throughout, we work over $\C$. By a \textit{stack} we will mean an Artin stack (for the \textit{fpqc} topology), which is locally of finite type. The stacks which we will consider may in fact be covered by global quotient stacks (of finite type). For such a stack $\mathcal{X}$, we let $D(\mathcal{X})$ stand for the triangulated category of complexes of $\Q$-vector spaces on $\mathcal{X}$ with constructible cohomology, and we define $D^b(\mathcal{X})$ as the subcategory of complexes whose restriction to each quasi-compact open is bounded. The categories $D^b(\mathcal{X})$ for $\mathcal{X}$ as above are equipped with the usual six-functor formalism (see \cite[Appendix A]{khan2019}). In particular, we denote by $\omega_X=p^!(\mathbb{Q}) \in D^b(\mathcal{X})$ the dualizing complex, and by $\text{H}_i(X,\mathbb{Q}):=\text{H}^{-i}(p_*\omega_X)$ the \textit{$i$th Borel-Moore homology group} of $\mathcal{X}$. Here $p: \mathcal{X} \to pt$ is the projection to the point. Alternatively, one may compute these groups as limits
\begin{equation}\label{E:BMhomologyasalim}
    \text{H}_i(\mathcal{X},\Q) = \lim_{\mathcal{U}} \text{H}_i(\mathcal{U},\Q)
\end{equation}
where the limit is taken over all quasi-compact open substacks $\mathcal{U}$, with respect to open immersion, see \cite[Equation(3.3.5)]{Kapranov2023}. We say that a  stack $\mathcal{X}$ is \textit{well-approximated by quasi-compacts} if for any $d \in \Z$ there exists a quasi-compact open substack $\mathcal{U} \subset \mathcal{X}$ such that $\text{codim}_{\mathcal{X}}(\mathcal{X} \backslash \mathcal{U}) >d$. For well-approximated stacks, the limit \eqref{E:BMhomologyasalim} stabilizes, i.e. for each $i$ there exists $\mathcal{U}$ for which the restriction map $\text{H}_i(\mathcal{X},\Q) \to \text{H}_i(\mathcal{U}',\Q)$ is an isomorphism for any $\mathcal{U}' \supset \mathcal{U}$.

If $\mathcal{X}$ is irreducible then there is a canonical element $[\mathcal{X}] \in \text{H}_{2\dim(\mathcal{X})}(\mathcal{X},\Q)$. Likewise, associated to each irreducible closed substack $i_{\mathcal{Z}}:\mathcal{Z} \to \mathcal{X}$ of a stack $\mathcal{X}$ is a fundamental class $[\mathcal{Z}] \in \text{H}_{2\dim(\mathcal{Z})}(\mathcal{X},\Q)$.

We will also consider the cohomology groups $\text{H}^i(\mathcal{X},\Q)$, which are defined as 
$\text{H}^i(\mathcal{X},\Q)=H^i(\Q_{\mathcal{X}})$, and which again satisfies
\begin{equation}\label{E:cohomologyasalim}
    \text{H}^i(\mathcal{X},\Q) = \lim_{\mathcal{U}} \text{H}^i(\mathcal{U},\Q)
\end{equation}
where the limit is taken over all quasi-compact open substacks $\mathcal{U}$. As usual there is a cup product 
$$\cup : \text{H}^i(\mathcal{X},\Q) \otimes\text{H}^j(\mathcal{X},\Q) \to \text{H}^{i+j}(\mathcal{X},\Q)$$ and a cap product
$$\cap : \text{H}^i(\mathcal{X},\Q) \otimes\text{H}_j(\mathcal{X},\Q) \to \text{H}_{j-i}(\mathcal{X},\Q).$$

Finally, it will be convenient to consider the following notion: we will say that a stack $\mathcal{X}$ as above which is locally of finite type is \textit{cohomologically pure} if there exists an open exhaustion $(\mathcal{U}_i)_{i \in \N}$ of $\mathcal{X}$ by  quasi-compact open substacks for which $\text{H}_j(\mathcal{U}_i,\Q)$ is pure of weight $j$ for any $i$ and $j$ in the sense of \cite{deligne1971,deligne1974,Achar2021}. 
We will often use without reference the following result, whose proof is left to the reader.

\begin{lemma}\label{L:purity} Let $\mathcal{X}$ be a stack which is locally of finite type, let $\mathcal{Z} \subset \mathcal{X}$ be a closed substack and $\mathcal{U}=\mathcal{X} \backslash \mathcal{Z}$ its open complement. Assume that both $\mathcal{Z}$ and $\mathcal{U}$ are cohomologically pure. Then $\mathcal{X}$ is cohomologically pure and for any $i$ there is a short exact sequence
$$\xymatrix{ 0\ar[r] & \text{H}_i(\mathcal{Z},\Q) \ar[r] &  \text{H}_i(\mathcal{X},\Q) \ar[r] &  \text{H}_i(\mathcal{U},\Q) \ar[r] & 0}.$$
\end{lemma}

\smallskip

\subsection{The stack of coherent sheaves on \texorpdfstring{$\P^1$}{Lg} and its cohomology ring} 

\medskip

\subsubsection{The Projective line}\label{sec:def correspondence}
  Let $\P^1$ be the projective line over $\C$. Its cohomology ring is identified with $\Q[\omega]/\omega^2$, where $\omega \in \text{H}^2(\P^1,\Q)$ is the class of a point. Let $\Coh\P^1$ be the category of coherent sheaves on $\P^1$. The Grothendieck group $K_0(\mathbb{P}^1)$ is identified with $\Z^2$ via the map 
  $$[\mathcal{F}] \mapsto \dim(\mathcal{F}):=(rank(\mathcal{F}), deg(\mathcal{F})).$$
  The slope of a coherent sheaf $\mathcal{F}$ is as usual given by
  $$\mu(\mathcal{F})=\mu([\mathcal{F}]):=deg(\mathcal{F})/rank(\mathcal{F}) \in \Q \cup \{\infty\}.$$
  We also denote by 
  $$K_0^+(\mathbb{P}^1)=\{(r,d) \in \Z^2\;|\; r \geq 0,\; d>0 \;\text{if}\;r=0\}=:(\Z^2)^+$$
  the monoid of classes of coherent sheaves. 
  
  For any $\alpha \in K_0(\mathbb{P}^1)^+$, we let $Coh_\alpha(\mathbb{P}^1)$ be the stack of coherent sheaves on $\mathbb{P}^1$ of class $\alpha$. It is a smooth irreducible stack, locally of finite type and of dimension 
  $$\dim(Coh_{(r,d)}(\mathbb{P}^1))=-r^2.$$
  Let us first consider the case of torsion sheaves. We have $Coh_{(0,1)}(\P^1)=\P^1 \times B\G_m$; for any $d >0$ the stack $Coh_{(0,d)}(\P^1)$ is of finite type and is equipped with a support map $supp: Coh_{(0,d)}(\P^1) \to S^d\P^1 \simeq \P^d$. It fits in a Springer-like diagram
  \begin{equation}
   \xymatrix{ Coh_{(0,1)}(\P^1)^d & \widetilde{Coh}_{(0,1)^d} \ar[r]^-p \ar[l]_-q & Coh_{(0,d)}(\P^1)}
   \end{equation}
  where $\widetilde{Coh}_{(0,1)^d}$ is the stack parametrizing full flags $\mathcal{F}_1 \subset \mathcal{F}_2 \subset \cdots \subset \mathcal{F}_d$ of torsion sheaves with $length(\mathcal{F}_i)=i$ for $i=1, \ldots, d$, and 
  $p(\mathcal{F}_\bullet)=\mathcal{F}_d$, $q(\mathcal{F}_\bullet)=(\mathcal{F}_d/\mathcal{F}_{d-1},  \cdots, \mathcal{F}_1)$. The morphism $q$ is a stack vector bundle while $p$ is proper, small, and is generically a Galois $\mathfrak{S}_d$-cover. It follows that 
  $$\text{H}^*(\widetilde{Coh}_{(0,1)^d},\Q) \simeq \text{H}^*(Coh_{(0,1)}(\P^1)^d,\Q) =(\text{H}^*(\P^1,\Q)[t])^{\otimes d}$$
  which yields an identification
  \begin{equation}\label{E:cohtorP1}
\text{H}^*(Coh_{(0,d)}(\P^1),\Q) \simeq \text{H}^*(\widetilde{Coh}_{(0,1)^d},\Q)^{\mathfrak{S}_d} \simeq Sym^d(\text{H}^*(\P^1,\Q)[t]).
  \end{equation}
In particular, this shows that $Coh_{(0,d)}(\P^1)$ is cohomologically pure.

  Let us now turn to the case of higher rank $\alpha$. The open substack parametrizing semistable sheaves will be denoted $Coh_\alpha^{ss}(\Q^1)$. There is a locally finite stratification
  $$Coh_\alpha(\P^1)=\bigsqcup_{\underline{\alpha}} HN_{\underline{\alpha}}$$
  into Harder-Narasimhan strata indexed by the collection of HN-types of weight $\alpha$ :
  $$\underline{\alpha}=(\alpha_1, \ldots, \alpha_s), \qquad \alpha_1 + \cdots + \alpha_s=\alpha, \qquad \mu(\alpha_1) < \mu(\alpha_2) <\cdots.$$
  Each Harder-Narasimhan strata $HN_{\ua}$ is a stack vector bundle over $\prod_{i}Coh^{ss}_{\alpha_i}(\P^1)$. When $\alpha_i=(r,dr)$, $Coh^{ss}_{\alpha_i}(\P^1)$ contains a single object $\sO(d)^{\oplus r}$ up to isomorphism, it is hence the classifying stack of $GL(r)$; in particular it is cohomologically pure and quasi-compact. When $\alpha_i=(0,d)$, $Coh^{ss}_{\alpha_i}(\P^1)=Coh_{\alpha_i}\P^1$. Therefore $HN_{\ua}$ is cohomologically pure and quasi-compact. For any $n\geq 0$ we define
  $$Coh^{I_n}_\alpha:=\{\mathcal{F} \in Coh_\alpha\P^1\;|\; \mathcal{F}(n-1)\;\text{is generated by global sections}\}.$$

  \begin{lemma}\label{lem:CohoPure} For any $\alpha\in (\Z^2)^+$,  the following holds
  \begin{enumerate}
      \item For any $n$, $Coh^{I_n}_\alpha$ is open in $Coh_\alpha(\P^1)$ and is a finite union of Harder-Narasimhan strata. In particular, $Coh_{\alpha}^{I_n}$ is cohomologically pure.
      \item For any $n$, the stack $Coh_\alpha(\P^1)$ is well-approximated by the collection of open substacks $Coh^{I_n}_\alpha$ for $n \geq 0$.
  \end{enumerate}
  In particular, $Coh_\alpha(\P^1)$ is cohomologically pure.
  \end{lemma}
\begin{proof} This is well-known, but we give the proof for the comfort of the reader. The operation of tensoring by $\sO(n)$ induces an isomorphism of stacks $Coh_\alpha(\P^1) \simeq Coh_{\alpha(n)}(\P^1)$ which maps $Coh_\alpha^{I_n}$ onto $\Coh_{\alpha(n)}^{I_1}$. Thus we may assume that $n=1$ in the proof of (1). Let us observe that a coherent sheaf $\sF$ on $\P^1$ is generated by global sections if and only if the vector bundle quotient of $\sF$ is isomorphic to a direct sum $\bigoplus_i \sO(n_i)$ with $n_i \geq 0$ for all $i$. This is equivalent to the condition $\Hom(\sF, \sO(-1))=\{0\}$, which is an open condition by the semicontinuity theorem \cite[Theorem 12.8]{Hartshorne1977}. It is easy to see from the above description that $Coh^{I_0}_\alpha$ is a union of a finite number of Harder-Narasimhan strata; indeed, a strata $HN_{\ua}$ is contained in $Coh^{I_0}_\alpha$ if and only if the minimal slope satisfies $\mu(\alpha_1) \geq 0$ and this condition cuts out a bounded region for the Harder-Narasimhan polytope.  Thus, $Coh_{\alpha}^{I_n}$ is cohomologically pure.

We turn to (2). We have to show that for any $d\in \Z$, there exists $n\geq 1$ such that 
$$\text{codim}_{Coh_\alpha(\P^1)}(Coh_\alpha\P^1\setminus Coh^{I_n}_\alpha ) >d.$$
Because the partition of $Coh_\alpha(\P^1)$ into Harder-Narasimhan strata is locally finite, it suffices to show that there exists $n$ such that for any $\ua$ for which $\mu(\alpha_1)\leq -n$ we have $$\text{codim}_{Coh_\alpha\P^1}( HN_{\ua} ) >d.$$
This can be done by an explicit computation, using the following formulas
$$\dim\;HN_{\ua}=-\sum_{i\leq j} \langle \alpha_i, \alpha_j\rangle,$$
$$\text{codim}_{Coh_\alpha\P^1}(HN_{\ua} )= -\sum_{i> j} \langle \alpha_i, \alpha_j\rangle$$
where $\langle (r,d), (s,l)\rangle=rs + rl- sd$ is the Euler form on $K_0(\Coh\P^1)$. We leave the details to the reader.
\end{proof}

\medskip

\subsubsection{Cohomology ring and tautological classes} The cohomology ring of $\Coh_\alpha(\P^1)$ is well understood thanks to the work of Atiyah-Bott and Heinloth. More precisely, let
  $\mathcal{E}_{\alpha} \in \Coh( Coh_\alpha(\P^1) \times \P^1)$ be the tautological coherent sheaf. For $\mu \in \text{H}^*(\P^1,\Q)$ and $i \geq 1$ we set 
  $$ch_i(\mu)=\int_{\P^1} ch_i(\mathcal{E}_\alpha) \cup \mu \in \text{H}^{*}(Coh_\alpha(\P^1),\Q).$$
  For notational convenience, we have suppressed the dependence on $\alpha$ in the notation. Note that 
  $$deg(ch_i(\mu))=2i+deg(\mu)-2.$$
In particular, the classes $ch_0(\omega)$ and $ch_1(1)$ which evaluate to $rank(\alpha)$ and $deg(\alpha)$ respectively are of degree 0 (i.e., scalars). It will be convenient to introduce a universal ring of tautological classes. We let $\mathbb{H}$ be the polynomial algebra generated by elements $\underline{ch}_i(\mu)$ for $i\geq 1, \mu \in\text{H}^*(\P^1,\Q)$ or $i=0, \mu=\omega$ modulo the relations $\underline{ch}_i(\mu + \lambda)=\underline{ch}_i(\mu) + \underline{ch}_i(\lambda)$. We define a grading on $\mathbb{H}$ by setting $deg(\underline{ch}_i(\mu))=2i-2+deg(\mu)$.
For any $\alpha \in K_0^+(\P^1)$ there is an evaluation map
$$\text{ev}_\alpha: \mathbb{H} \to \text{H}^*(Coh_\alpha(\P^1),\Q), \qquad \underline{ch}_i(\mu) \mapsto ch_i(\mu)$$
which is a morphism of graded algebras. Let $\mathbb{H}^{red} \subset \H$ be the subalgebra generated by $\underline{ch}_i(\mu)$ with $i \geq 2$ or $i=1, \mu=\omega$ (i.e. by the generators of positive degree). We have
 $$\H\simeq \H^{red} \otimes \Q[\underline{ch}_0(\omega), \underline{ch}_1(1)].$$
The ring $\H$ is equipped with a cocommutative coproduct defined by
$$\Delta(\underline{ch}_i(\mu))=\underline{ch}_i(\mu) \otimes 1 + 1 \otimes \underline{ch}_i(\mu)$$
for any $i$ and $\mu$. It may be upgraded to a Hopf algebra if we define the antipode as  $$S(\underline{ch}_i(\mu))=-\underline{ch}_i(\mu).$$

We now provide a `symmetric functions' realization of $\H^{red}$. Write a vector space 
  $$\V=\text{H}^*(Coh_{(0,1)},\Q)=\Q[\omega,u]/\omega^2$$ where $u=c_1(\mathcal{O}(1))\in \text{H}^*(B\mathbb{G}_m)$ (hence $deg(u)=2$).
There is an augmentation map 
 $$\epsilon:\V \to \text{H}^0(Coh_{(0,1)},\Q)=\Q.$$
This allows us to form the rings
$$\Lambda_n:=(\V^{\otimes n})^{\mathfrak{S}_n}, \qquad \Lambda_\infty:= \varprojlim_n \Lambda_n$$
where the transition map $\Lambda_{n+1} \to \Lambda_n$ is induced by $\epsilon\otimes Id^{n-1}: \V^{\otimes n+1} \to \V^{\otimes n}$. The ring $\Lambda_\infty$ has a bialgebra structure with coproduct induced by the inclusions $\Lambda_{n+m} \to \Lambda_n \otimes \Lambda_m$. We may view elements of $\Lambda_n$ as symmetric sums of monomials $\gamma_1u^{d_1} \otimes \gamma_2u^{d_2} \otimes \cdots \otimes \gamma_nu^{d_n}$, where $\gamma_i \in \text{H}^*(\P^1,\Q)$ and $d_i \in \N$. Likewise we may view elements of $\Lambda_\infty$ as (infinite) symmetric sums of monomials  $\gamma_1u^{d_1} \otimes \gamma_2u^{d_2} \otimes \cdots$, with all but finitely many components of the tensor product equal to 1. We will often write such a monomial as $(\gamma_1u^{d_1})_1(\gamma_2u^{d_2})_2\cdots$.

\medskip

The following is standard :

\begin{proposition}\label{PropSymmRealization} The assignment $\underline{ch}_i(\mu) \mapsto \sum_l (\mu u^i)_l$ extends to a bialgebra isomorphism $\H^{red} \to \Lambda_\infty$.
\end{proposition}

We may now state Atiyah-Bott and Heinloth's descriptions of $\text{H}^*(Coh_\alpha,\Q)$ in \cite{Heinloth2012} :
  
  \begin{theorem}[Atiyah-Bott, Heinloth]\label{ThmHeinloth} The following hold~: 
  \begin{itemize}
      \item for any $n >0$ the map $\text{ev}_{0,n}: \H^{red} \to \text{H}^*(Coh_{0,n}(\P^1),\Q)$ factors to an isomorphism $\Lambda_n \stackrel{\sim}{\to} \text{H}^*(Coh_{0,n}(\P^1),\Q)$,
      \item for any $r >0$ and any $d \in \Z$, the map $\text{ev}_{r,d}: \H^{red} \to \text{H}^*(Coh_{(r,d)}(\P^1),\Q)$ is an isomorphism,
      \item denoting by $\bigoplus: Coh_\alpha(\P^1) \times Coh_\beta(\P^1) \to Coh_{\alpha+\beta}(\P^1), (\mathcal{F},\mathcal{G}) \mapsto \mathcal{F}\oplus \mathcal{G}$ the direct sum map, we have a commuting square
      $$\xymatrix{\H^{red} \ar[d]_-{\text{ev}_{\alpha+\beta}} \ar[r]^-{\Delta} & \H^{red} \otimes \H^{red} \ar[d]^-{\text{ev}_\alpha \otimes \text{ev}_{\beta}} \\ \text{H}^*(Coh_{\alpha+\beta}(\P^1),\Q)\ar[r]^-{\bigoplus^*}& \text{H}^*(Coh_{\alpha}(\P^1),\Q) \times \text{H}^*(Coh_{\beta}(\P^1),\Q)
      }.$$
  \end{itemize} 
  \end{theorem}
  \begin{proof}
     Statement i) and ii) may be found in \cite{Heinloth2012}. The last statement follows from the equality $\bigoplus^*(\mathcal{E}_{\alpha+\beta})=(\mathcal{E}_{\alpha})_{13} \oplus (\mathcal{E}_{\beta})_{23} \in \Coh(Coh_\alpha(\P^1) \times Coh_\beta(\P^1) \times \P^1)$.
  \end{proof}

  \medskip
  
  As $Coh_{(r,d)}(\mathbb{P}^1)$ is smooth, the map 
     $$\text{H}^*(Coh_\alpha(\P^1),\Q) \to \text{H}_{*}(Coh_\alpha(\P^1),\Q),\qquad  c \mapsto c \cap [Coh_\alpha(\P^1)]$$
  is an isomorphism.

\subsection{The stack of Kronecker quiver representations}
  Let $Q=(Q_0,Q_1 )$ be the Kronecker quiver, that is,
    $$\begin{tikzcd}
        Q: 1  &2\arrow[l,shift left]\arrow[l,shift right]
    \end{tikzcd}$$
 We denote by $\mathrm{mod}\CQ$ the category of finite dimensional $\C Q$-modules. Its Grothendieck group $K_0(Q)$ is identified with $\Z^2$ through the dimension vector map $[M] \mapsto \underline{\dim}M:= (\dim(M_1), \dim(M_2))$. The monoid of classes of representations is $K_0^+(Q)=\N^2$. Recall that indecomposable objects are split into three families: the \textit{ pre-injective, regular, and pre-projective} modules, denoted by $\sP$, $\sR$, and $\sI$, respectively. In the following, denote by $S_i$, $P_i$, $I_i$ the simple, indecomposable projective, indecomposable injective modules at vertex $i$.
  
  For $d \in\N^2$ we denote by $Rep_dQ$ the stack of $d$ dimensional representations of $Q$. It is a smooth stack of finite type, isomorphic to the quotient stack $E_d/G_d$ where
  $$E_{(d_1,d_2)}=\text{Hom}(\C^{d_2},\C^{d_1})\otimes \C^2, \qquad G_d=GL_{d_1} \times GL_{d_2}.$$
  Let $V_1, V_2$ be the tautological vector bundles of rank $d_1,d_2$ respectively on $Rep_dQ$. Denote by $ch_i(V_j)\in \text{H}^{2i}(Rep_dQ)$ the ith componnet of  Chern character of $V_j$. We have
  $$\text{H}^*(Rep_dQ,\Q)=\Q[ch_i(V_j)\,|\,i=1,2, j=1,\ldots, d_i].$$
  Note that we have $ch_0(V_i)=d_i$.

\medskip

\subsection{The tilting equivalence}

\subsubsection{Reminder on the tilting functor} The following theorem summarizes some well-known properties of the tilting equivalence between $\mathrm{Coh}\P^1$ and $\mathrm{mod}\CQ$~:

  \begin{theorem}\label{th:tilting}
  The following holds~:
  \begin{enumerate}
      \item The derived functor $\R\Hom_{D^b(\Coh\P^1)}(\sO\oplus \sO(1),-)$ gives rise to a tilting equivalence of triangulated categories $\Phi:D^b(\Coh\P^1)\simeq D^b(\mathrm{mod}\CQ)$, sending $\sO$ and $\sO(1)$ to the indecomposable projective modules $P_1$ and $P_2$ respectively.
      
      \item The induced isomorphism of Grothendieck groups is equal to 
      $$\begin{aligned}
          \Phi: K_0(D^b(\Coh\P^1)) \simeq \Z^2 &\to \Z^2 \simeq K_0(D^b(\mathrm{mod}\CQ)),\\
          (r,d) & \mapsto (r+d, d)\ .
      \end{aligned}$$
     
      \item Let $\mathcal{F}\in \mathrm{Coh}\P^1$ and $M\in \mathrm{mod}\CQ$. Then
      $\Phi(\mathcal{F}) \in \mathrm{mod}\CQ$ if and only if the canonical vector bundle quotient of $\mathcal{F}$ belongs to the additive envelope of $\{\mathcal{O}(n)\;;\; n \geq 0\}$. Likewise, $\Phi^{-1}(M) \in \mathrm{Coh}\P^1$ if and only if $M \in \mathcal{P} \oplus \mathcal{R}$.
  \end{enumerate}
  \end{theorem}

  For reasons that will become clear in the next section, let us denote by $Coh^{[0,\frac{1}{2}]}_\alpha$ the stack parameterizing coherent sheaves $\mathcal{F} \in \mathrm{Coh}P^1$ of class $\alpha$ satisfying the condition in Theorem~\ref{th:tilting}, (iii) and let $j:Coh^{[0,\frac{1}{2}]}_\alpha \to Coh_\alpha(\P^1)$ be the embedding. There is likewise an embedding
    $$\iota:Coh^{[0,\frac{1}{2}]}_\alpha \to Rep_{\Phi(\alpha)}Q.$$
  We will later prove, in a more general context, that $j,\iota$ are open immersions.

\subsubsection{The tilting functor and cohomology rings} We next turn to tautological classes.
  Recall that for each $M=(\oplus_i M_i,\oplus_{h\in Q} M_h)\in \mathrm{mod}\CQ$, there is a Ringel canonical projective resolution $$P_M^{\bullet}=(M_2\otimes_{\C}P_1)^{\oplus 2}\stackrel{f_1\oplus f_2}\longrightarrow  (M_1\otimes_{\C}P_1)\oplus (M_2\otimes_{\C}P_2),$$
  where $f_1=\oplus_{h\in Q_1}M_h\otimes id$ and $f_2=\oplus_{h\in Q_1}id\otimes h$. We may thus consider the complex of bundles over $Rep_d(Q)\times \P^1$ defined by
    $$ V^{\bullet}_d: (V_2\boxtimes \sO)^{\oplus 2} \stackrel{f_1\oplus f_2}\lrw (V_1\boxtimes \sO)\oplus(V_2\boxtimes \sO(1)).$$
  By Theorem~\ref{th:tilting}, we have for any $\alpha \in (\Z^2)^+$ a canonical quasi-isomorphism of tautological complexes
  $$(\iota^*\times Id_{\P^1})(V^{\bullet}_{\Phi(\alpha)}) \simeq (j^*\times Id_{\P^1})(\sE_\alpha)$$
  over $Coh^{[0,\frac{1}{2}]}_\alpha\times \P^1$. Recall that  $\sE_{\alpha}$  is the universal tautological bundle over $Coh_{\alpha}(\P^1)\times \P^1$.

  \begin{proposition}\label{cohoP1&Q} Fix $\alpha=(r,d)\in (\Z^2)^+$.
  \begin{enumerate}
      \item Assume that $r>0$. Then the ring homomorphism
  \begin{equation*}
  \begin{aligned}
       \psi_{\alpha}: \text{H}^*(Coh_{\alpha}\P^1)&\lrw H^*(Rep_{\Phi(\alpha)}Q)\\
             ch_i(1)&\mapsto ch_{i-1}(V_2), \quad i\geq 1;\\
             ch_i(\omega)&\mapsto ch_i(V_1)-ch_i(V_2), \quad i\geq 0.
  \end{aligned}     
  \end{equation*}
  makes the following diagram commute~:
    $$\begin{tikzcd}
        \text{H}^{*}(Coh_{\alpha}(\P^1))\arrow[r,"\psi_\alpha"]\arrow[d,"j^*"]  & \text{H}^*(Rep_{\Phi(\alpha)}Q)\arrow[ld,"\iota^*"]\\
        \text{H}^{*}(Coh^{[0,\frac{1}{2}]}_{\alpha}) 
    \end{tikzcd}\ .$$
    Moreover, the maps $j^*$ and $\iota^*$ are surjective.
    \item Assume that $r=0$. Then $j$ (hence also $j^*$) is an isomorphism and $\iota^*$ is surjective.  
    \end{enumerate}
  \begin{proof}
We begin with (1). First observe that as $r>0$, the ring $\text{H}^*(Coh_\alpha\P^1)$ is freely generated by $\{ch_i(1), ch_j(\omega)\;|\; i \geq 2, j \geq 1\}$ by Theorem \ref{ThmHeinloth}, so that $\psi_\alpha$ is well-defined. Note that $ch_0(\sO(l))=1$, $ch_1(\sO)=0$ and $ch_1(\sO(1))=\omega$. We have
  \begin{equation*}
  \begin{aligned}
    (j^* \times Id)ch_i(\sE_{\alpha})&= -2(\iota^*ch_i(V_2)\boxtimes 1)\\
    &+\iota^*(ch_i(V_1)+ch_i(V_2))\boxtimes 1+\iota^*ch_{i-1}(V_2)\boxtimes \omega.
  \end{aligned}
  \end{equation*}
  Then $j^*ch_i(\omega)=\iota^*(ch_{i}(V_1)-ch_i(V_2))$ and $j^*ch_i(1)=\iota^*(ch_{i-1}(V_2))$, as wanted. The surjectivity of $j^*$ (and hence also of $\iota^*$) follows from the cohomological purity of Harder-Narasimhan strata, see Lemma~\ref{L:purity}.
  The proof of (2) follows likewise from the purity of Harder-Narsimhan strata in the quiver moduli stack $Rep_\alpha Q$ as in Lemma \ref{lem:CohoPure}.
  \end{proof}                    
  \end{proposition}

\begin{remark}
To get a uniform statement in Proposition~\ref{cohoP1&Q}, we may define a morphism $\underline{\psi}_\alpha: \mathbb{H} \to \text{H}^*(Rep_{\Phi(\alpha)}\C Q)$ by the same formulas as above, and get a commutative square
$$\begin{tikzcd}
\mathbb{H} \arrow[d, "ev_\alpha"] \arrow[r,"\underline{\psi}_\alpha"]  & \text{H}^*(Rep_{\Phi(\alpha)}Q)\arrow[d,"\iota^*"]\\
        \text{H}^{*}(Coh_{\alpha}(\P^1))\arrow[r,"j^*"] &
        \text{H}^{*}(Coh^{[0,\frac{1}{2}]}_{\alpha}) 
    \end{tikzcd}\ .
$$
\end{remark}
  
\subsection{Stability conditions and BGP reflections}

\subsubsection{Notations} We will use the notion of Bridgeland stability condition, for which we refer the reader to \cite{Bridgeland2007}. Consider the charge function given by
     $$Z: K_0(\Coh\P^1)\lrw \C,\ \ \ [\sF]\mapsto \mathrm{rank}(\sF)+\sqrt{-1}\deg(\sF).$$
  Denote by $\phi$ be the phase function of $Z$. Then any nonzero coherent sheaf $\sF$ has phase in $]-\frac{1}{2},\frac{1}{2}]$, and any torsion sheaf has phase $\frac{1}{2}$. Let $\sP$ be the slicing of $D^b(\Coh \P^1)$ such that $\sP(\phi)$ consists of semistable coherent sheaves of phase $\phi$. Then $\sigma=(Z,\sP)$ is a Bridgeland stability condition on $D^b(\Coh\P^1)$ by \cite[Example 5.6]{Bridgeland2007}.  In the following, we will denote by $\Coh^{\nu} \subset D^b(\Coh\P^1)$ the abelian heart of phase $]\nu-1,\nu]$. Note that
  $$\Coh\P^1=\Coh^{\frac{1}{2}}, \qquad \Phi^{-1}(\mathrm{mod}\C Q)=\Coh^{\nu}, \quad \forall\; \nu=1-\epsilon, \;0 <\epsilon \ll 1.$$
  For any interval $I$ of length at most 1, denote by $\Coh^{I}$ the subcategory of $D^b(\Coh\P^1)$ of objects whose Harder-Narasimhan factors have phases in $I$. Thus $\Coh^I$ is a subcategory of the abelian category $\Coh^J$ for any length one interval $J\supseteq I$. We use the same notation $Coh^J_\alpha, Coh^\nu_\alpha$ for the stacks of objects in $\Coh^J, \Coh^\nu$ of class $\alpha$.
  
  \begin{remark}\label{rem:discrete_phases}
    Because the phases of indecomposable objects in $D^b(\Coh\P^1)$ belong to the countable set $\Sigma=\Z + \{\frac{1}{\pi} arctan(n)\;;\; n \in \Z\}$, the category $\Coh^J$ only depends on which of the intervals in $\R$ cut out by $\Sigma$ the endpoints of $J$ lie. Note also that the set of accumulation points of $\Sigma$ is $\Z + \frac{1}{2}$.
  \end{remark}

\begin{lemma}\label{LemOpenness}
   Let $J$ be a length one interval and let $I \subseteq J$. Then the substack $Coh^I_{\alpha}$ is open in $Coh^{J}_{\alpha}$. Moreover, $Coh^I_\alpha$ is quasi-compact if $I$ is of length $<1$.
\begin{proof}
    We may write $J=]\nu-1,\nu]$ (the other case, $J=[\nu-1,\nu[$ is dual and dealt with in the same manner). To prove the first statement of the lemma, it suffices to check that $Coh^{I}_\alpha$ is open in $Coh^{J}_{\alpha}$ for $I$ of the form $]\nu-1,a]$, $]b,\nu]$, $]\nu-1,a[ $ or $[b,\nu]$. We will deal with the first two cases, the other ones are similar. By remark~\ref{rem:discrete_phases}, $\Coh^\nu$ is either equivalent to $\text{mod}\C Q$ (if $\nu \not\in \frac{1}{2} + \Z)$ or to $\Coh\P^1$ (if $\nu \in \frac{1}{2}+\Z$). Note that any indecomposable object of $\Coh^\nu$ is semistable. An object $M \in \Coh^\nu$ of class $\alpha$ belongs to $\Coh^{]\varphi, \nu]}$ if and only if $\Hom(M,N)=\{0\}$ for all indecomposable $N$ of phase $\phi(N)\leq\varphi$ such that $[N] <_\nu [\alpha]$. Here $[N] <_\nu \alpha$ means that $\alpha-[N] \in K_0(\Coh^\nu)^+$. Likewise an object $M \in \Coh^\nu$ of class $\alpha$ belongs to $\Coh^{]\nu-1,\varphi]}$ if and only if
     $$\Hom(N,M)=\Hom_{D^b(\Coh\P^1)}(M,N\otimes \sO(-2)[1])^\vee=\{0\}$$
    for all indecomposable $N$ of phase $\phi(N)>\varphi$ such that $[N] <_\nu \alpha$. By the semicontinuity theorem (see \cite[Lemma 4.2]{CB2002} in the quiver case), each of these conditions is open. When $\Coh^\nu \simeq \text{mod}\C Q$, there are finitely many such conditions, hence the desired result. When $\Coh^\nu \simeq \Coh\P^1$ we may replace these infinite sets of conditions by finitely many ones (in fact, a single one) by noticing that either $\varphi=\frac{1}{2}$, or $\varphi\neq\frac{1}{2}$ and it suffices to take $N=\sO(l)$ with $l=\mathrm{min}\{l|\ \phi(\sO(l))> \varphi\}$ or $l=\mathrm{max}\{l|\ \phi(\sO(l))\leq \varphi\}$. This proves that $Coh^I_\alpha$ is open in $Coh^\nu_\alpha$. To prove that $Coh^I_\alpha$ is quasi-compact if $I$ is of length strictly less than $1$, observe that in this case, we may include $I$ in an interval $]\nu', \nu'+1]$ for which $\Coh^{\nu'}$ is equivalent to $\text{mod}\C Q$, and thus for which $\Coh^{\nu'}_\alpha$ is itself quasi-compact. 

\end{proof}
\end{lemma}

\subsubsection{The braid group action}
We now recall the action of the braid group $B_2 \simeq \Z$ on $D^b(\Coh\P^1) \simeq D^b(\mathrm{mod}\CQ)$ given by BGP reflection functors. Let $\tau$ be the Auslander-Reiten translation functor on $D^b(\mathrm{mod}\CQ)$. The object $T=\tau^{-1}P_1\oplus P_2$ is tilting and there is a natural isomorphism $\text{End}(T)\simeq \C Q'$ where $Q'$ is the Kronecker quiver with opposite orientation, i.e. with vertex 2 being the sink. We denote by $P'_i$ the indecomposable projective modules of $Q'$. Denote by $\Phi'=\R\Hom(\sO\oplus\sO(1),-) : D^b(\Coh\P^1) \simeq D^b(\mathrm{mod}\CQ')$ the tilting equivalence\footnote{it is simply the composition of $\Phi$ with the obvious equivalence $D^b(\mathrm{mod}\CQ) \simeq D^b(\mathrm{mod}\CQ')$.} as in Theorem~\ref{th:tilting} (1).

\begin{theorem}\label{thm:tilting_BGP} The following holds~:
   \begin{enumerate}
       \item The functors $S^+=\R\Hom(T,-)$ and $S^-=-\otimes^\L T$ form a pair of adjoint functors
     $$S^-: \begin{tikzcd} D^b(\mathrm{mod}\CQ')\arrow[r,shift left]  &D^b(\mathrm{mod}\CQ)\arrow[l,shift left] \end{tikzcd} :S^+$$
  which are quasi-inverse to each other. 
  \item We have $S^+(P_2)=P'_2$, $S^+(\tau^{-1}(P_1))=P'_1$ and $S^+(P_1[1])$ is the indecomposable injective  $\C Q'$-module at vertex 1.
  \item Define $\tau^{\pm\frac{1}{2}} \in \text{Aut}(D^b(\Coh\P^1))$ such that the following diagram commutes~:
  $$\begin{tikzcd}[row sep=large,column sep=large]
  D^b(\Coh\P^1)\arrow[d,"\Phi'"]\arrow[r,bend left=10,"\tau^{-\frac{1}{2}}"]
  &D^b(\Coh\P^1)\arrow[l,bend left=10,"\tau^{\frac{1}{2}}"]\arrow[d,"\Phi"] \\
  D^b(\mathrm{mod}\CQ')\arrow[r,bend left=10,"S^-"] &D^b(\mathrm{mod}\CQ)\arrow[l,bend left=10,"S^+"]
  \end{tikzcd}\ .$$
  Then $\tau^{\pm \frac{1}{2}}$ is equivalent to the functor of tensoring by the line bundle $\mathcal{O}(\mp 1)$.
  \item Set $$\nu_n=\phi(\sO(-n))+1=1-\frac{1}{\pi}arctan(n), \qquad  (\forall\, n\geq 1)$$
  so that $\Coh^{\nu_1}=\Phi^{-1}(\mathrm{mod}\CQ)$. The functors $\tau^{\pm\frac{1}{2}}$ restrict to abelian equivalences 
  $$ \tau^{\frac{1}{2}}: \Coh^{\nu_n}\longrightarrow \Coh^{\nu_{n+1}},\qquad 
       \tau^{-\frac{1}{2}}: \Coh^{\nu_n}\longrightarrow \Coh^{\nu_{n-1}}.$$
    In particular, $\Coh^{\nu_n}=\tau^{\frac{n-1}{2}}(\Coh^{\nu_1})=\tau^{\frac{n-1}{2}}(\Phi^{-1}(\mathrm{mod}\CQ))$ for any $n \geq 1$.
   \end{enumerate} 
\end{theorem}
\begin{proof}
Statements (1) and (2) are classical, see e.g. \cite[Section VII.5]{ASS2006}. Statements (3) and (4) are also classical; they may be checked for indecomposables by a direct computation.
\end{proof}

  \begin{remark}\label{rem:tilting}
    Due to Remark~\ref{rem:discrete_phases}, every abelian heart of phase $]\nu-1,\nu]$ with $\nu\in ]\frac{1}{2},1]$ is equal to $\Coh^{\nu_n}$ for some $n$, and in particular equivalent, as an abelian category, to $\mathrm{mod}\CQ$.
\end{remark}
  
  Note that the simple objects in $\Coh^{\nu_n}$ are $\sO(-n+1)$ and $\sO(-n)[1]$. The transformation matrix $B_n$ from the standard basis $\{[\sO],[\sO(1)]-[\sO]\}$ of $K_0(\Coh\P^1)$ to the basis $\{[\sO(-n+1)],-[\sO(-n)]\}$ is equal to~:
       $$B_n=\begin{bmatrix}
           n &1\\ n-1 &1
       \end{bmatrix}.$$
  For any $\alpha=(r,d)\in \Z^2$, set $$d(n,\alpha):=B_n\cdot\alpha=B_n\cdot(\begin{smallmatrix} r\\d\end{smallmatrix})=(nr+d,(n-1)r+d).$$
  
  Denote by $Coh^{\nu}_{\alpha}$ (resp. $Coh^I_{\alpha}$) the stack parameterizing objects in $\Coh^{\nu}_{\alpha}$ (resp. $\Coh^{I}_{\alpha}$) of class $\alpha\in K_0(\Coh^{\nu})^+$.  Theorem~\ref{thm:tilting_BGP} implies the following.
  \begin{lemma}
  For any $n\geq 1$, the functor $\Phi \circ \tau^{\frac{1-n}{2}}$ gives rise to an equivalence of stacks $Coh^{\nu_n}_{\alpha}\cong Rep_{d(n,\alpha)}Q$.
  \end{lemma}
  From now on, we will identify these two stacks using the above equivalence. Next, we relate $\text{H}^*(Coh_\alpha\P^1)=\text{H}^*(Coh^{\frac{1}{2}}_{\alpha})$ to $\text{H}^*(Coh^{\nu_n}_{\alpha}) \simeq \text{H}^*(Rep_{d(n,\alpha)}Q)$ via the open embeddings 
  $$j_n:Coh^{]\nu_n-1,\frac{1}{2}]}_{\alpha} \to Coh^{\frac{1}{2}}_\alpha,\qquad\iota_n:Coh^{]\nu_n-1,\frac{1}{2}]}_{\alpha} \to Rep_{d(n,\alpha)}Q.$$ 
   
  \begin{corollary}\label{cohoP1&Coh^nu} Fix $\alpha=(r,d) \in (\Z^2)^+$ and $n \geq 1$.
   \begin{enumerate}
   \item Assume that $r>0$ and let $V_{i,n}$ be the tautological vector bundles on $Coh^{\nu_n}_\alpha=Rep_{d(n,\alpha)}Q$ of rank $d(n,\alpha)_i$. Then there exists a surjective ring homomorphism
   $$\begin{aligned}
       \psi_{\alpha,n}: H^*(Coh^{\frac{1}{2}}_{\alpha})&\lrw  H^*(Coh^{\nu_n}_{\alpha})\\
       ch_i(1)&\mapsto (1-n)ch_{i-1}(V_1)+nch_{i-1}(V_{2,n}),\ i\geq 1\\
       ch_i(\omega)&\mapsto ch_i(V_{1,n})-ch_i(V_{2,n}),\ i\geq 0.
   \end{aligned}$$
   making the following diagram commute:
    $$\begin{tikzcd}
    \text{H}^{*}(Coh^{\frac{1}{2}}_{\alpha}) \arrow[r,"\psi_{\alpha,n}"]\arrow[d,"j_n^*"']  & \text{H}^*(Coh^{\nu_n}_{\alpha})\arrow[ld,"\iota_n^*"]\\
        \text{H}^{*}(Coh^{]\nu_n-1,\frac{1}{2}]}_\alpha).
    \end{tikzcd}$$ 
\item Assume that $r=0$. Then $j_n$ (hence also $j_n^*$) is an isomorphism and $\iota_n^*$ is surjective.
    \end{enumerate}
  \begin{proof}
  The proof is similar to Proposition \ref{cohoP1&Q}.
  \end{proof}
  \end{corollary}

\subsubsection{Limit construction of $\text{H}^*(Coh_\alpha\P^1)$}
 When $r>0$, we may recover the cohomology ring $\text{H}^*(Coh_\alpha^{\frac{1}{2}})$ by forming the projective system
 $\varprojlim \;\text{H}^*(Coh^{\nu_n}_{\alpha})=\varprojlim \;\text{H}^*(Rep_{d(n,\alpha)}Q)$ associated to the transition morphisms
 \begin{equation}
 \begin{split}
    \phi_{n,n-1}: \text{H}^*(Rep_{d(n,\alpha)}Q) &\to \text{H}^*(Rep_{d(n-1,\alpha)}Q)\\
   ch_i(V_{2,n}) &\mapsto ch_i(V_{1,n-1}),\\
    ch_i(V_{1,n}) &\mapsto 2 ch_i(V_{1,n-1}) - 
    ch_i(V_{2,n-1})\ .
 \end{split}
 \end{equation}
 Note that for any fixed cohomological degree, the maps $\phi_{n,n-1}$ become isomorphisms for $n \gg 0$.
  
  \begin{proposition}\label{PropCohoP1&Quiver} 
  Let $\alpha=(r,d)\in (\Z^2)^+$ and assume that $r>0$. There is a unique ring isomorphism
  $$\Psi_\alpha~:\text{H}^*(Coh_{\alpha}\P^1)\stackrel{\sim}{\longrightarrow} \varprojlim \text{H}^*(Rep_{d(n,\alpha)}Q)$$
  such that for any $n\geq 1$, the following diagram of canonical maps commutes
  $$\begin{tikzcd}
    \text{H}^{*}(Coh^{\frac{1}{2}}_{\alpha}) \arrow[r,"\Psi_\alpha"]\arrow[d,"j_n^*"']  & \varprojlim \text{H}^*(Rep_{d(n,\alpha)}Q)\arrow[ld,"\iota_n^*"]\\
        \text{H}^{*}(Coh^{]\nu_n-1,\frac{1}{2}]}_\alpha).
    \end{tikzcd}\ .$$
    \end{proposition}
  \begin{proof}
    The commutation of the diagram of restriction maps is a direct consequence of Corollary \ref{cohoP1&Coh^nu}. The uniqueness of $\Psi$ now follows from the fact the injectivity of $\prod_n j_n^*$, which itself follows from Lemma~\ref{lem:CohoPure} (2).
  \end{proof}

\medskip
\section{Steinberg stacks}\label{sec:Steinberg stacks}

  The aim of this section is to introduce and compare the analogs, in the context of the categories $\Coh\P^1$ and $\Coh^\nu$, of the Steinberg stacks.

\medskip

\subsection{Definition of the Steinberg stacks}

  For any $\alpha\in \Z^2$, denote by $\Seq(\alpha)$ the set of all sequences $\underline{\alpha}=(\alpha_s, \cdots,\alpha_{2},\alpha_1)$ of elements in $\Z^2$ with $|\ua|:=\sum_i \alpha_i=\alpha$. For any $\nu$, we let $\Seq^\nu(\alpha) \subset \Seq(\alpha)$ be the subset determined by the condition that $\alpha_i \in K_0(\Coh^\nu)^+$ for any $i$.
  For each 
  $$\underline{\alpha}=(\alpha_{s},\cdots,\alpha_1)\in \Seq^\nu(\alpha),$$
  let $\widetilde{Coh}^{\nu}_{\underline{\alpha}}$ be the stack parameterizing filtrations 
      $$F_{\boldsymbol{\cdot}} ~:0=F_0\subseteq F_1\subseteq\cdots\subseteq F_s,$$
  of objects in $\Coh^{\nu}$ such that the quotient $F_i/F_{i-1}$ is of dimension vector $\alpha_{i}$. There is a natural morphism
  $$q_{\ua}:\widetilde{Coh}^\nu_{\ua} \to Coh^\nu_{\alpha_s} \times \dots \times Coh^\nu_{\alpha_1}, \qquad F_{\bdot}\mapsto (F_s/F_{s-1}, \cdots, F_1/F_0)\ .$$
  which is a stack vector bundle (because $\Coh^I$ is of homological dimension one, see e.g. \cite{Schiffmann2006}). In particular, $\widetilde{Coh}^\nu_{\ua}$ is smooth. There is a (representable) proper map 
      $$\pi^\nu_{\ua}: \widetilde{Coh}^\nu_{\ua}\lrw Coh^{\nu}_{\alpha},\ F_{\bdot}\mapsto F_s. $$
    Note that both $\Seq^\nu(\alpha)$ and $\widetilde{Coh}^\nu_{\ua}$ heavily depend on $\nu$.
  For any  two sequences $\ua'$, $\ua''$ in $\Seq(\alpha)$, the \textit{Steinberg stack} associated to $\nu$ is defined as the fiber product
   $$Z^\nu_{\ua',\ua''}:=\widetilde{Coh}^\nu_{\ua'}\underset{{Coh^\nu_{\alpha}}}\times \widetilde{Coh}^\nu_{\ua''}$$
when $\ua,\ua'\in \Seq^\nu(\alpha)$ and is empty otherwise. Because $\pi^\nu_{\ua}, \pi^\nu_{\ua'}$ are proper, the natural morphism $Z^\nu_{\ua,\ua'} \to \widetilde{Coh}^{\nu}_{\ua} \times \widetilde{Coh}^\nu_{\ua'}$ is proper as well.
\medskip
\subsection{Quasi-compact approximation theorem} \

  Our approach to computing Schur algebras associated to $\Coh\P^1$ rests upon the following simple but crucial result providing criteria for a flag, subobjects and quotients to be lying in both $\Coh^{\frac{1}{2}}$ and $\Coh^{\nu}$. It will help us relate $Z^{\frac{1}{2}}_{\ua,\ua'}$ with $Z^\nu_{\ua,\ua'}$ for $\nu$ sufficiently close to $\frac{1}{2}$. For any $\nu$ and any interval $I\subset ]\nu-1,\nu]$ we consider the fiber products $$\widetilde{Coh}^{\nu,I}_{\ua}=\widetilde{Coh}^{\nu}_{\ua}\underset{{Coh^\nu_{\alpha}}}{\times} Coh^I_\alpha\ ,\qquad Z^{\nu,I}_{\ua',\ua''}=Z^\nu_{\ua',\ua''}\underset{Coh^\nu_{\alpha}}{\times} Coh^I_\alpha\ ,$$  which are respective open substacks of $\widetilde{Coh}^\nu_{\ua}$ and $Z^{\nu}_{\ua',\ua''}$ by Lemma \ref{LemOpenness}. For any $\gamma \in ]\frac12,1]$ we set 
  $$I_\gamma=]\gamma-1,\frac12].$$
  
\begin{theorem}[{Quasi-compact approximation theorem}]\label{qcthm}
   Fix $\alpha_1, \alpha_2\in (\Z^2)^+$, $\alpha:=\alpha_1+\alpha_2$ and $\sigma\in ]\frac{1}{2},1[$.   
   \begin{enumerate}
   \item[(i)] There exists $\nu_0\in ]\frac{1}{2},\sigma]$ such that for any $\frac{1}{2}<\nu<\nu_0$, we have a factorization
   $$\begin{tikzcd}
     Coh^{\frac{1}{2}}_{\alpha_2}\times Coh^{\frac{1}{2}}_{\alpha_1} & \widetilde{Coh}^{\frac{1}{2}}_{\alpha_2,\alpha_1}\arrow[l,"q"']\arrow[r,"p"] & Coh^{\frac{1}{2}}_{\alpha}\\
    Coh^{I_{\nu}}_{\alpha_2}\times Coh^{I_{\nu}}_{\alpha_1}\arrow[u,hook] &\widetilde{Coh}^{\frac{1}{2},I_\sigma}_{\alpha_2,\alpha_1}\arrow[u,hook]\arrow[l,dashed]\arrow[r]  & Coh^{I_\sigma}_{\alpha}\arrow[u,hook],
    \end{tikzcd}$$
    where the right square is cartesian and the vertical maps are the open immersions. 

   \item[(ii)] There exists $\nu'_0\in ]\frac{1}{2},\sigma]$ such that for any $\frac{1}{2}<\nu<\nu'_0$, we have the following commutative diagram:
   $$\begin{tikzcd}
     Coh^{\nu}_{\alpha_2}\times Coh^{\nu}_{\alpha_1} & \widetilde{Coh}^{\nu}_{\alpha_2,\alpha_1}\arrow[l,"q"']\arrow[r,"p"] & Coh^{\nu}_{\alpha}\\
    Coh^{I_{\nu}}_{\alpha_2}\times Coh^{I_{\nu}}_{\alpha_1}\arrow[u,hook] &\widetilde{Coh}^{\frac{1}{2},I_\sigma}_{\alpha_2,\alpha_1}\arrow[u,hook]\arrow[l]\arrow[r]  & Coh^{I_\sigma}_{\alpha}\arrow[u,hook],
   \end{tikzcd}$$
   where the right square is cartesian and the vertical maps are the open immersions.
   \end{enumerate}
    As a consequence, there is a canonical stack isomorphism $$\widetilde{Coh}^{\frac{1}{2},I_\sigma}_{\alpha_2,\alpha_1} \simeq \widetilde{Coh}^{\nu,I_\sigma}_{\alpha_2,\alpha_1}.$$
\end{theorem}
 We will make use of the following lemma, whose proof is given in the Appendix
 \ref{App:qclem}.
  \begin{lemma}\label{qclem}
  Let $Y$ be a classical stack over $\C$ and $Y=\bigsqcup_{i\in \N}Z_i$ be a partition into (nonempty) irreducible locally closed substacks. Let $X\subset Y$ be a quasi-compact constructible subset. Then there exists a finite set $J\subset \N$ such that $X\subset \bigsqcup_{i\in J}Z_i$.
  \end{lemma}

\begin{proof}[Proof of Theorem~\ref{qcthm}]
  (i) Since $p$ is proper and $Coh^{I_\sigma}_{\alpha}$ is quasi-compact, $S^{\frac{1}{2},\sigma}_{\alpha_2,\alpha_1}$ is also quasi-compact, see \cite[\href{https://stacks.math.columbia.edu/tag/04XU}{Tag 04XU}]{stacks-project}. Thus $q(\widetilde{Coh}^{\frac{1}{2},I_\sigma}_{\alpha_2,\alpha_1})$ is quasi-compact and there exists a quasi-compact substack $U\subset Coh^{\frac{1}{2}}_{\alpha_2}\times Coh^{\frac{1}{2}}_{\alpha_1}$ such that we have a factorization:
    $$\begin{tikzcd}
    \widetilde{Coh}_{\alpha_2,\alpha_1}^{\frac{1}{2}}\arrow[r,"q"] & Coh^{\frac{1}{2}}_{\alpha_2}\times Coh^{\frac{1}{2}}_{\alpha_1}\\
    \widetilde{Coh}^{\frac{1}{2},I_\sigma}_{\alpha_2,\alpha_1}\arrow[u,hook]\arrow[r,dashed] &U\arrow[u,hook].
    \end{tikzcd}$$
   Now, $(Coh^{I_{\nu}}_{\alpha_2}\times Coh^{I_\nu}_{\alpha_1})_{\nu}$ form an open exhaustion of $Coh^{\frac{1}{2}}_{\alpha_2}\times Coh^{\frac{1}{2}}_{\alpha_1}$. By Lemma \ref{qclem}, there exists $\nu_0 \in ]\frac12,1)$ such that for any $\nu\leq \nu_0$, we have $U\subset Coh^{I_{\nu}}_{\alpha_2}\times Coh^{I_{\nu}}_{\alpha_1}$.

   (ii) For any $\sF\in \Coh^{I_\sigma}_{\alpha}$ and $\sF'\subseteq \sF$ in $\Coh_{\alpha_1}^{\nu}$, we have $\sF'\in \Coh^{I_\nu}_{\alpha_1}$ and $\sF/\sF'\in \Coh^{I_\nu}_{\alpha_2}$. Indeed, for a short exact sequence in $\Coh^{\nu}$
     $$0\lrw \sF'\lrw \sF\lrw \sF/\sF' \lrw 0,$$
   we have $\phi^+(\sF')\leq \phi^+(\sF)\leq \frac{1}{2}$ and $\sigma-1\leq \phi^-(\sF)\leq \phi^-(\sF/\sF')$, where $\phi^+(\sF)$ (resp. $\phi^-(\sF)$) is the maximum (resp. minimum) among phases of HN-factors of $\sF$. Thus we have a factorization 
     $$\begin{tikzcd}
     \widetilde{Coh}_{\alpha_2,\alpha_1}^{\nu,I_\sigma} \arrow[r,"q"]\arrow[rd,dashed] &Coh^{\nu}_{\alpha_2}\times Coh^{\nu}_{\alpha_1}\\
       &Coh^{ ]\sigma-1,\nu]}_{\alpha_2} \times Coh^{I_\nu}_{\alpha_1}\arrow[u,hook].
      \end{tikzcd}$$
    Proving (ii) amounts to showing that for $\nu$ close enough to $\frac{1}{2}$, there is in fact a factorization
      $$\begin{tikzcd}
     \widetilde{Coh}_{\alpha_2,\alpha_1}^{\nu,I_\sigma} \arrow[r,"q"]\arrow[rd,dashed] &Coh^{\nu}_{\alpha_2}\times Coh^{\nu}_{\alpha_1}\\
       &Coh^{I_\sigma}_{\alpha_2} \times Coh^{I_\nu}_{\alpha_1}\arrow[u,hook].
      \end{tikzcd}$$
    Note that $Coh^{]\sigma-1,\nu]}_{\alpha_2}\subseteq Coh^{\sigma}_{\alpha_2}$, whose HN-stratification $Coh^{\sigma}_{\alpha_2}=\bigsqcup_{\underline{\alpha_2}} HN^{\sigma}_{\underline{\alpha_2}}$ is finite. Denote by $\nu'_0$ the minimum of phases of these HN-factors strictly greater than $\frac{1}{2}$, and $\Sigma:=\{\underline{\alpha_2}\in \Seq(\alpha_2)|\ \phi((\underline{\alpha_2})_i\geq \nu'_0 \text{ for some i}\}$, then we have   $$Coh^{\sigma}_{\alpha_2}=Coh^{I_\sigma}_{\alpha_2}\sqcup \bigsqcup_{\underline{\alpha_2}\in\Sigma} HN^{\sigma}_{\underline{\alpha_2}}$$
    and it follows that
    $$Coh_{\alpha_2}^{]\sigma-1,\nu]}=Coh_{\alpha_2}^{\nu} \cap Coh_{\alpha_2}^{\sigma}\subseteq Coh_{\alpha_2}^{I_\nu}$$
    for any $\frac{1}{2}<\nu<\nu'_0$.
  \end{proof}

 Recall that we have set $\nu_n:=\phi(\sO(-n))+1$ and put $I_n:=]\nu_n-1,\frac{1}{2}]$. We have the following corollary of Theorem~\ref{qcthm}.
\begin{corollary}\label{cor:QuasiCompact}
  Fix $\alpha\in (\Z^2)^+$ and $m \in \Z$.
  \begin{enumerate}
    \item [(i)] For any $\ua\in \Seq^{\frac{1}{2}}(\alpha)$ there exists $n_0\in \N$ such that for any $n\geq n_0$, we have a canonical isomorphism
      $${\Psi}^{\nu_n,I_m}_{\underline{\alpha}}: \widetilde{Coh}_{\underline{\alpha}}^{\frac{1}{2},I_{m}}\cong \widetilde{Coh}_{\underline{\alpha}}^{\nu_n,I_{m}}.$$
    \item [(ii)] For any $\ua,\ua' \in \Seq^{\frac12}(\alpha)$ there exists $n_0$ such that for any $n>n_0$ there is a canonical isomorphism
      $$\Psi^{\nu_n,I_m}_{\ua,\ua'}: Z^{\frac12, I_m}_{\ua,\ua'} \simeq Z^{\nu_n, I_m}_{\ua,\ua'}\ .$$
  \end{enumerate}
\begin{proof}
  We prove the first statement by induction on the length of $\ua$. For any $\ua=(\alpha_2,\alpha_1)\in \Seq^{\frac{1}{2}}(\alpha)$, the statement follows from Theorem \ref{qcthm}.  
  For $\ua=(\alpha_{s},\cdots,\alpha_1)$ with $s\geq 3$, consider $\ua=(\alpha_s,|\ua'|)$ with $\ua'=(\alpha_{s-1},\cdots,\alpha_{1})$. Then there exists $n_0\geq m$ such that for any $n\geq n_0$, we have an isomorphism
    $${\Psi}^{\nu_n,I_m}_{\ua'}: \widetilde{Coh}^{\frac{1}{2},I_m}_{(\alpha_s,|\ua'|)} \cong \widetilde{Coh}^{\nu_n,I_m}_{(\alpha_s,|\ua'|)}.$$
  As a result, for any $(F_{s-1}\subseteq F_s)\in \widetilde{Coh}^{\frac{1}{2},I_m}_{(\alpha_s,|\ua'|)} $, we have $F_{s-1}\in \widetilde{Coh}^{\nu_{n_0},I_{n_0}}_{|\ua'|}$. There are canonical isomorphisms
    $$\widetilde{Coh}^{\frac{1}{2},I_m}_{\ua}\cong \widetilde{Coh}^{\frac{1}{2},I_{n_0}}_{\ua'}\times_{Coh^{I_{n_0}}_{|\ua'|}} \widetilde{Coh}^{\frac{1}{2},I_m}_{(\alpha_s,|\ua'|)},$$
    $$F_{\bdot}\mapsto ((0\subseteq F_1\subseteq \cdots\subseteq F_{s-1}),(F_{s-1}\subseteq F_{s})),$$
  and 
    $$\widetilde{Coh}^{\nu_n,I_m}_{\ua}\cong \widetilde{Coh}^{\nu_n,I_{n_0}}_{\ua'}\times_{Coh^{I_{n_0}}_{|\ua'|}} \widetilde{Coh}^{\nu_n,I_m}_{(\alpha_s,|\ua'|)},$$
    $$F_{\bdot}\mapsto ((0\subseteq F_1\subseteq \cdots\subseteq F_{s-1}),(F_{s-1}\subseteq F_{s}))$$
  where $n\geq n_0$. By the induction hypothesis for $\ua'$, there is $n_1\geq n_0$ such that for any $n\geq n_1$, we have an isomorphism
    $${\Psi}^{\nu_n,I_{n_0}}_{\ua'}: \widetilde{Coh}^{\frac{1}{2},I_{n_0}}_{\ua'} \cong \widetilde{Coh}^{\nu_n,I_{n_0}}_{\ua'},$$
  which leads to the desired isomorphism
    $$ {\Psi}^{\nu_n,I_m}_{\ua}: \widetilde{Coh}^{\nu_n,I_m}_{\ua}\cong \widetilde{Coh}^{\frac{1}{2},I_m}_{\ua}.$$

  Statement (ii) follows from (i) and the fact that for any $I$
  \[Z_{\ua,\ua'}^{\nu,I}= \widetilde{Coh}^{\nu,I}_{\ua}\times_{Coh^{\nu,I}_{\alpha}}\widetilde{Coh}^{\nu,I}_{\ua'}.\]
\end{proof}
\end{corollary}

\medskip
\subsection{Partitions of Steinberg stacks}
  In this section, we will describe several types of partitions and open exhaustions of the Steinberg stacks as well as their intersections. This will be used later to study their Borel-Moore homology and the structure of convolution algebra which one can put on the direct sum of these Borel-Moore homology spaces.

  \subsubsection{The cell partition}\label{sec:cellpartition} This is the standard partition, whose parts in the case of the usual, full flag Springer resolution are parametrized by the Weyl group.
  
  Let $\ua,\ua' \in \Seq^\nu(\alpha)$ be of length $l,s$. To $(F_{\bdot},F'_{\bdot})\in Z^{\nu}_{\ua,\ua'}$ we associate the matrix $w(F_{\bdot},F'_{\bdot})$ whose (i,j)-term is the class in $\Z^2$ of the subquotient 
      \begin{equation}\label{eq:def_FF'ij}(F_{\bdot},F'_{\bdot})_{i,j}:=F_i\cap F'_j\big/ (F_{i-1}\cap F'_j+F_i\cap F'_{j-1})
      \end{equation}
  where the intersection and quotient is taken in the abelian category $\Coh^{\nu}$.
  We denote by $W(\ua,\ua')$ the set of $(s\times l)$-matrices $w$ with entries $w_{ij}\in\Z^2$ such that 
    $$\sum_{i=1}^s w_{ij}=\alpha'_j \text{  and  }\sum_{j=1}^l w_{ij}=\alpha_i\ ,$$
  and by $W^\nu(\ua,\ua')\subset W(\ua,\ua')$ the subset of matrices $w$ as above with entries in $K_0(\Coh^\nu)^+$. Note that, again,  $W^\nu(\ua,\ua')$ depends on $\nu$. We will say that a pair $(F_{\bdot},F'_{\bdot})\in Z_{\ua,\ua'}^{\nu}$ is of type $w$ if $(F_{\bdot},F'_{\bdot})\in Z_{\ua,\ua'}^{\nu}(w)$.

  For $w\in W^\nu(\ua,\ua')$, let $Z^{\nu}_{\ua,\ua'}(w)$ be the constructible subset of $Z^{\nu}_{\ua,\ua'}$ consisting of pairs $(F_{\bdot},F'_{\bdot})$ of type $w$. 
  There is a canonical morphism
    $$\pi_{w}:Z^\nu_{\ua,\ua'}(w) \to \prod_{i,j} Coh^\nu_{w_{i,j}}, \qquad (F_{\bdot}, F'_{\bdot}) \mapsto \left((F_{\bdot}, F'_{\bdot})_{i,j}\right).$$

  The following is classical (and holds verbatim for any smooth projective curve). 

\begin{lemma}\label{lem:dimStrata}
  The morphism $\pi_w$ is an iterated stack vector bundle of rank $$r(w):=-\sum_{(i,j)<(i',j')} \langle w_{i'j'}, w_{ij}\rangle\ ,$$ where $(i,j) \leq (i',j')$ if and only if $i \leq i'$ and $j \leq j'$. In particular, $Z^\nu_{\ua,\ua'}(w)$ is smooth irreducible, well approximated by quasi-compact opens, cohomologically pure and of dimension 
   $$\dim(Z_{\ua,\ua'}^\nu(w))=-\sum_{(i,j)\leq(i',j')} \langle w_{i'j'}, w_{i,j}\rangle\ .$$
\end{lemma}

By construction, $Z^{\nu}_{\ua,\ua'}$ has a partition into constructible subsets
  \begin{equation}\label{eq:Partition}
  Z^{\nu}_{\ua,\ua'}=\bigsqcup_{w\in W^\nu(\ua,\ua')} Z^{\nu}_{\ua,\ua'}(w).    
  \end{equation}
  
  For $\nu\not\in \frac{1}{2}+\Z$, we have $\Coh^\nu \simeq \text{mod}\C Q$, hence the set $W^\nu(\ua,\ua')$ is finite because the dimension vectors of quiver subrepresentations of representations of a fixed dimension form a bounded set. It is easy to see that in general, $W^{\nu}(\ua,\ua')$ is infinite when $\nu \in \frac12 + \Z$.

\medskip

\subsubsection{The global open exhaustion.}
  For any \textit{strict} subinterval $I\subset ]\nu-1,\nu]$ (i.e. of length $<1$), the stack $Z^{\nu,I}_{\ua,\ua'}$ is quasi-compact by Lemma \ref{LemOpenness} and \cite[\href{https://stacks.math.columbia.edu/tag/04XU}{Tag 04XU}]{stacks-project}. We deduce in particular that
the Steinberg stack $Z^{\nu}_{\ua,\ua'}$ admits a quasi-compact open exhaustion
    \begin{equation}\label{eq:OpenEx}
    Z^{\nu}_{\ua,\ua'}=\bigcup_{I} Z^{\nu,I}_{\ua,\ua'}.    
    \end{equation}
   
  Let us set
   $$ W(\ua,\ua')^{\nu,I}:=\{w\in W(\ua,\ua')^{\nu}|\ Z^{\nu,I}_{\ua,\ua'}(w)\neq \varnothing\}.$$
  The partition (\ref{eq:Partition}) of $Z^{\nu}_{\ua,\ua'}$ induces a similar partition $Z^{\nu,I}_{\ua,\ua'}$ into constructible subsets
  $$Z^{\nu,I}_{\ua,\ua'}=\bigsqcup_{w\in W(\ua,\ua')^{\nu,I}} Z^{\nu,I}_{\ua,\ua'}(w)\ .$$
  
\begin{lemma}\label{lem:finiteness}
  For any $\nu$ and any strict subinterval $I \subset ]\nu-1,\nu]$, the set $W(\ua,\ua')^{\nu,I}$ is finite.
\end{lemma}
\begin{proof} 
This follows from the fact that $Z^{\nu,I}_{\ua,\ua'}$ is quasi-compact and Lemma~\ref{qclem}. Note that for $\nu\not\in \frac{1}{2}+\Z$, this is obvious as $W(\ua,\ua')^{\nu,I}\subset W(\ua,\ua')^{\nu}$, which is already finite. 
\end{proof}

Recall that for $\sigma \in ]\frac12,1]$ we have set $I_\sigma=]\sigma-1,\frac12]$. The following is a key technical result.

\begin{proposition}\label{prop:purity_prop}
For any $\sigma \in ]\frac12,1]$ and any $\ua,\ua' \in \Seq^{\frac12}(\alpha)$, the stack $Z^{\frac12,I_\sigma}_{\ua,\ua'}$ is cohomologically pure.
\end{proposition}

\begin{proof} Thanks to Remark~\ref{rem:discrete_phases}, it suffices to consider the cases $I_\sigma=I_m$ for $m \geq 0$. If $\sV$ is vector bundle on $\P^1$ and $d\geq 0$ we denote by $i_{\sV,d}: X_{\sV,d} \to Coh_{\alpha}$ the locally closed substack parametrizing coherent sheaves isomorphic to $\sV \oplus \sT$ where $\sT \in Coh_{d\delta}$. Note that $Coh^{I_m}_\alpha$ admits a partition
$$Coh^{I_m}_\alpha=\bigsqcup_{\sV, d} X_{\sV,d}$$
where the union ranges over the \textit{finite} set of pairs $(\sV,d)$ for which $\sV$ is a vector bundle of class $\alpha-d\delta$ generated by $\sO(m)$. We will prove Proposition~\ref{prop:purity_prop} by showing that for any $(\sV,d)$ the stack
$$Z^{\frac12,I_m}_{\ua,\ua'}\underset{Coh_\alpha}{\times} X_{\sV,d}$$
is cohomologically pure. 

We will first reduce ourselves to the case $d=0$. For any $\underline{\sigma}=(\sigma_s, \ldots, \sigma_1) \in \Seq^{\frac12}_\alpha$,  there is a partition
$$\widetilde{Coh}^{\frac12}_{\underline{\sigma}} \underset{Coh_\alpha}{\times} X_{\sV,d}=\bigsqcup_{\underline{c}} K^{\underline{c}}_{\underline{\sigma},\sV}$$
where $\underline{c}$ runs through the set of compositions of $d$ of length $s$ and where
$$K^{\underline{c}}_{\underline{\sigma},\sV}=\{\sF_1 \subset \sF_2 \subset \cdots \subset \sF_s=\sV \oplus \sT\;|\;\deg(\sF_i \cap \sT /\sF_{i-1} \cap \sT )= c_i\;\forall\;i\}.$$
Setting $\underline{\sigma}-\underline{c}=(\sigma_i -c_i\delta)_i \in \Seq^{\frac12}(\alpha-d\delta)$ and $\underline{c}\delta=(c_i\delta)_i\in \Seq^{\frac12}(d\delta)$, the assignment
$$(\sF_1 \subset \cdots \subset \sF_s) \mapsto \left((\sF_1/\sF_1 \cap \sT \subset \cdots \subset \sF_s / \sF_s \cap \sT), (\sF_1 \cap \sT \subset \cdots \subset \sF_s \cap \sT) \right)$$
induce a morphism
$$K^{\underline{c}}_{\underline{\sigma},\sV} \longrightarrow \left(\widetilde{Coh}^{\frac12}_{\underline{\sigma}-\underline{c}} \underset{Coh_{\alpha-d\delta}}{\times} X_{\sV,0}\right) \times \widetilde{Coh}^{\frac12}_{\underline{c}\delta} $$
which is easily seen to be a vector bundle stack. For similar reasons, the induced map
$$K^{\underline{c}}_{\underline{\sigma},\sV} \underset{X_{\sV,d}}{\times} K^{\underline{c}'}_{\underline{\sigma}',\sV} \longrightarrow \left( Z^{\frac12}_{\underline{\sigma}-\underline{c}, \underline{\sigma}'-\underline{c}'} \underset{Coh_{\alpha-d\delta}}{\times} X_{\sV,0} \right) \times Z^{\frac12}_{\underline{c}\delta, \underline{c}'\delta}$$
is a vector bundle stack.

Observe that $X_{\sV,0}=B\text{Aut}(\sV)$ and that $\widetilde{Coh}^{\frac12}_{\underline{\gamma}} \times_{Coh_{\alpha-d\delta}} X_{\sV,0}$ is isomorphic to the stack quotient of the iterated quot scheme
$$\text{Quot}(\underline{\gamma},\sV) := \{\sH_1 \subset \cdots \subset \sH_s=\sV\;|\; [\sH_i/\sH_{i-1}]=\gamma_i\; \forall i\}$$
by $\text{Aut}(\sV)$. It's thus enough to show

\begin{lemma}\label{lem:hyperquot is pure}
For any $\underline{\gamma} \in \Seq^{\frac12}(\alpha-d\delta)$ the projective scheme $\text{Quot}(\underline{\gamma},\sV)$ is cohomologically pure.
\end{lemma}
\begin{proof}
Fix a direct sum decomposition $\sV=\bigoplus_{i=1}^r \sO(d_i)$ with $d_1 \geq d_2 \geq \cdots \geq d_r$. Consider the co-character 
$$\chi: \C^* \to \prod_i \text{Aut}(\sO(d_i)) \subset \text{Aut}(\sV), \qquad t \mapsto (t^{-d_1}, t^{-d_2}, \ldots, t^{-d_r})\ .$$
It defines an action of $\C^*$ on $\text{Quot}(\underline{\gamma},\sV)$ whose fixed point scheme parametrizes those filtrations $\sH_1 \subset \cdots \subset \sH_s=\sV$ for which $\sH_j=\bigoplus_{i=1}^r \left(\sH_j \cap \sO(d_i)\right)$. We have $\text{Quot}(\underline{\gamma},\sV)^{\chi}=\bigsqcup_{\mathbf{b}} Q_{\mathbf{b}}$ where $\mathbf{b}$ runs through the set of tuples $\mathbf{b}=(\beta^j_i)^{j=1,\ldots, s}_{i=1,\ldots,r} \in ((\Z^2)^+)^{rs}$ such that
$$ \sum_i \beta^j_i= \gamma_j\ \qquad \text{and} \qquad\ \;  \sum_{j}\beta^j_i=[\sO(d_i)]=(1,d_i) $$
for any $i,j$ and where 
$$Q_{\mathbf{b}}=\{(\sH_1 \subset \cdots \subset\sH_s)\in \text{Quot}(\underline{\gamma},\sV)^\chi\;|\;[\sH_j \cap \sO(d_i) / \sH_{j-1} \cap \sO(d_i)]=\beta_i^j\;\;\forall \, i,j\}\ .$$
Observe that $Q_{\mathbf{b}}$ is a product of factors $S^k\P^1\simeq \P^k$ because $\text{Quot}((1,d-k),\sO(d))\simeq S^k\P^1$. As $\text{Quot}(\underline{\gamma},\sV)$ is projective, there is a decomposition into attracting varieties
$$\text{Quot}(\underline{\gamma},\sV)=\bigsqcup_{\mathbf{b}} Q_{\mathbf{b}}^+, \qquad Q_{\mathbf{b}}^+:=\{\sH_\bullet\;|\; \text{lim}_{t \to 0} \chi(t)\cdot \sH_\bullet \in Q_{\mathbf{b}}\}\ .$$
We claim that for each $\mathbf{b}$, the canonical morphism $\pi_\mathbf{b}:Q_{\mathbf{b}}^+ \to Q_{\mathbf{b}}$ is an affine fibration. This will prove Lemma~\ref{lem:hyperquot is pure}, and hence Proposition~\ref{prop:purity_prop}. By the Bialynicki-Birula theorem, it is enough to show that $Q_{\mathbf{b}}^+$ is smooth. We will do this by relating $Q_{\mathbf{b}}^+$ to a (smooth) cell $Z_{\underline{\gamma},\underline{\sigma}}^{\frac12}(w)$ for suitable $w, \underline{\sigma}$. Set $\sV_l:=\bigoplus_{i=1}^l \sO(d_i)$ so that we have a filtration $\sV_1 \subset \cdots \subset \sV_r=\sV$. The morphism $\pi_\mathbf{b}$ sends a filtration $\sG_\bullet \in \widetilde{Coh}^{\frac12}_{\underline{\gamma}}$ to
$$\pi_\mathbf{b}(\sG_{\bullet})_j=\left(\bigoplus_l\: \sG_j \cap \sV_l/\sG_{j} \cap \sV_{l-1}\right)_j\ .$$
Because $\text{Ext}^1(\sO(d_i),\sO(d_j))=\{0\}$ whenever $d_i\leq d_j$, any coherent sheaf $\sF$ equipped with a filtration $\sF_1 \subset \cdots \subset \sF_r$ with $\sF_i/\sF_{i-1} \simeq \sO(d_i)$ must be isomorphic to $\sV$. Equivalently, the map $$\widetilde{Coh}^{\frac12}_{\underline{\sigma}} \;\underset{\prod_i Coh_{(1,d_i)}}{\times}\;(\text{Pic}^{d_1} \times \cdots \text{Pic}^{d_r}) \to Coh_{\alpha-d\delta}$$
factors through $X_{\sV,0}$ and is a flag variety bundle over its image, with fiber isomorphic to $G/B$ for $G=\text{Aut}(\sV)/\text{rad}(\text{Aut}(\sV))$.
Note that by construction, $\mathbf{b}$ determines an element of $W(\underline{\gamma},\underline{\sigma})$, where $\underline{\sigma}=((1,d_i))_i$. Consider $$Z^{\frac12,\circ}_{\underline{\gamma},\underline{\sigma}}(\mathbf{b}):=Z^{\frac12}_{\underline{\gamma},\underline{\sigma}}(\mathbf{b})\;\underset{\prod_i Coh_{(1,d_i)}}{\times}\;(\text{Pic}^{d_1} \times \cdots \text{Pic}^{d_r}),$$
an open substack of the smooth cell $Z^{\frac12}_{\underline{\gamma},\underline{\sigma}}(\mathbf{b})$.
From the above, we deduce that there is a canonical smooth $G/B$ fibration $Z^{\frac12,\circ}_{\underline{\gamma},\underline{\sigma}}(\mathbf{b}) \to Q_{\mathbf{b}}^+$, which shows that $Q_{\mathbf{b}}^+$ is smooth as wanted.
\end{proof}
\end{proof}

From the above proof we also obtain

\begin{corollary}\label{cor:pure coh tilde}
For any pair $(\sV,d)$ with $\sV \in Coh^{\frac{1}{2}}_{\alpha-d\delta}$ a vector bundle on $\P^1$ and $d \geq 0$ and for any $\underline{\sigma}$ such that $\sum_i \sigma_i=\alpha$ the stack 
$$\widetilde{Coh}^{\frac12}_{\underline{\sigma}} \underset{Coh_\alpha}{\times} X_{\sV,d}$$
is cohomologically pure.
\end{corollary}

\begin{proposition}\label{prop:CohoPureSteinberg}
The following holds:
\begin{enumerate}
    \item [(i)] Assume that $\nu \not\in \frac12 + \Z$. Then $Z^\nu_{\ua,\ua'}$ is cohomologically pure for any $\ua,\ua'$.
    \item [(ii)] Assume that $\nu \in \frac12 + \Z$. Then for any $\ua,\ua' \in \Seq^{\nu}(\alpha)$, the stack $Z^{\nu}_{\ua,\ua'}$ admits an exhaustion by  cohomologically pure quasi-compact open substacks $Z^{\frac{1}{2},I_m}_{\ua,\ua'}$. In particular, it is cohomologically pure. 
\end{enumerate}
\begin{proof}
  The first statement follows from the fact that the partition~(\ref{eq:Partition}) consists of finitely many strata,  each of which is cohomologically pure, together with Lemma~\ref{lem:OpennessStrata}. The second statement is deduced from the definition, along with Proposition \ref{prop:purity_prop}.
\end{proof}
\end{proposition}

\subsubsection{The local open exhaustion}\label{sec:LocalOpenExhaustion} Although it is well suited to the tilting equivalence between $\text{Coh}\P^1$ and $\text{Rep}Q$, we do not know how to prove that $Z_{\ua,\ua'}^{\frac12,I_\sigma}(w)$ is cohomologically pure. We now introduce another open exhaustion which does satisfy this condition. For any $\sigma \in ]\frac12,1]$ and any $\ua,\ua' \in \Seq^{\frac12}(\alpha)$ we set
$$Y_{\ua,\ua'}^{\frac12,I_\sigma}:=Z^{\frac12}_{\ua,\ua'} \underset{\prod_i Coh^{\frac12}_{\alpha_i}}{\times} \prod_i Coh^{\frac12,I_\sigma}_{\alpha_i}$$
where the morphism $Z^{\frac12}_{\ua,\ua'} \to \prod_i Coh^{\frac12}_{\alpha_i}$ is the composition of maps
$$\xymatrix{
Z^{\frac12}_{\ua,\ua'} \ar[r] & \widetilde{Coh}^{\frac12}_{\ua} \ar[r]^-{\pi_{\ua}} & \prod_i Coh^{\frac12}_{\alpha_i}\ .
}$$
There are open embeddings 
$$Y^{\frac12,I_\sigma}_{\ua,\ua'} \subset Z^{\frac12,I_\sigma}_{\ua,\ua'} \subset Z^{\frac12}_{\ua,\ua'}$$
and $Y^{\frac12,I_\sigma}_{\ua,\ua'}$ is quasi-compact.
Intersecting with the partition \eqref{eq:Partition} yields a (finite) partition into locally closed substacks
\begin{equation}\label{eq: partition Y}
Y^{\frac12,I_\sigma}_{\ua,\ua'} =\bigsqcup_{w} Y^{\frac12,I_\sigma}_{\ua,\ua'}(w)\ .
\end{equation}

\begin{proposition}\label{prop:purity local} Fix $\alpha \in (\Z^2)^+$, $I=I_\sigma$ and $\ua,\ua'\in \Seq^{\frac12}(\alpha)$. Then
\begin{enumerate}
\item[(i)] For any $w \in W(\ua,\ua')^{\frac12}$ the stack $Y^{\frac12,I}_{\ua,\ua'}(w)$ is cohomologically pure,
\item[(ii)] The stack $Y^{\frac12,I}_{\ua,\ua'}$ is cohomologically pure.
\end{enumerate}
\end{proposition}
\begin{proof} Statement (ii) is a consequence of (i), since the partition \eqref{eq: partition Y} is locally closed and finite. We turn to (i). First observe that there is a factorization 
$$\xymatrix{
\pi_w : Z^{\frac12}_{\ua,\ua'}(w) \ar[r]^{f} & \prod_{i} \widetilde{Coh}^{\frac12}_{\underline{\gamma}^{(i)}} \ar[r]^-{\left( q_{\underline{\gamma}^{(i)}}\right)_i}& \prod_{i,j} \Coh^{\frac12}_{w_{ij}} } $$
where $\underline{\gamma}^{(i)}=(w_{i1},w_{i2},\cdots,w_{i,s})$, and $f$ is the morphism which associates to a pair $(F_\cdot, G_\cdot)$ the tuple of filtrations on $F_1, F_2/F_1, \ldots, F_{r}/F_{r-1}$ induced by $G_\cdot$. By the same arguments as in the proof of Lemma~\ref{lem:dimStrata}, the morphism $f$ is an iterated vector bundle stack. Next, by construction, 
$$Y^{\frac12,I}_{\ua,\ua'}(w)=Z^{\frac12}_{\ua,\ua'}(w) \underset{\prod_{i} \widetilde{Coh}^{\frac12}_{\underline{\gamma}^{(i)}}}{\times}\prod_{i} \widetilde{Coh}^{\frac12,I}_{\underline{\gamma}^{(i)}}$$
thus it suffices to show that for any $\underline{\gamma}$ the stack $\widetilde{Coh}^{\frac12,I}_{\underline{\gamma}}$ is cohomologically pure. This follows from Corollary~\ref{cor:pure coh tilde} together with the fact that there is a finite, locally closed partition $Coh^{\frac12,I}_{\alpha}=\bigsqcup_{\sV,d} X_{\sV,d}$.
\end{proof}

\subsection{Isomorphisms on strata}

  Recall from Corollary~\ref{cor:QuasiCompact} the stack isomorphism 
    $$\bigsqcup_{w\in W(\ua,\ua')^{\nu_n,I}}Z^{\nu_n,I}_{\ua,\ua'}(w)=Z^{\nu_n,I}_{\ua,\ua'}\ \cong Z^{\frac{1}{2},I}_{\ua,\ua'}=\bigsqcup_{w\in W(\ua,\ua')^{\frac{1}{2},I}} Z^{\frac{1}{2},I}_{\ua,\ua'}(w)$$
    for large enough $n$. We next show that this isomorphism is compatible with the respective cell partitions of $Z^{\nu_n,I}_{\ua,\ua'}$ and $Z^{\frac12,I}_{\ua,\ua'}$.

 \begin{proposition}\label{prop:strata}
  Fix $\alpha\in (\Z^2)^+$, an interval $I=]\sigma-1,\frac{1}{2}]$, and  $\ua$, $\ua'\in \mathrm{Seq}^{\frac{1}{2}}(\alpha)$. There exists $n_0\in \N$ such that for any $n\geq n_0$, we have 
    $$Z^{\frac{1}{2},I}_{\ua,\ua'}(w) \cong Z^{\nu_n,I}_{\ua,\ua'}(w) .$$
  for any $w\in W(\ua,\ua')^{\frac{1}{2},I}$.
  \begin{proof}
  Let $s$ and $l$ be the length of $\ua$ and $\ua'$ respectively.
  By Lemma \ref{lem:finiteness}, the set $W(\ua,\ua')^{\frac{1}{2},I}$ is finite, it follows that 
   $$\underline{\alpha_i}=(w(i,1),w(i,2),\cdots,w(i,l)) \text{  and } \underline{\alpha'_j}=(w(j,1),w(j,2),\cdots,w(j,s))$$ 
  of various $w\in W(\ua,\ua')^{\frac{1}{2},I}$ have finitely many choices for each $1\leq i\leq s$ and $1\leq j\leq l$. Note that for any $(F_{\bdot},F'_{\bdot})\in Z^{\frac{1}{2},I}_{\ua,\ua'}$, $F_i/F_{i-1}$ and $F'_j/F'_{j-1}$ belong to $\Coh^{\frac{1}{2},(\nu_{n_0}-1,\frac{1}{2}]}$ for some $n_0$ by Lemma \ref{qclem}. Applying Corollary \ref{cor:QuasiCompact} to $I'=(\nu_{n_0}-1,\frac{1}{2}]$, all $\underline{\alpha_i}\in \mathrm{Seq}(\alpha_i)$, $\underline{\alpha'_j}\in \mathrm{Seq}(\alpha'_j)$ for $1\leq i\leq s$, $1\leq j\leq l$ and $w\in W(\ua,\ua')^{\frac{1}{2},I}$, then there exists $n_1\geq n_0$ such that for any $n\geq n_1$, we have
     $$\widetilde{Coh}^{\frac{1}{2}, I'}_{\underline{\alpha_i}} \cong \widetilde{Coh}^{\nu_{n_1}, I'}_{\underline{\alpha_i}}\ \  \text{and}\ \  \widetilde{Coh}^{\frac{1}{2}, I'}_{\underline{\alpha'_j}} \cong \widetilde{Coh}^{\nu_{n_1}, I'}_{\underline{\alpha'_j}}, $$
  for all $\underline{\alpha_i}$ and $\underline{\alpha'_j}$, which means all factors $(F_{\bdot},F'_{\bdot})_{i,j}$ of $(F_{\bdot},F'_{\bdot})\in Z^{\frac{1}{2},I}_{\ua,\ua'}$ lie in $\Coh^{\frac{1}{2},(\nu_{n_1}-1,\frac{1}{2}]}$. 

  For any $(F_{\bdot},F'_{\bdot})\in Z^{\frac{1}{2},I}_{\ua,\ua'}\cong Z^{\nu,I}_{\ua,\ua'}$, denote by $F_i\cap^{\nu} F'_j$ the subobject $F_i\cap F'_j$ in $\Coh^{\nu}$. We claim that $F_i\cap^{\frac{1}{2}} F'_j = F_i\cap^{\nu} F'_j$ as subobjects of $F_s=F'_l$ for any $\nu\leq \nu_{n_1}$. As a result, we have
     $$\frac{F_i\cap^{\frac{1}{2}}F'_j}{F_{i-1}\cap^{\frac{1}{2}}F'_j +^{\frac{1}{2}} F_i\cap^{\frac{1}{2}}F'_{j-1}}= \frac{F_i\cap^{\nu} F'_j}{F_{i-1}\cap^{\nu} F'_j +^{\nu} F_i\cap^{\nu} F'_{j-1}}.$$
  Equivalently, we obtain that $(F_{\bdot},F'_{\bdot})\in Z^{\frac{1}{2},I}_{\ua,\ua'}(w)$ if and only if $(F_{\bdot},F'_{\bdot})\in Z^{\nu,I}_{\ua,\ua'}(w)$.
  
  Now we begin to prove the claim. Firstly, note that any short exact sequence in $\Coh^{\frac{1}{2}}$ with three terms in $\Coh^{]\nu-1,\frac{1}{2}]}$ is also a short exact sequence in $\Coh^{\nu}$, because any triangle in $D^b(\Coh^{\nu})$ with three terms in $\Coh^{\nu}$ is induced by a short exact sequence in $\Coh^{\nu}$. Then we have $F_i\cap^{\frac{1}{2}} F'_j\subseteq F_i$ and $F_i\cap^{\frac{1}{2}} F'_j\subseteq F'_j$ in $\Coh^{\nu}$. In particular $F_i\cap^{\frac{1}{2}} F'_j \subseteq F_i\cap^{\nu} F'_j$.

  On the other hand, note that $F_i/(F_i\cap^{\nu} F'_j)\cong (F_i+^{\nu}F'_j)/F'_j$ in $\Coh^{\nu}$ is a subobject of $F'_l/F'_j\in \Coh^{\nu,]\nu-1,\frac{1}{2}]}$, it follows that $F_i/(F_i\cap^{\nu} F'_j)$ belongs to $\Coh^{\nu,]\nu-1,\frac{1}{2}]}$. Hence, the following short exact sequence in $\Coh^{\nu}$
    $$0\lrw F_i\cap^{\nu} F'_j  \lrw F_i \lrw F_i/(F_i\cap^{\nu} F'_j)\lrw  0 \text{ in } \Coh^{\nu}$$
  has three terms in $\Coh^{\nu,]\nu-1,\frac{1}{2}]}$, which means that it is also a short exact sequence in $\Coh^{\frac{1}{2}}$. In particular, $F_i\cap^{\nu} F'_j\subseteq F_i$ in $\Coh^{\frac{1}{2}}$. Similarly, we have $F_i\cap^{\nu} F'_j\subseteq F'_j$ in $\Coh^{\frac{1}{2}}$, it follows that $F_i\cap^{\nu} F'_j\subseteq F_i\cap^{\frac{1}{2}} F'_j$. Finally we obtain that $F_i\cap^{\nu} F'_j= F_i\cap^{\frac{1}{2}} F'_j$.  
  \end{proof}
 \end{proposition}
  Proposition~\ref{prop:strata} also establishes a bijection between sets $W(\ua,\ua')^{\frac{1}{2},I}$ and $W(\ua,\ua')^{\nu,I}$ for $\nu$ close enough to $\frac{1}{2}$.
 \begin{corollary}\label{CorSets}
  Keep notations as above. Then there exists $\nu_0\in (\frac{1}{2},\sigma]$ such that for any $\frac{1}{2}<\nu\leq \nu_0$, we have 
     $$W(\ua,\ua')^{\frac{1}{2},I}= 
        W(\ua,\ua')^{\nu,I}$$
 \end{corollary}
 \medskip

\subsection{Partial order on cells} We next introduce a partial order $\le_{\nu}$ on $W^\nu(\ua,\ua')$. Set
  \[ w([1\!:\!i],[1\!:\!j])
  := \sum_{\substack{1 \le i' \le i \\ 1 \le j' \le j}} w_{i'j'}.\]
  For $w, w' \in W^\nu(\ua,\ua')$ we put 
  $ w \le_{\nu} w'$ if and only if
  $$ w([1:i],[1:j])\geq_{\nu}w'([1:i],[1:j])
  \quad \quad (\forall\; 1 \le i \le l,\ 1 \le j \le s)\ ,$$
  where $\alpha_1\geq_{\nu}\alpha_2$ means $\alpha_1-\alpha_2\in K_0(\Coh^{\nu})^+$.  We then set
   $$ Z^{\nu}_{\ua,\ua'}(\leq w)_=\bigsqcup_{w'\leq_{\nu} w}Z^{\nu}_{\ua,\ua'}(w').$$
 When there is no risk of confusion we simply write $\geq$ instead of $\geq_\nu$.

\medskip
 
\begin{lemma}\label{lem:OpennessStrata}
For any $\nu$ and any $\ua,\ua'$, 
$Z^{\nu}_{\ua,\ua'}(\leq w)$ is closed in $Z^{\nu}_{\ua,\ua'}$ and $Z^{\nu}_{\ua,\ua'}(w)$ is open in $Z^{\nu}_{\ua,\ua'}(\leq w)$.
 \end{lemma}
\begin{proof}
For $\nu\not\in \frac12 + \Z$ this follows e.g. from \cite[Lemma 2.8]{Przezdziecki2019}. A similar proof can be given in the case $\nu=\frac{1}{2}$.
Instead, we will deduce it from the case $\nu \not\in \frac12 + \Z$ using the comparison results of the previous section. Let $w \in W(\ua,\ua')^{\frac12}$. It suffices to prove that for any interval $I_\sigma=]\sigma-1,\frac12]$ with $\sigma>\frac12$ the constructible substack 
$$Z_{\ua,\ua'}^{\frac12,I_\sigma}(\leq w):=\bigsqcup_{w'\leq_{\frac12} w} Z_{\ua,\ua'}^{\frac12,I_\sigma}(w')$$
is closed in $Z_{\ua,\ua'}^{\frac12,I_\sigma}$. By Proposition~\ref{prop:strata} and Corollary~\ref{CorSets},  we have 
$$W(\ua,\ua')^{\frac12,I_\sigma}=W(\ua,\ua')^{\nu,I_\sigma}, \qquad Z^{\frac12,I_\sigma}_{\ua,\ua'}(w)=Z^{\nu,I_\sigma}_{\ua,\ua'}(w)$$
for $\nu$ close enough to $\frac12$, hence it is enough to check that the restriction of the partial orders $\leq_{\frac12}$ and $\leq_{\nu}$ to $W(\ua,\ua')^{\frac12,I_\sigma}$ coincide for $\nu$ close enough to $\frac12$. This follows from the facts that the intervals $I_\nu, I_{\frac12}$ are open on the left and that $W(\ua,\ua')^{\frac12,I_{\sigma}}$ is finite, so that 
$\min\{\mu(w_{ij})\}>-\frac12$, where the minimum is taken over all the entries $w_{ij}$ of the elements of
$ W(\ua,\ua')^{\frac12,I_{\sigma}}$.
\end{proof}


When $\nu \not\in \frac12+\Z$, the poset $W(\ua,\ua')^{\nu}$ is finite and admits a refinement to a total order. Things are considerably more complicated when $\nu=\frac12$, as illustrated by the following example.
\begin{example} 
  Take $\alpha=(4,0)$ and $\ua=\ua'=((2,0),(2,0))\in \Seq^{\frac{1}{2}}(\alpha)$. Then  $W^{\frac{1}{2}}(\ua,\ua')$ consists of the following matrices:
  \[u_n=\begin{bmatrix} (2,-n) &(0,n)\\ (0,n) &(2,-n)  \end{bmatrix},~
  w_m=\begin{bmatrix} (1,m) &(1,-m)\\(1,-m) &(1,m)  \end{bmatrix},~
   v_{d}=\begin{bmatrix} (0,d) &(2,-d)\\(2,-d) &(0,d)    \end{bmatrix},\]
  where $n,d\ge 0$ and $m\in \Z$. By Lemma~\ref{lem:dimStrata}, the dimension of $Z(w)=Z^{\frac12}_{\ua,\ua'}(w)$ is equal to $-12, -9$ or $-8$ depending on which group $w$ belongs to. The partial order is (in this special case) linear, given by:
  \[u_0<\cdots< u_n<\cdots<w_m<\cdots <w_0<\cdots<w_{-m}<\cdots<v_{d}<\cdots<v_0.\]
Observe that the poset intervals of the form $[u_n,w_m], [u_n,v_m]$ or $[w_n,v_m]$ are all infinite. The restriction of $\geq$ to the three blocks is respectively isomorphic to $(\N, \geq)$, $(\Z, \geq)$ and $(-\N,\geq)$.   
Note that $Z^{\mathrm{rk}=0}:=\bigcup_{d\ge 0} Z(\ge v_d)$ is an open exhaustion by well-approximated, cohomologically pure stacks. Hence
  \[ \mathrm{H}_*(Z^{\mathrm{rk}=0}) = \varprojlim_{d}\, \mathrm{H}_*\bigl(Z(\ge v_d)\bigr)\]
carries a filtration whose associated graded is $\bigoplus_{d \ge 0}\mathrm{H}_*(Z(v_d))$. However, this is not the case of the open filtrations $Z(\ge w_l)$ or $Z(\ge u_n)$. 
On the opposite, the filtration $Z^{\mathrm{rk}=2}:= \bigcup_{n\ge 0} Z(\le u_n)$ exhibits $Z^{\mathrm{rk}=2}$ naturally as a colimit, with respect to closed immersions, of well-approximated, cohomologically pure stacks. The pushforward maps of the closed embeddings
  \[(i_n)_*: \text{H}_*(Z(\le u_n))\to \text{H}_*(Z(\le u_{n+1}))\]
  are thus injective. However, there is \textit{à priori} no reason to expect the induced map  
  \[ \varinjlim_{n} \;\mathrm{H}_*(Z(\le u_n))\to \mathrm{H}_*(Z^{\mathrm{rk}=2}).\]
  to be a dense embedding in general.
\end{example}

\medskip
\section{Schur and KLR algebras of coherent sheaves on {\texorpdfstring{$\P^1$}{Lg}}}\label{sec:Schur algebras}
\subsection{Definition of Schur and KLR algebras }
  Recall that we denote by $\text{H}_{*}(-)$ the Borel-Moore homology. For any $\nu$ and any $\alpha\in \Z^2$ and sequences $\underline{\alpha_1}$, $\underline{\alpha_2}$, $\underline{\alpha_3}\in \Seq^\nu(\alpha)$, there is a convolution product
    \begin{equation}\label{eq:def_conv_prod}
        *: \text{H}_{i}(Z^{\nu}_{\underline{\alpha_1},\underline{\alpha_2}})\otimes \text{H}_{j}(Z^{\nu}_{\underline{\alpha_2},\underline{\alpha_3}})\lrw \text{H}_{i+j-2\text{dim}(\widetilde{Coh}^\nu_{\ua_2})}(Z^{\nu}_{\underline{\alpha_1},\underline{\alpha_3}}),
    \end{equation}
    defined as follows. Consider the triple product $X=\widetilde{Coh}^\nu_{\underline{\alpha_1}} \times \widetilde{Coh}^\nu_{\underline{\alpha_2}} \times \widetilde{Coh}^\nu_{\underline{\alpha_3}}$ and the (smooth) projections
    $p_{ij}: X \to \widetilde{Coh}^\nu_{\underline{\alpha_i}} \times \widetilde{Coh}^\nu_{\underline{\alpha_j}}$. Then $p_{13}$ restricts to a proper morphism 
    $$p_{13}:p_{12}^{-1}(Z^\nu_{\underline{\alpha_1},\underline{\alpha_2}}) \underset{X}{\times}p_{23}^{-1}(Z^\nu_{\underline{\alpha_3},\underline{\alpha_3}}) \to Z^\nu_{\underline{\alpha_1}, \underline{\alpha_3}}$$
and for $c_{ij} \in \text{H}_{*}(Z^{\nu}_{\underline{\alpha_i},\underline{\alpha_j}})$ we set
$$c_{12} * c_{23}=p_{13!}\left(p_{12}^*(c_{12}) \cap p_{23}^*(c_{23}) \right)$$
where in the above $\cap$ stands for the intersection with the supports in the smooth ambient stack $X$. Note that when $\nu \in \Z + \frac{1}{2}$, the stacks $Z^\nu_{\underline{\gamma},\underline{\gamma}'}$ are only locally of finite type, but the convolution product is well-defined. Indeed, the convolution product being local on the base $Coh^\nu_\alpha$, we may define and compute it on every quasi-compact open substack (which we may choose to be a global quotient stack) using the formalism of \cite[Sec. 2.6]{Chriss2010}. For the reader's comfort, a direct sheaf-theoretic construction of this convolution product is given in Appendix~\ref{app:convolution product}.

\smallskip

In particular, there is a restricted convolution product 
    $$*: \text{H}_{*}(Z^{\nu,I}_{\ua_1,\ua_2})\otimes \text{H}_{*}(Z^{\nu,I}_{\ua_2,\ua_3})\lrw \text{H}_{*}(Z^{\nu,I}_{\ua_1,\ua_3})$$
for any interval $I \subset ]\nu-1,\nu]$. Summing over all pairs of sequences in $\Seq^\nu(\alpha)$, we obtain associative algebras 
$$A_{\alpha}^{\nu}:=\bigoplus\limits_{\ua,\ua'\in \Seq^\nu(\alpha)} \text{H}_{*}( Z^{\nu}_{\ua,\ua'}), \qquad A_{\alpha}^{\nu,I}:=\bigoplus\limits_{\ua,\ua'\in \Seq^\nu(\alpha)} \text{H}_{*}( Z^{\nu,I}_{\ua,\ua'}).$$
We set $A^\nu_{\ua,\ua'}=\text{H}_*(Z^\nu_{\ua,\ua'})$ so that $A^\nu_{\alpha}=\bigoplus_{\ua,\ua'}A^\nu_{\ua,\ua'}$.

\medskip

\begin{definition}
  We call $A^{\nu}_{\alpha}$, resp. $A^{\nu,I}_{\alpha}$  the \textit{Schur algebra} associated to $(\nu,\alpha)$, resp. $(\nu,I,\alpha)$.
\end{definition}

We turn to the KLR algebra associated to $\P^1$, which one may expect to be easier to understand than the Schur algebra $A^{\frac12}_\alpha$. Set 
  $$(\Z^2)^{+}_\circ=\{(1,l), l \in \Z\} \cup \{(0,1)\}\subseteq (\Z^2)^+\ .$$ 
Define
  \[\Seq^{\frac{1}{2}}_{\circ}(\alpha):=\{\ua\in\Seq^{\frac{1}{2}}(\alpha)|~\forall\;i,\;\alpha_i\in (\Z^2)_{\circ}^+\}.\]
  We use the notation $\ua\vDash \alpha$ to denote $\ua\in \Seq^{\frac{1}{2}}_{\circ}(\alpha)$. 

\medskip

\begin{definition}~
Let $\alpha \in (\Z^2)^+$. We define the \textit{KLR algebra} associated with $\alpha$ as the graded subalgebra
$$\sR_{\alpha}=\sR_{\alpha}^{\frac12} := \bigoplus_{\ua,\ua'\vDash \alpha} \text{H}_*(Z^{\frac12}_{\ua,\ua'})$$
of the Schur algebra $A^{\frac12}_\alpha$.
\end{definition}

\begin{remark}~
  When $\nu \not\in \Z + \frac{1}{2}$, the Schur algebra is often called \textit{quiver-Schur algebra} in the literature. It is studied in e.g. \cite{Stroppel2012}. For $Q$ a finite-type quiver,  the quiver Schur algebra is Morita equivalent to the \textit{quiver-Hecke algebra}, also known as \textit{Khovanov-Lauda-Rouquier algebra}, which admits a well-known presentation as a diagrammatic algebra. Outside of finite type, as it is the case here, no explicit presentation of quiver-Schur algebras is known. 
\end{remark}

  Note that the canonical restriction map $\iota_{\nu,I}: A_{\alpha}^{\nu} \to A_{\alpha}^{\nu,I}$ is an algebra homomorphism for any $\nu$ and any $I \subset ]\nu-1,\nu]$. Moreover, there is an isomorphism of algebras
     $$A_\alpha^\nu \simeq \varprojlim \;A^{\nu,I}_\alpha$$
  where the limit is taken over all strict open intervals $I \subset ]\nu-1,\nu]$, with respect to inclusion, in the category of $\Seq(\alpha)$-graded vector spaces. Indeed, the substacks $Coh^{\nu,I}_\alpha$ form a quasi-compact open exhaustion of $Coh^\nu_\alpha$, hence their preimages under $\pi_{\ua} \times \pi_{\ua'}$ form a quasi-compact open exhaustion of $Z^\nu_{\ua,\ua'}$. 

\subsection{The limit construction}\label{sec:limit}
  In this section we use the quasi-compact approximation theorem to relate the algebra $A^{\frac{1}{2}}_\alpha$, whose structure is of interest, say, in the context of the geometric Langlands program (for $\P^1$) to the more studied Kronecker quiver Schur algebras $A^{\nu}_\alpha$ (for $\nu \not\in \Z + \frac{1}{2}$). We begin with the following consequence of Theorem~\ref{qcthm}.

\begin{proposition}
  Fix $\alpha \in (\Z^2)^+$, $\sigma \in ]\frac{1}{2},1]$ and set $I_\sigma=]\sigma-1,\frac12]$.  Let $\ua,\ua', \ua'' \in \Seq^{\frac12}(\alpha)$. Then there exists $n_0$ such that for any $n>n_0$ the following diagram of convolution products commutes
    $$\begin{tikzcd}
    \text{H}_{*}(Z^{\frac{1}{2},I}_{\ua,\ua'})\otimes \text{H}_{*}(Z^{\frac{1}{2},I}_{\ua',\ua''})\arrow[r,"*"]\arrow[d,"\Psi \otimes \Psi"']   &\text{H}_{*}(Z^{\frac{1}{2},I}_{\ua,\ua''})\arrow[d,"\Psi"]\\
    \text{H}_{*}(Z^{\nu_n,I}_{\ua,\ua'})\otimes \text{H}_{*}(Z^{\nu_n,I}_{\ua',\ua''})\arrow[r,"*"]   &\text{H}_{*}(Z^{\nu_n,I}_{\ua,\ua''})\ .
  \end{tikzcd}$$
\end{proposition}
\begin{proof}
 The statement is proved in the same manner as Corollary \ref{cor:QuasiCompact}.
\end{proof}

  As a direct consequence of the above Proposition, we see that for any fixed $I=]\sigma-1,\frac12]$ with $\sigma \in ]\frac12,1]$, there are canonical isomorphisms $H_*(Z^{\nu_n,I}_{\ua,\ua'}) \simeq H_*(Z^{\nu_m,I}_{\ua,\ua'})$ for $n,m \gg 0$, in a way compatible with the convolution product. We may thus unambiguously write $$ \lim_n \;A^{\nu_n,I}_\alpha \simeq A^{\frac12,I}_\alpha\ ,$$
though one should be careful not to confuse this 'naive' limit with a projective limit. 

Moreover, by Proposition \ref{prop:purity_prop}, the restriction map $A_{\alpha}^{\frac{1}{2}}\to A^{\frac{1}{2},I}_{\alpha}$ is a surjection, and when $n$ is large enough, the restriction map $A^{\nu}_{\alpha}\to A^{\nu_n,I}_{\alpha}$ becomes surjective.  

\begin{corollary}\label{cor:LimitConstruction}
  There is a canonical isomorphism of algebras
 $$A_{\alpha}^{\frac{1}{2}}=\varprojlim\limits_{I}\lim \limits_{n}\,A^{\nu_n,I}_{\alpha}.$$
\end{corollary}

\medskip
\subsection{Polynomial representation}
  Fixing $\ua_3=(\alpha)$ in \eqref{eq:def_conv_prod}, noting that 
  $Z^\nu_{\ua,\alpha}=\widetilde{Coh}^\nu_{\ua}$ and
  allowing $\ua_1,\ua_2$ to vary gives rise to a continuous representation 
  $$*:~\bigoplus_{\ua,\ua'} \text{H}_*(Z^{\nu}_{\ua,\ua'})\otimes \bigoplus_{\ua}\text{H}_*(\widetilde{Coh}^\nu_{\ua})\lrw \bigoplus_{\ua}\text{H}_*(\widetilde{Coh}^{\nu}_{\ua})$$
  of $A_{\alpha}^{\nu}$ on $P^\nu_\alpha:=\bigoplus_{\ua}\text{H}_*(\widetilde{Coh}^{\nu}_{\ua})$.
Using Lemma~\ref{lem:dimStrata}, we have an identification
    $$\bigotimes_{i}\text{H}^*(Coh^{\nu}_{\alpha_i}) \xrightarrow{\sim} \text{H}_*(\widetilde{Coh}^{\nu}_{\ua}), \qquad c \mapsto \pi_{\ua}^*(c) \cap [\widetilde{Coh}^\nu_{\ua}] $$
and hence a canonical isomorphism
$$P^\nu_\alpha=\bigoplus_{\ua \in \Seq^\nu(\alpha)} \bigotimes_{i} \text{H}^{*}(Coh^\nu_{\alpha_i}) .$$
When $\nu=\frac12$, Theorem \ref{ThmHeinloth} allows us to further identify $P^\nu_\alpha$ with an explicit direct sum of polynomial algebras. 

If $\nu\neq \frac{1}{2}$, from \cite[Proposition 4.4]{Przezdziecki2019} it follows that the polynomial representation $P^{\nu}_{\alpha}$ is faithful. 
We conjecture that the same holds for $\nu=\frac12$.

\begin{conjecture} \label{conj:PolyRep}
 For any $\alpha \in (\Z^2)^+$, the polynomial representation $P^{\frac12}_\alpha$ of $A^{\frac12}_\alpha$ is faithful.  
\end{conjecture}

\medskip
\section{PBW bases of Schur algebras {\texorpdfstring{$A_{\alpha}^{\nu}$}{Lg}}}\label{sec:PBW bases}

\medskip

PBW-type bases of Schur algebras associated to quivers have been considered, e.g. in \cite{Przezdziecki2019}. The aim of this section is to construct explicit topological bases of $A_\alpha^{\frac12}$ and $\sR_\alpha$. All of the constructions performed in Sections \ref{sec:4.1}-\ref{sec:proof PBW basis} can be done for any $\nu$ (and are typically much simpler for $\nu \not\in \frac12 + \Z)$; we drop the superscript $\nu$ in an effort to unburden the notation.

\subsection{Split and merge}\label{sec:4.1}
In this paragraph we set $\nu=\frac12$ and drop the superscript $\nu$.
  For $\alpha \in (\Z^2)^+$ and $\ua,\ua'\in \Seq(\alpha)$, we say $\ua'$ subdivides $\ua$, if $\ua'$ is obtained from $\ua$ by replacing each $\alpha_i$ with $\alpha^{(1)}_i,\cdots,\alpha_i^{(k_i)}$. For such pair of sequences, we associate two matrices as follows (where $s$ is the length of $\ua$):
  \[w_{\ua,\ua'}=\begin{bmatrix}
      \alpha_1^{(1)} &\cdots &\alpha_1^{(k_1)}\\
       & & &\alpha_2^{(1)} &\cdots &\alpha_2^{(k_2)} \\
      & & & & & &\ddots\\
      & & & & & & &\alpha_s^{(1)} &\cdots &\alpha_s^{(k_s)}
  \end{bmatrix}, \qquad  w_{\ua',\ua}=(w_{\ua,\ua'})^T\]
  There are closed embeddings:
    \[\begin{aligned}
    &i_{\ua,\ua'}:~\widetilde{Coh}_{\ua'}=Z_{\ua,\ua'}(w_{\ua,\ua'})\hookrightarrow Z_{\ua,\ua'}\\
    &i_{\ua',\ua}:~\widetilde{Coh}_{\ua'}=Z_{\ua',\ua}(w_{\ua',\ua})\hookrightarrow Z_{\ua',\ua}.
   \end{aligned}\]
  In particular, we define 
  \[S^{\ua}_{\ua'}:=(i_{\ua,\ua'})_*[Z_{\ua,\ua'}(w_{\ua,\ua'})]\in A_{\ua,\ua'},\qquad M^{\ua'}_{\ua}:=(i_{\ua',\ua})_*[Z_{\ua',\ua}(w_{\ua',\ua})]\in A_{\ua',\ua},\]
  and call them \textit{split} and \textit{merge}, respectively.

  It is customary to represent these elements diagrammatically (see e.g. \cite{Przezdziecki2019}). Take $\ua=(\alpha_s,\cdots,\alpha_1)$, and $\ua'=(\alpha_s,\cdots,\alpha_{i+1},\alpha'_i,\alpha''_i,\alpha_{i-1},\cdots,\alpha_1)$ for some $\alpha'_i+\alpha''_i=\alpha_i$. For such pair of sequences, we draw the corresponding split and merge as follows:
\begin{equation}
\begin{aligned}
  \begin{tikzpicture}[thick,scale=0.9,font=\footnotesize]
  \draw (1.5,0.5) parabola (1,0);
  \draw (1.5,0.5) parabola (2,0);
  \draw (1.5,1)--(1.5,0.5);
  \draw (3,1)--(3,0);
  \draw (4.5,1)--(4.5,0.5);
  \draw (4.5,0.5) parabola (5,0);
  \draw (4.5,0.5) parabola (4,0);
  \draw (6,1)--(6,0);

  \node at (0.5,0.5) {$ S_{\ua}^{\ua'}=$};
  \node[anchor=north] at (1,0) {$\alpha'_i$};  
  \node[anchor=north] at (2,0) {$\alpha''_i$};  
  \node[anchor=south] at (1.5,1) {$\alpha_i$}; 
  \node at (2.2,0.5) {:=};
  \node[anchor=north] at (3,0) {$\alpha_1$};
  \node[anchor=south] at (3,1) {$\alpha_1$};
  \node at (3.5,0.5) {$\cdots$};
  \node at (5.5,0.5) {$\cdots$};
  \node[anchor=north] at (4,0) {$\alpha''_i$};  
  \node[anchor=north]at (5,0) {$\alpha'_i$}; 
  \node[anchor=south] at (4.5,1) {$\alpha_i$};
  \node[anchor=north] at (6,0) {$\alpha_s$};
  \node[anchor=south] at (6,1) {$\alpha_s$};
  \node[anchor=west] at (6,0.5) {,};
  \end{tikzpicture}\qquad
 \begin{tikzpicture} [thick,font=\footnotesize]
  \draw (1.5,-0.5) parabola (1,0);
  \draw (1.5,-0.5) parabola (2,0);
  \draw (1.5,-0.5)--(1.5,-1); 
  \draw (3,0)--(3,-1);
  \draw (4.5,-0.5)--(4.5,-1);
  \draw (4.5,-0.5) parabola (5,0);
  \draw (4.5,-0.5) parabola (4,0);
  \draw (6,-1)--(6,0);

  \node at (0.5,-0.5) {$M^{\ua}_{\ua'}=$};
  \node[anchor=south] at (1,0) {$\alpha'_i$};
  \node[anchor=south] at (2,0) {$\alpha''_i$};  
  \node[anchor=north] at (1.5,-1) {$\alpha_i$};
  \node at (2.2,-0.5) {:=};
  \node[anchor=south] at (3,0) {$\alpha_1$};
  \node[anchor=north] at (3,-1) {$\alpha_1$};
  \node at (3.5,-0.5) {$\cdots$};
  \node at (5.5,-0.5) {$\cdots$};
  \node[anchor=south] at (6,0) {$\alpha_s$};
  \node[anchor=north] at (6,-1) {$\alpha_s$};
  \node[anchor=south] at (4,0) {$\alpha'_i$};
  \node[anchor=south] at (5,0) {$\alpha''_i$}; 
  \node[anchor=north] at (4.5,-1) {$\alpha_i$};
  \node[anchor=west] at (6,-0.5) {.};
  \end{tikzpicture}
\end{aligned}
\end{equation}
  We call such splits and merges $elementary$. We always read diagrams from top to bottom.  Using the associativity of the convolution product it is easy to see that every split (resp. merge) can be written as a product of elementary splits (resp. merges).

  Now let $\ua, \ua'$ be as above and set $\ua''=(\alpha_s,\cdots,\alpha_{i+1},\alpha''_i,\alpha'_i,\alpha_{-1},\cdots,\alpha_1)$. Consider the element $$C_{\ua',\ua''}:=S^{\ua'}_{\ua}*M^{\ua}_{\ua''}\in A_{\ua',\ua''}\ ,$$ which we will call an \textit{elementary crossing}. Diagrammatically, we may draw it as follows:
  \[\begin{tikzpicture}[thick,scale=0.9, font=\footnotesize]
  \draw (-2.5,-0.5) parabola (-3,0);
  \draw (-2.5,-0.5) parabola (-2,0);
  \draw (-2.5,-0.5)--(-2.5,-1);
  \draw (-2.5,-1) parabola (-3,-1.5);
  \draw (-2.5,-1) parabola (-2,-1.5);
  \draw (-1,-1.5)--(-1,0);
  \draw (0.5,-0.5) parabola (0,0);
  \draw (0.5,-0.5) parabola (1,0);
  \draw (0.5,-0.5)--(0.5,-1);
  \draw (0.5,-1) parabola (0,-1.5);
  \draw (0.5,-1) parabola (1,-1.5);
  \draw (2,-1.5)--(2,0);

  \node at (-3.7,-0.75) {$C_{\ua',\ua''}=$};
  \node at (-1.7,-0.75) {$:=$};
  \node[anchor=south] at (-3,0) {$\alpha'_i$};
  \node[anchor=south] at (-2,0) {$\alpha''_i$};
  \node[anchor=north] at (-3,-1.5) {$\alpha''_i$};
  \node[anchor=north] at (-2,-1.5) {$\alpha'_i$};
  \node[anchor=south] at (-1,0) {$\alpha_1$};
  \node at (-0.5,-0.75) {$\cdots$};
  \node[anchor=south] at (2,0) {$\alpha_s$};
  \node[anchor=south] at (0,0) {$\alpha'_i$};
  \node[anchor=south] at (1,0) {$\alpha''_i$};
  \node[anchor=north] at (0,-1.5) {$\alpha''_i$};
  \node[anchor=north] at (1,-1.5) {$\alpha'_i$};
  \node[anchor=north] at (-1,-1.5) {$\alpha_1$};
  \node at (1.5,-0.75) {$\cdots$};
  \node[anchor=north] at (2,-1.5) {$\alpha_s$};
  \end{tikzpicture}\]

   Denote by $\Delta_{\ua'}\subset\widetilde{Coh}_{\ua'}\times_{{Coh}_{a}}\widetilde{Coh}_{\ua'}$ the diagonal. For any
  \[c=c_1 \otimes \cdots \otimes c_s\in  \bigotimes_{i}\text{H}^*(Coh_{\alpha'_i})\]
  the element $c\cap [\Delta_{\ua'}] \in \text{H}_*(Z_{\ua',\ua'})$ acts on $\bigotimes_i\text{H}^*(Coh_{\alpha'_i}) \subset P_\alpha$ by multiplication by $c$. Diagrammatically, this is represented by the picture
   \[\begin{tikzpicture}[thick,scale=0.9, font=\footnotesize]
  \draw (-2.5,-1.5)--(-2.5,0);
  \draw (-1,-1.5)--(-1,0);
  \draw (0.5,-1.5)--(0.5,0);
  \draw (2,-1.5)--(2,0);
  \filldraw[color=black, fill=black](-2.5,-0.75) circle (.1);
  \filldraw[color=black, fill=black](-1,-0.75) circle (.1);
  \filldraw[color=black, fill=black](0.5,-0.75) circle (.1);
  \filldraw[color=black, fill=black](2,-0.75) circle (.1);
  \node at (-2.8,-0.7) {$c_1$};
  \node at (-1.3,-0.7) {$c_2$};
  \node at (1.05,-0.7) {$c_{s-1}$};
  \node at (2.3,-0.7) {$c_s$};
  \node at (-3.7,-0.75) {$c:=$};
  \node[anchor=north] at (-2.5,-1.5) {$\alpha'_1$};
  \node[anchor=south] at (-2.5,0) {$\alpha'_1$};
   \node[anchor=north] at (-1,-1.5) {$\alpha'_2$};
  \node[anchor=south] at (-1,0) {$\alpha'_2$};
   \node[anchor=north] at (0.5,-1.5) {$\alpha'_{s-1}$};
  \node[anchor=south] at (0.5,0) {$\alpha'_{s-1}$};
   \node[anchor=north] at (2,-1.5) {$\alpha'_S$};
  \node[anchor=south] at (2,0) {$\alpha'_s$};
  \node at (-0.25,-1.2) {$\cdots$};
  \node at (-0.25,-0.3) {$\cdots$};
  \end{tikzpicture}\]

As we will show in Section \ref{sec:PBW basis}, the KLR and Schur algebras $\sR_\alpha, A_\alpha$ are generated, as associative algebras, by elementary crossings, split and merge operators together with operators of multiplication by tautological classes. 

    
\subsection{Construction of sequences of matrices}\label{sec:def sequence matrices}
  To each $w \in W(\ua,\ua')$, we will associate a sequence
   \[\underline{w} := (w_1, w_2, \ldots, w_t)\]
   of 'elementary' (in particular, minimal for the partial order) matrices. Before describing the recipe, we give an example.

\begin{example}Let us consider the case of a $2\times 3$ matrix $w \in W(\ua,\ua')$. The process is best described by the diagram in Figure~\ref{fig:split cross and merge1}
\begin{figure}[h]
\centering
\[\begin{tikzpicture}[thick,scale=0.9, font=\footnotesize]
    \draw (-2,-0.2) parabola (-3,-0.6);
    \draw (-2,-0.2) parabola (-1,-0.6);
    \draw (-2,0.1)--(-2,-0.6);
    \draw (1,-0.2) parabola (0,-0.6);
    \draw (1,-0.2) parabola (2,-0.6);
    \draw (1,0.1)--(1,-0.6);
    \draw (-2, -1)--(-2,-2.9);
    \draw (-3, -1)--(-3,-4.8);
    \draw (2, -1)--(2,-4.8);
    \draw (1, -1)--(1,-2.9);
    \draw (-0.5,-1.3) parabola (-1,-1);
    \draw (-0.5,-1.3) parabola (0,-1);
    \draw (-0.5, -1.3)--(-0.5,-1.6);
    \node[anchor=north] at (-0.5,-1.5) {$w_{13}+w_{21}$};
    \draw (-0.5, -1.9)--(-0.5,-2.2);
    \draw (-0.5,-2.2) parabola (0,-2.5);
    \draw (-0.5,-2.2) parabola (-1,-2.5);
  \node[anchor=south] at (1,0) {$\alpha_2$};
  \node[anchor=south] at (-2,0) {$\alpha_1$};
  \node[anchor=south] at (-3,-1) {$w_{11}$};
  \node[anchor=south] at (-2,-1) {$w_{12}$};
  \node[anchor=south] at (-1,-1) {$w_{13}$};
  \node[anchor=south] at (0,-1) {$w_{21}$};
  \node[anchor=south] at (1,-1) {$w_{22}$};
  \node[anchor=south] at (2,-1) {$w_{23}$};
  \node[anchor=south] at (-1,-2.9) {$w_{21}$};
  \node[anchor=south] at (0,-2.9) {$w_{13}$};
  \draw (0.5,-3.2) parabola (0,-2.9);
    \draw (0.5,-3.2) parabola (1,-2.9);
    \draw (0.5, -3.2)--(0.5,-3.5);
    \node[anchor=north] at (0.5,-3.4) {$w_{13}+w_{22}$};
    \draw (0.5, -3.8)--(0.5,-4.1);
    \draw (0.5,-4.1) parabola (1,-4.4);
    \draw (0.5,-4.1) parabola (0,-4.4);
 \node[anchor=south] at (-1,-4.8) {$w_{12}$};
  \node[anchor=south] at (0,-4.8) {$w_{22}$};
\node[anchor=south] at (-2,-4.8) {$w_{21}$};
  \node[anchor=south] at (1,-4.8) {$w_{13}$};
    \draw (-1.5,-3.2) parabola (-2,-2.9);
    \draw (-1.5,-3.2) parabola (-1,-2.9);
    \draw (-1.5, -3.2)--(-1.5,-3.5);
    \node[anchor=north] at (-1.5,-3.4) {$w_{12}+w_{21}$};
    \draw (-1.5, -3.8)--(-1.5,-4.1);
    \draw (-1.5,-4.1) parabola (-1,-4.4);
    \draw (-1.5,-4.1) parabola (-2,-4.4);
    \draw (1.5,-5.1) parabola (1,-4.8);
    \draw (1.5,-5.1) parabola (2,-4.8);
    \draw (-2.5,-5.1) parabola (-3,-4.8);
    \draw (-2.5,-5.1) parabola (-2,-4.8);
    \draw (-0.5,-5.1) parabola (-1,-4.8);
    \draw (-0.5,-5.1) parabola (0,-4.8);
    \draw (-0.5, -5.1)--(-0.5,-5.4);
    \node[anchor=north] at (-0.5,-5.3) {$\alpha_2'$};
    \node[anchor=north] at (-2.5,-5.3) {$\alpha'_1$};
    \node[anchor=north] at (1.5,-5.3) {$\alpha'_3$};
    \draw (-2.5, -5.1)--(-2.5,-5.4);
    \draw (1.5, -5.1)--(1.5,-5.4);
  \end{tikzpicture}\]
  \caption{The split, cross and merge sequence}\label{fig:split cross and merge1}
  \end{figure}
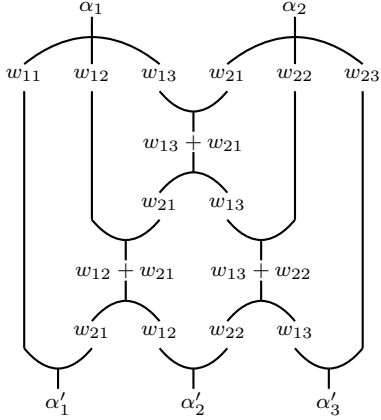
 to which is associated the following sequence of matrices. To the first split, we attach the matrix
$$\begin{bmatrix}w_{11} & w_{12} & w_{13} & & & \\ & & & w_{21} & w_{22} & w_{23}\end{bmatrix}\ ,$$
to the first crossing we attach the pair of matrices
$$\begin{bmatrix}w_{11} & & & & \\ & w_{12} & & & \\ & & w_{13} & & \\ & & w_{21} & & \\ & & & w_{22} & \\ & & & & w_{23}\end{bmatrix},\quad \begin{bmatrix}w_{11} & & & & & \\ & w_{12} & & & & \\ & & w_{21} & w_{13} & & \\  & & & & w_{22} & \\ & & & & & w_{23}\end{bmatrix},$$
and likewise for the next two crossings; finally to the last merge we associate the matrix
$$\begin{bmatrix}w_{11} & &\\ w_{21} & &\\ &w_{12} & \\ & w_{22} &\\ & & w_{13}\\ & & w_{23}\end{bmatrix}\ . $$
{\flushright{\qed}}
\end{example}

\medskip

In the general case, write $\ua=(\alpha_s,\alpha_{s-1},\ldots,\alpha_1),\; \ua'=(\alpha_l',\alpha'_{l-1},\ldots, \alpha'_1)$. We begin by applying a split operation to $\ua$ to decompose each $\alpha_i$ into an $r$-tuple
  \[
    (w_{il}, \ldots, w_{i2}, w_{i1}),
  \]
  to obtain the first row. Next, we perform successive individual crossings, each corresponding to a sequence of a merge followed by a split, in order to bring the first row $(w_{s,l}, w_{s,l-1}, \ldots, w_{12},w_{11})$ into
  $$(w_{s,l},w_{s-1,l}, \ldots, w_{1,l},\ldots, w_{s,2},\ldots, w_{12}, \ldots, w_{s,1}, \ldots ,w_{21},w_{11})\ .
  $$
  The sequence of crossings we choose is given by a reduced decomposition of the permutation described by the following picture (the choice of reduced decomposition is irrelevant)

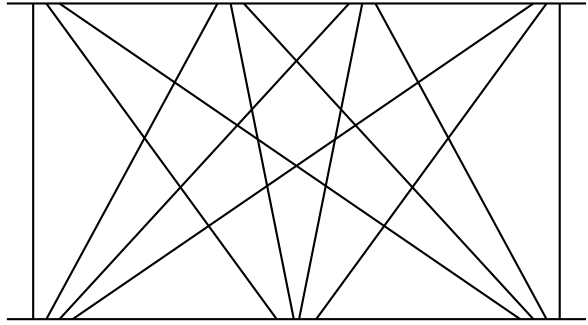
\begin{figure}[h]
\centering
\[\begin{tikzpicture}[thick,scale=0.7, font=\footnotesize]
\draw(-5.5,6)--(5.5,6);
\draw(-5.5,0)--(5.5,0);
    \draw (-5,6)-- (-5,0);
    \draw (-4.75,6)-- (-0.375,0);
    \draw (-4.5,6)-- (4.25,0);
    \draw (-1.5,6)-- (-4.75,0);
    \draw (-1.25,6)-- (-0.05,0);
    \draw (-1.,6)-- (4.5,0);
    \draw (1.,6)-- (-4.5,0);
    \draw (1.25,6)-- (0.05,0);
    \draw (1.5,6)-- (4.75,0);
    \draw (4.5,6)-- (-4.25,0);
    \draw (4.75,6)-- (0.375,0);
    \draw (5,6)-- (5,0);
  \end{tikzpicture}\]

\caption{The permutation associated to $s=4, l=3$.}\label{fig:reduced decomp}
\end{figure}

\vspace{.1in}
  
Finally, we perform $l$ merges, each of length $s$, to reach the tuple $\ua'$. The total number of crossings in the process is easily computed to be equal to $\frac14 ls(l-1)(s-1)$, to which we add the first split and the last merge. To each step is canonically associated a matrix, minimal for the partial order, as above. We denote by $$\underline{w}:= (w_1, \ldots, w_t)$$ the sequence of minimal matrices attached to our initial matrix $w \in W(\ua,\ua')$. We also let 
$$\underline{\beta}_0=\ua, \;\underline{\beta}_1,\; \ldots, \;\underline{\beta}_{t-1},\: \underline{\beta}_t=\ua'$$ be the sequence of elements in $\Seq(\alpha)$ such that $w_i\in W(\underline{\beta}_{i-1}, \underline{\beta}_i)$.
The integer $t=2+\frac12 ls(l-1)(s-1)$ may be thought of as some form of length function on $W(\ua,\ua')$.

\subsection{The projection map}\label{sec:projection map}
  For any $w\in W(\ua,\ua')$ with associated sequence $\underline{w}=(w_1,w_2,\cdots,w_t)$, set $\widetilde{S}(w):= \prod_i \widetilde{Coh}_{\underline{\beta_{i}}}$ and consider the stack
  $$S(w):=Z_1 \underset{\widetilde{S}(w)}{\times} Z_2 \underset{\widetilde{S}(w)}{\times} \cdots \underset{\widetilde{S}(w)}{\times} Z_t\ ,$$
  where
  $$Z_i=\widetilde{S}(w) \underset{\widetilde{Coh}_{\underline{\beta}_{i-1}} \times\widetilde{Coh}_{\underline{\beta}_{i}} }{\times} Z_{\underline{\beta}_{i-1},\underline{\beta}_i}(w_i)\ .$$
In words, $S(w)$ parametrizes tuples $(F, F_\cdot^{(0)}, \ldots, F_\cdot^{(t)})$ where $F \in Coh_\alpha$, $F_\cdot^{(i)}$ is a filtration of $F$ of type $\underline{\beta}_{i}$ for any $i$ and such that  the flags $F_\cdot^{(i-1)}, F_\cdot^{(i)}$ are in relative position $w_i$. 
  Note that the morphisms $Z_{\underline{\beta}_{i-1} ,\underline{\beta}_i}(w_i) \to \widetilde{Coh}_{\underline{\beta}_{i-1}} \times \widetilde{Coh}_{\underline{\beta}_{i}}$ are proper by the minimality of $w_i$ for all $i$. Thus the canonical morphism $S(w) \to \widetilde{S}(w)$ is proper as well. There are natural projection morphisms
  $$\Pi_{i,j}: {S}(w) \lrw \widetilde{Coh}_{\underline{\beta}_{i}} \times \widetilde{Coh}_{\underline{\beta}_{j}}$$
  and in particular a map $\Pi_{0,t}: {S}(w) \lrw \widetilde{Coh}_{\ua}\times \widetilde{Coh}_{\ua'}$.
We set
    $$U(w):=\Pi^{-1}_{0,t}(Z_{\ua,\ua'}(w))\ .$$
    
\begin{proposition}\label{prop:InverseImage}
 For any $\ua,\ua'\in \Seq(\alpha)$ and any $w\in W(\ua,\ua')$ we have
 \begin{equation}\label{eq:propInverseImage1}
\Pi_{0,t}(S(w))  \subseteq Z_{\ua,\ua'}(\leq w)\ ,
 \end{equation}
 \begin{equation}\label{eq:propInverseImage2}
(\Pi_{0,t})_{|U(w)}: U(w)\xrightarrow{\sim} Z_{\ua,\ua'}(w)\ . 
 \end{equation}
 In particular, $U(w)$ is an open substack of $S(w)$.
\end{proposition}
\begin{proof}
Fix $\ua,\ua' \in \Seq(\alpha)$ of length $s$ and $l$ respectively and let $w \in W(\ua,\ua')$. Let $(F_\cdot, G_\cdot) \in \widetilde{Coh}_{\ua} \times \widetilde{Coh}_{\ua'}$ and assume that there exists a sequence of flags $H^{(i)}_\cdot \in \widetilde{Coh}_{\underline{\beta}_i}$ for $i=0, \ldots, t$ such that $H^{(0)}_\cdot=F_\cdot, H^{(t)}_\cdot=G_\cdot$ and 
\begin{equation}\label{eq:propInverseImage3}
(H^{(i-1)}_\cdot, H^{(i)}_{\cdot})\in Z_{\underline{\beta}_{i-1},\underline{\beta}_i}(w_i), \qquad  (i=1, \ldots, t)\ .
\end{equation}
The Proposition may be deduced from the following two statements:

(a) $[F_i \cap G_j] \geq \sum_{k\leq i, h \leq j}w_{hk}$,

(b) If $(F_\cdot,G_\cdot) \in Z_{\ua,\ua'}(w)$ then there exists a unique collection of flags $(H^{(i)}_\cdot)_{i=0, \ldots, t}$ satisfying conditions \eqref{eq:propInverseImage3}.

 By construction, the factors of the flag $$H^{(1)}_\cdot=(H^{(1)}_1 \subseteq H^{(1)}_2 \subseteq \cdots \subseteq H^{(1)}_{sl})$$
are of class $w_{11}, w_{12}, \ldots, w_{s,l-1},w_{sl}$, and $F_i=H^{(1)}_{il}$ for $i=1, \ldots, s$. Likewise, the factors of the flag $$H^{(t-1)}_\cdot=H^{(t-1)}_1 \subseteq H^{(t-1)}_2 \subseteq \cdots \subseteq H^{(t-1)}_{ls}$$
are of class $w_{11}, w_{21}, \ldots, w_{s-1,l},w_{sl}$, and $G_i=H^{(t-1)}_{si}$ for $i=1, \ldots, l$. Next, consider Figure~\ref{fig:reduced decomp}.  Let us call \textit{weight} of a strand the class $w_{ij}$ attached to it. We may cut Figure~\ref{fig:reduced decomp} into $t$ horizontal slices, each one besides the first and the last corresponding to a single crossing, and a relative position of flags $(H^{(i)}_\cdot, H^{(i+1)}_\cdot)$. In particular, a slice with a crossing between the $\ell$th and $\ell+1$st strands corresponds to a pair of flags $(K_\cdot, L_\cdot)$ for which $K_j=L_j$ for all $j \neq \ell$. In that situation, we have the following lower bound for the class of $K_{\ell} \cap L_{\ell}$:
\begin{equation}\label{eq:propInverseImage4}
[K_{\ell} \cap L_{\ell}] \geq [K_{\ell}] + [L_\ell]-[L_{\ell+1}]=[K_\ell]-w
\end{equation}
where $w$ is the weight of the strand going from right to left.

 We next make the following combinatorial observation: there is a canonical bijection between the set of connected regions in the diagram~\ref{fig:reduced decomp} representing the sequence of crossings, and the set of partitions inside a rectangular box of size $(s,l)$. The bijection is obtained by considering, along each horizontal line, the weights of the strands crossed from left to right to reach a given region $R$, and associating the corresponding subset of the matrix $w=(w_{ij})_{i,j}$ (which forms a partition $\lambda$ by construction). An example is provided by Figure~\ref{fig:split cross and merge2} below.
 To any $(H^{(0)}_\cdot, \ldots, H^{(t)}_\cdot) \in S(w)$ there corresponds a filling of the regions of the diagram of crossings by the corresponding flags of subobjects so that one reads the filtration at step $1 \leq \ell \leq t-1$ by going along the horizontal line a level $\ell$ from left to right.

\smallskip

To prove (a), we will exhibit an explicit subobject $U \subseteq F_i \cap G_j$ of class at least $\sum_{k\leq i, h \leq j}w_{hk}$, built out of the flags $H^{(1)}_\cdot$ and $H^{(t-1)}_\cdot$.
We claim that $F_i \cap G_j$ has a subobject whose class is at least the sum of the weights of all strands whose (top) starting point is in position $\leq il$ and whose (bottom) ending point is in position $\leq sj$, i.e. all the strands starting from the first $i$ clusters on the top and ending in the first $j$ clusters at the bottom. This total weight is easily verified to be equal to $\sum_{h\leq i, k \leq j}w_{hk}$. We now prove the claim. Consider the strand of weight $w_{ij}$ linking the top position $(i-1)l+j$ to the bottom position $(j-1)s+i$. This strand crosses $(i-1)(l-j)$ strands coming from the left and $(j-1)(s-i)$ strands coming from the right. As we go from top to bottom, this strand corresponds to a sequence of objects, say 
$L_1, \ldots, L_t$, starting with $L_1=H^{(1)}_{(i-1)l+j} \subseteq F_i$ on the top slice and ending with $L_t=H^{(t-1)}_{(j-1)s+i} \subseteq G_j$ on the bottom slice. Note that $L_i \subset L_{i+1}$ if the $i$th slice corresponds to a crossing coming from the right. Applying the lower bound \eqref{eq:propInverseImage4} every time we encounter a crossing coming from the left we deduce that
$$\left[ \bigcap_{u=1}^t L_u \right] \geq \sum_{\substack{h< i\\ k \leq l}} w_{hk} + \sum_{k \leq j} w_{ik}-\sum_{\substack{h<i\\ k >j}} w_{hk}=\sum_{\substack{h\leq i\\ k \leq j}}w_{hk}\ .$$
This proves (a) since by construction $\left[ \bigcap_{u=1}^t L_u \right] \subseteq F_i \cap G_j$. 

We now turn to (b). If $(F_\cdot,G_\cdot) \in Z(w)$ then in the argument for (a) above, we necessarily have $\bigcap_{u=1}^t L_u=F_i \cap G_j$ for any pair $(i,j)$. This means that at every crossing corresponding to a pair of flags $(K_\cdot, L_\cdot)$ with $K_i=L_i$ for all $i \neq \ell$ we have
\begin{equation}\label{eq:propInverseImage5}
K_\ell \cap L_\ell =K_{\ell-1}=L_{\ell-1}, \qquad (K_{\ell}/K_{\ell-1}) \oplus (L_\ell/K_{\ell-1})=K_{\ell+1}/K_{\ell-1}\ .
\end{equation}
We deduce that the subobject $H_R$ associated to a given region $R$ is of class precisely given by the sum of the weights of the associated partition $\lambda_R$. Moreover, at every crossing, knowing the subobjects associated to the top and bottom regions is enough to determine the subobjects associated to the left and right regions (as the intersection, resp. sum). Graphically, this says that if $\lambda_R, \lambda_{R'}$ are partitions of the same size $n$ containing a common subpartition of size $n-1$ and if $\lambda_{R''}=\lambda_R \cup \lambda_{R'}$ then $H_{R''}=H_R + H_{R'}$. Arguing as in (a) above using the inclusion relation between subobjects as we go through crossings, one can see that the subobject $H_R$ corresponding to the region associated to a rectangular partition $\lambda=(i^j)$ contains $F_i \cap G_j$. The equation $[F_i\cap G_j]=\sum_{h\leq i, k \leq j}w_{ij}$ thus implies that $H_R=F_i \cap G_j$. This in turn fully determines the subobject associated to any region; indeed, any partition may be obtained by successive union of partitions containing one box less, starting from the collection of rectangular partitions.
\end{proof}
\medskip

\begin{example} Let us consider the case $s=2,r=3$ again. The correspondence between regions of the crossing diagram and the partitions fitting in a $2\times 3$ rectangle is given as follows:

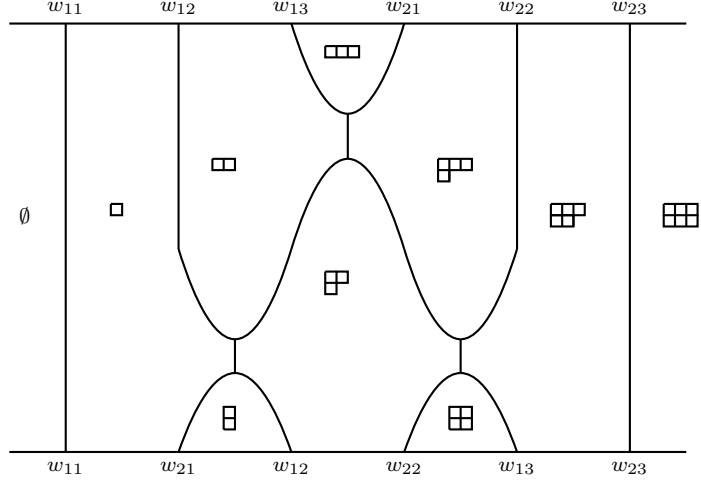
\begin{figure}[h]
\centering
\[\begin{tikzpicture}[thick,scale=1.5, font=\footnotesize]
    \draw (-3.5,-1)--(2.5,-1);
    \draw (-3.5,-4.8)--(2.5,-4.8);
    \draw (-2, -1)--(-2,-3);
    \draw (-3, -1)--(-3,-4.8);
    \draw (2, -1)--(2,-4.8);
    \draw (1, -1)--(1,-3);
    \draw (-0.5,-1.8) parabola (-1,-1);
    \draw (-0.5,-1.8) parabola (0,-1);
    \draw (-0.5, -1.8)--(-0.5,-2.2);
    \draw (-0.5,-2.2) parabola (0,-3);
    \draw (-0.5,-2.2) parabola (-1,-3);
    \draw (-2.6,-2.6)--(-2.5,-2.6)--(-2.5,-2.7) --(-2.6,-2.7)--(-2.6,-2.6);
     \draw (-1.7,-2.2)--(-1.7,-2.3)--(-1.5,-2.3) --(-1.5,-2.2)--(-1.7,-2.2);
     \draw (-1.6,-2.2)--(-1.6,-2.3);
     \draw (-0.7,-1.2)--(-0.7,-1.3)--(-0.4,-1.3) --(-0.4,-1.2)--(-0.7,-1.2);
     \draw (-0.6,-1.2)--(-0.6,-1.3);
      \draw (-0.5,-1.2)--(-0.5,-1.3);
     \draw (0.3,-2.2)--(0.3,-2.3)--(0.6,-2.3) --(0.6,-2.2)--(0.3,-2.2);
    \draw (0.4,-2.2)--(0.4,-2.4);
    \draw (0.5,-2.2)--(0.5,-2.3);
    \draw (0.3,-2.3)--(0.3,-2.4)--(0.4,-2.4) --(0.4,-2.3);
    \draw (1.3,-2.6)--(1.3,-2.7)--(1.6,-2.7) --(1.6,-2.6)--(1.3,-2.6);
    \draw (1.3,-2.7)--(1.3,-2.8)--(1.5,-2.8) --(1.5,-2.7)--(1.3,-2.7);
    \draw (1.4,-2.8)--(1.4,-2.6);
     \draw (1.5,-2.8)--(1.5,-2.6);
     \draw (2.3,-2.6)--(2.3,-2.7)--(2.6,-2.7) --(2.6,-2.6)--(2.3,-2.6);
     \draw (2.3,-2.7)--(2.3,-2.8)--(2.6,-2.8) --(2.6,-2.7);
     \draw (2.4,-2.8)--(2.4,-2.6);
     \draw (2.5,-2.8)--(2.5,-2.6);
    \draw (-0.7,-3.2)--(-0.7,-3.4)--(-0.6,-3.4) --(-0.6,-3.2)--(-0.5,-3.2)--(-0.5,-3.3)--(-0.7,-3.3);
    \draw (-0.7,-3.2)--(-0.6,-3.2);
    \draw (-1.6,-4.6)--(-1.5,-4.6)--(-1.5,-4.4)--(-1.6,-4.4)--(-1.6,-4.6);
    \draw (-1.6,-4.5)--(-1.5,-4.5);
    \draw (0.4,-4.6)--(0.6,-4.6)--(0.6,-4.4)--(0.4,-4.4)--(0.4,-4.6);
    \draw (0.5,-4.6)--(0.5,-4.4);
    \draw (0.4,-4.5)--(0.6,-4.5);
\node[anchor=west]  at (-3.5,-2.7) {$\emptyset$};
  \node[anchor=south] at (-3,-1) {$w_{11}$};
  \node[anchor=south] at (-2,-1) {$w_{12}$};
  \node[anchor=south] at (-1,-1) {$w_{13}$};
  \node[anchor=south] at (0,-1) {$w_{21}$};
  \node[anchor=south] at (1,-1) {$w_{22}$};
  \node[anchor=south] at (2,-1) {$w_{23}$};
  \draw (0.5,-3.8) parabola (0,-3);
    \draw (0.5,-3.8) parabola (1,-3);
    \draw (0.5, -3.8)--(0.5,-4.1);
    \draw (0.5,-4.1) parabola (1,-4.8);
    \draw (0.5,-4.1) parabola (0,-4.8);
\node[anchor=north] at (-3,-4.8) {$w_{11}$};
  \node[anchor=north] at (2,-4.8) {$w_{23}$};
 \node[anchor=north] at (-1,-4.8) {$w_{12}$};
  \node[anchor=north] at (0,-4.8) {$w_{22}$};
\node[anchor=north] at (-2,-4.8) {$w_{21}$};
  \node[anchor=north] at (1,-4.8) {$w_{13}$};
    \draw (-1.5,-3.8) parabola (-2,-3);
    \draw (-1.5,-3.8) parabola (-1,-3);
    \draw (-1.5, -3.8)--(-1.5,-4.1);
    \draw (-1.5,-4.1) parabola (-1,-4.8);
    \draw (-1.5,-4.1) parabola (-2,-4.8);
  \end{tikzpicture}\]
  \caption{Regions and partitions}\label{fig:split cross and merge2}
  \end{figure}

The subobjects associated to the partitions are
$$R_{\emptyset}=\{0\}, \qquad R_{(1)}=F_1\cap G_1, \qquad R_{(2)}=F_1\cap G_2, \qquad R_{(3)}=F_1,$$
$$R_{(31)}=F_1+ G_1, \qquad R_{(32)}=F_1+G_2, \qquad R_{(1^2)}=G_1, \qquad R_{(2^2)}=G_2,\qquad R_{(3^2)}=F_2=G_3\ .$$

\flushright{\qed}
\end{example}

\bigskip

\subsection{PBW bases}\label{sec:PBW basis} We finally arrive at the notion of PBW bases.

\subsubsection{Definition} Fix $\ua,\ua'\in \Seq(\alpha)$ and $w\in W(\ua,\ua')$ and denote by $\underline{w}=(w_1,\cdots,w_{t})$ the associated sequence, as in Section~\ref{sec:def sequence matrices}. 
Thanks to Proposition~\ref{prop:InverseImage}, we may define the following element of $\text{H}_{*}(Z_{\ua,\ua'}(\leq w))$:
    $$b_{w}:=[Z_{\ua,\underline{\beta}_1}(w_1)]*[Z_{\underline{\beta}_1,\underline{\beta}_2} (w_2)]*\cdots*[Z_{\underline{\beta}_{t-1},\ua'}(w_t)]\in \text{H}_*(Z_{\ua,\ua'}(\le w))\ .$$
More generally, for any
  $$c\in \Lambda_{w}:=\text{H}^*(\widetilde{Coh}_{\underline{\beta}_1})\cong  \bigotimes_{i,j}\text{H}^*(Coh_{w_{i,j}})$$
  define
    $$b_{c,w}:=[Z_{\ua,\underline{\beta}_1}(w_1)]*c*[Z_{\underline{\beta}_1,\underline{\beta}_2} (w_2)]*\cdots*[Z_{\underline{\beta}_{t-1},\ua'}(w_t)]\in \text{H}_*(Z_{\ua,\ua'}(\le w)).$$
Thus, for $c=\bigotimes c_{ij}$, the element $b_{c,w}$ corresponds to the diagram in Figure~\ref{fig:split cross and merge}.

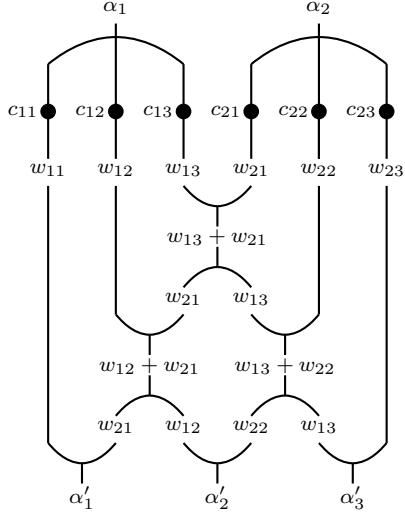
\begin{figure}[h]
\centering
\[\begin{tikzpicture}[thick,scale=0.9, font=\footnotesize]
    \draw (-2,1.2) parabola (-3,0.8);
    \draw (-2,1.2) parabola (-1,0.8);
    \draw (-2,1.4)--(-2,-0.6);
    \draw (1,1.4)--(1,-0.6);
    \draw (-3,0.8)--(-3,-0.6);
    \draw (-1,0.8)--(-1,-0.6);
    \draw (0,0.8)--(0,-0.6);
    \draw (2,0.8)--(2,-0.6);
    \filldraw[color=black, fill=black](-3,0.1) circle (0.1);
    
    \filldraw[color=black, fill=black](-2,0.1) circle (0.1);
    \filldraw[color=black, fill=black](-1,0.1) circle (0.1);
    \filldraw[color=black, fill=black](1,0.1) circle (0.1);
    \filldraw[color=black, fill=black](0,0.1) circle (0.1);
    \filldraw[color=black, fill=black](2,0.1) circle (0.1);
     \node[anchor=east] at (-3,0.1) {$c_{11}$};
     \node[anchor=east] at (-2,0.1) {$c_{12}$};
     \node[anchor=east] at (-1,0.1) {$c_{13}$};
     \node[anchor=east] at (0,0.1) {$c_{21}$};
     \node[anchor=east] at (1,0.1) {$c_{22}$};
     \node[anchor=east] at (2,0.1) {$c_{23}$};
    \draw (1,1.2) parabola (0,0.8);
    \draw (1,1.2) parabola (2,0.8);
    \draw (1,0.1)--(1,-0.6);
    \draw (-2, -1)--(-2,-2.9);
    \draw (-3, -1)--(-3,-4.8);
    \draw (2, -1)--(2,-4.8);
    \draw (1, -1)--(1,-2.9);
    \draw (-0.5,-1.3) parabola (-1,-1);
    \draw (-0.5,-1.3) parabola (0,-1);
    \draw (-0.5, -1.3)--(-0.5,-1.6);
    \node[anchor=north] at (-0.5,-1.5) {$w_{13}+w_{21}$};
    \draw (-0.5, -1.9)--(-0.5,-2.2);
    \draw (-0.5,-2.2) parabola (0,-2.5);
    \draw (-0.5,-2.2) parabola (-1,-2.5);
  \node[anchor=south] at (1,1.4) {$\alpha_2$};
  \node[anchor=south] at (-2,1.4) {$\alpha_1$};
  \node[anchor=south] at (-3,-1) {$w_{11}$};
  \node[anchor=south] at (-2,-1) {$w_{12}$};
  \node[anchor=south] at (-1,-1) {$w_{13}$};
  \node[anchor=south] at (0,-1) {$w_{21}$};
  \node[anchor=south] at (1,-1) {$w_{22}$};
  \node[anchor=south] at (2,-1) {$w_{23}$};
  \node[anchor=south] at (-1,-2.9) {$w_{21}$};
  \node[anchor=south] at (0,-2.9) {$w_{13}$};
  \draw (0.5,-3.2) parabola (0,-2.9);
    \draw (0.5,-3.2) parabola (1,-2.9);
    \draw (0.5, -3.2)--(0.5,-3.5);
    \node[anchor=north] at (0.5,-3.4) {$w_{13}+w_{22}$};
    \draw (0.5, -3.8)--(0.5,-4.1);
    \draw (0.5,-4.1) parabola (1,-4.4);
    \draw (0.5,-4.1) parabola (0,-4.4);
 \node[anchor=south] at (-1,-4.8) {$w_{12}$};
  \node[anchor=south] at (0,-4.8) {$w_{22}$};
\node[anchor=south] at (-2,-4.8) {$w_{21}$};
  \node[anchor=south] at (1,-4.8) {$w_{13}$};
    \draw (-1.5,-3.2) parabola (-2,-2.9);
    \draw (-1.5,-3.2) parabola (-1,-2.9);
    \draw (-1.5, -3.2)--(-1.5,-3.5);
    \node[anchor=north] at (-1.5,-3.4) {$w_{12}+w_{21}$};
    \draw (-1.5, -3.8)--(-1.5,-4.1);
    \draw (-1.5,-4.1) parabola (-1,-4.4);
    \draw (-1.5,-4.1) parabola (-2,-4.4);
    \draw (1.5,-5.1) parabola (1,-4.8);
    \draw (1.5,-5.1) parabola (2,-4.8);
    \draw (-2.5,-5.1) parabola (-3,-4.8);
    \draw (-2.5,-5.1) parabola (-2,-4.8);
    \draw (-0.5,-5.1) parabola (-1,-4.8);
    \draw (-0.5,-5.1) parabola (0,-4.8);
    \draw (-0.5, -5.1)--(-0.5,-5.4);
    \node[anchor=north] at (-0.5,-5.3) {$\alpha_2'$};
    \node[anchor=north] at (-2.5,-5.3) {$\alpha'_1$};
    \node[anchor=north] at (1.5,-5.3) {$\alpha'_3$};
    \draw (-2.5, -5.1)--(-2.5,-5.4);
    \draw (1.5, -5.1)--(1.5,-5.4);
  \end{tikzpicture}\]
  \caption{A PBW basis element}\label{fig:split cross and merge}
  \end{figure}

We set $$\mathbf{b}_{c,w}=(i_{\leq w})_*(b_{c,w}) \in \text{H}_*(Z_{\ua,\ua'}),$$
where $i_{\leq w}: Z_{\ua,\ua'}(\leq w) \to Z_{\ua,\ua'}$ is the closed embedding.

\medskip

\subsubsection{Main Theorem} The proof of the following result will occupy the next Section.
\begin{theorem}\label{thm:PBWbases}
  Fix $\alpha\in (\Z^2)^+$. For each $\ua,\ua'\in \Seq(\alpha)$ and $w\in W(\ua,\ua')$, choose a basis $\B_w$ of $\Lambda_{w}$. The set 
      \[\bigl\{ \mathbf{b}_{c,w} \ \big|~\ua,\ua'\in \mathrm{Seq}^{\frac{1}{2}}(\alpha),~\ 
      w\in W(\ua,\ua'),\ c\in \B_w \,\bigr\}\]
    forms a topological basis of $A^{\frac{1}{2}}_{\alpha}$.
\end{theorem}

The analog of this result in the case $\nu \not\in \frac12 + \Z$ appears in \cite[Theorem 3.25]{Przezdziecki2019}) in the context of (arbitrary) quivers. The proof of (ii) is necessarily more complex due to nature of the poset $W(\ua,\ua')^{\frac12}$ and the fact that $Z^{\frac12}_{\underline{\beta},\underline{\beta}'}(w)$ is only locally quasi-compact, see Section~\ref{sec:cellpartition}. 

\subsection{Proof of Theorem~\ref{thm:PBWbases}}\label{sec:proof PBW basis}
From now on, we will simply write $Z(\sigma)$ instead of $Z_{\beta,\beta'}(\sigma)$ --note that $\sigma$ determines $\beta ,\beta'$. Let us denote by $j_w: Z(w) \to Z(\leq w)$ the open embedding (see Lemma~\ref{lem:OpennessStrata}). 

\subsubsection{}We begin with the following preliminary result:

\begin{proposition}\label{prop:BasisforStata}
  Let $\B_{w}$ be a basis of $\Lambda_{w}$, then $\{j_w^*(b_{c,w})|\ c\in \B_{w}\}$ is a basis of $\text{H}_*(Z_{\ua,\ua'}(w),\Q)$.
\end{proposition}
\begin{proof}
By Lemma~\ref{lem:dimStrata}, the stack $Z(w)$ carries tautological bundles $E_{ij}$ pulled back from the morphism to $Coh_{w_{ij}}$, the Künneth components of whose Chern classes (freely) generate $\text{H}^*(Z(w),\Q)$. Note that by the proof of Proposition~\ref{prop:InverseImage}, the diagram
$$ \begin{tikzcd}[column sep=large]
    U(w)\arrow[d,"\Pi_{0,t}","\simeq"'] \arrow[r,"pr_1"]  &\widetilde{Coh}_{\underline{\beta}_{1}} \arrow[d,"q"] \\
    Z_{\ua,\ua'}(w) \arrow[r,"\pi_w"]&\prod_{i,j}Coh_{w_{ij}}\\
   \end{tikzcd}
$$
is commutative. Because the maps 
$$q^*: \bigotimes_{i,j}\text{H}^*(Coh_{w_{ij}},\Q) \to \text{H}^*(\widetilde{Coh}_{\underline{\beta}_1},\Q),$$
$$\pi_w^*: \bigotimes_{i,j}\text{H}^*(Coh_{w_{ij}},\Q) \to \text{H}^*(Z(w),\Q)$$ are isomorphisms and $Z(w)$ is smooth, it suffices to show that 
\begin{equation}\label{eq:proof PBW basis thm1}
j_w^*(b_{1,w})=[Z(w)] \in \text{H}_*(Z(w),\Q).
\end{equation}
For $1<i<t$ and $i$ odd (resp. even), the matrix $w_i$ represents an elementary merge (resp. split). There are thus canonical projections $\widetilde{Coh}_{\underline{\beta}_{2j\pm 1}} \to \widetilde{Coh}_{\underline{\beta}_{2j}}$ for any $j$, and we have
$$Z(w_{2j})=\widetilde{Coh}_{\underline{\beta}_{2j- 1}} \underset{\widetilde{Coh}_{\underline{\beta}_{2j}} }{\times} \widetilde{Coh}_{\underline{\beta}_{2j}} \simeq \widetilde{Coh}_{\underline{\beta}_{2j-1}}\ ,$$
$$Z(w_{2j+1})=\widetilde{Coh}_{\underline{\beta}_{2j}} \;\underset{\widetilde{Coh}_{\underline{\beta}_{2j}} }{\times} \widetilde{Coh}_{\underline{\beta}_{2j+1}} \simeq \widetilde{Coh}_{\underline{\beta}_{2j+1}}\ .$$
Let us denote by
$$\kappa_j:U_j \to \widetilde{Coh}_{\underline{\beta}_{2j-1}} \;\underset{\widetilde{Coh}_{\underline{\beta}_{2j}} }{\times} \widetilde{Coh}_{\underline{\beta}_{2j+1}}$$
the inclusion of the open substack parametrizing pairs of filtrations $(F_\cdot,G_\cdot)$ differing in a single index, say 
$$F_k=G_k \subset F_{k+1},G_{k+1} \subset F_{k+2}=G_{k+2}$$
and such that $F_{k+1}/F_k$ and $G_{k+1}/F_k$ are supplementary subspaces in $F_{k+2}/F_k$. Observe that $U_j=Z(w^{(j)})$, where
$w^{(j)}$ is a suitable matrix in $ W(\underline{\beta}_{2j-1}, \underline{\beta}_{2j+1})$ of the form
$$w^{(j)}=\begin{bmatrix} 
 w_{11} \\
 &\ddots\\
 & & &w_{ab}\\
 & &w_{cd}\\
 & & & & \ddots\\
 & & & & &w_{sr}
 \end{bmatrix}\ .$$
 In particular, $U_j$ is irreducible.

 \begin{lemma}\label{eq: proof PBW bases 2}
We have $\kappa_j^*\left([Z(w_{2j})] \star [Z(w_{2j+1})]\right)=[U_j]$.     
\end{lemma}
\begin{proof}
Let us write 
$$\underline{\beta}_{2j-1}=(a_1, a_2, \ldots, \delta,\beta, \ldots, a_N),$$
$$\underline{\beta}_{2j}=(a_1, a_2, \ldots, \delta+\beta, \ldots, a_N),$$
$$\underline{\beta}_{2j+1}=(a_1, a_2, \ldots, \beta,\delta, \ldots, a_N),$$
Using the smooth morphisms
$$\widetilde{Coh}_{\underline{\beta}_{2j- 1}} \to Coh_{a_1} \times Coh_{a_2} \times \cdots \times \widetilde{Coh}_{\delta,\beta} \times \cdots Coh_{a_N},$$
$$\widetilde{Coh}_{\underline{\beta}_{2j}} \to Coh_{a_1} \times Coh_{a_2} \times \cdots \times {Coh}_{\delta+\beta} \times \cdots Coh_{a_N},$$
$$\widetilde{Coh}_{\underline{\beta}_{2j+ 1}} \to Coh_{a_1} \times Coh_{a_2} \times \cdots \times \widetilde{Coh}_{\beta,\delta} \times \cdots Coh_{a_N},$$
we easily reduce ourselves to the case of a convolution product
\begin{equation*}
\begin{split}
\text{H}_*\left(\widetilde{Coh}_{\delta,\beta} \underset{Coh_{\delta+\beta}}{\times} Coh_{\delta+\beta},\Q\right) \otimes &\text{H}_*\left({Coh}_{\delta+\beta} \underset{Coh_{\delta+\beta}}{\times} \widetilde{Coh}_{\beta,\delta},\Q\right)\\
&\to \text{H}_*\left(\widetilde{Coh}_{\delta,\beta} \underset{Coh_{\delta+\beta}}{\times} \widetilde{Coh}_{\beta,\delta},\Q\right)
\end{split}
\end{equation*}
restricted to the open subset $U\subset \widetilde{Coh}_{\delta,\beta} \underset{Coh_{\delta+\beta}}{\times} \widetilde{Coh}_{\beta,\delta}$ parametrizing pairs $(F_1,G_1 \subset F_2=G_2)$ such that $F_1 \oplus G_1=F_2$. Note that $U \simeq Coh_\delta \times Coh_\beta$. Let us write $\gamma=\delta+\beta$,
$$Z_1 :=\widetilde{Coh}_{\delta,\beta} \underset{Coh_{\gamma}}{\times} Coh_{\gamma} \times \widetilde{Coh}_{\beta,\delta} \qquad  Z_2:=\widetilde{Coh}_{\delta,\beta} {\times} Coh_{\gamma}  \underset{Coh_{\gamma}}{\times}\widetilde{Coh}_{\beta,\delta}\ ,$$ 
$$X:= \widetilde{Coh}_{\delta,\beta} {\times} Coh_{\gamma} \times \widetilde{Coh}_{\beta,\delta}\ .$$
 Observe that $$Z_1 \underset{X}{\times} Z_2 =\widetilde{Coh}_{\delta,\beta} \underset{Coh_{\gamma}}{\times} \widetilde{Coh}_{\beta,\delta}$$ and denote by $p_i: Z_i \to X$, $i=1,2$ and $p: Z_1 \underset{X}{\times} Z_2 \to X$ the natural morphisms.
The stack $X \times \mathbb{P}^1$ carries tautological bundles $E_\delta^{(1)}, E_\beta^{(1)}, E_\gamma, E_\beta^{(2)}, E_{\delta}^{(2)}$. For simplicity, we denote by the same symbol their pullbacks by the maps $p_i, p$. Over $Z_1 \underset{X}{\times} Z_2$, we have canonical short exact sequences
$$0 \to E_\beta^{(1)} \to E_\gamma \to E_\delta^{(1)} \to 0, \qquad 0 \to E_\delta^{(2)} \to E_\gamma \to E_\beta^{(2)} \to 0\ .$$
Recall that for any $\sigma$, $\mathbb{T}_{Coh_\sigma} = \text{R}Hom(E_\sigma,E_\sigma)[1]$.
Thanks to Lemma~\ref{lem:app differential} there are distinguished triangles
$$\mathbb{T}_{\widetilde{Coh}_{\delta,\beta}} \xrightarrow{dp_{\delta,\beta}} \text{R}Hom(E_\gamma,E_\gamma)[1] \longrightarrow \text{R}Hom(E^{(1)}_\beta,E^{(1)}_\delta)[1]\xrightarrow{+1} $$
$$\mathbb{T}_{\widetilde{Coh}_{\beta,\delta}} \xrightarrow{dp_{\beta,\delta}}\text{R}Hom(E_\gamma,E_\gamma)[1] \longrightarrow\text{R}Hom(E^{(2)}_\delta,E^{(2)}_\beta)[1]\xrightarrow{+1} \ .$$
Over the open substack $U \subset Z_1 \underset{X}{\times} Z_2$, we have identifications $E_\delta^{(1)} \simeq E_\delta^{(2)},E_\beta^{(1)} \simeq E_\beta^{(2)}$ and, dropping the (superfluous) superscript, $E_\gamma \simeq E_\delta \oplus E_\beta$. This gives rise to a commutative diagram
$$ \begin{tikzcd}[column sep=large]
    & \text{R}Hom(E_\delta,E_\beta)[1] \arrow[dl] &\\
    \mathbb{T}_{\widetilde{Coh}_{\delta,\beta}} \arrow[r,"dp_{\delta,\beta}"'] &\text{R}Hom(E_\gamma,E_\gamma)[1] \arrow[u] \arrow[r] & \text{R}Hom(E_\beta,E_\delta)[1] \arrow[dl]\\
    & \mathbb{T}_{\widetilde{Coh}_{\beta,\delta}} \arrow[u,"dp_{\beta,\delta}"] &
   \end{tikzcd}
$$
and a compatible splitting
$$\text{R}Hom(E_\gamma,E_\gamma)[1] \simeq \text{R}Hom(E_\delta,E_\beta)[1] \oplus \text{R}Hom(E_\beta,E_\delta)[1] \oplus \mathbb{T}_{Coh_\delta \times Coh_\beta}\ .$$
From this we deduce a decomposition over $U$:
$$\mathbb{T}_{Z_1} \oplus \mathbb{T}_{Z_2}=p^*(\mathbb{T}_X) \oplus \mathbb{T}_{Coh_\delta \times Coh_\beta}\ ,$$
which implies that $U$ has no derived structure, i.e. that the complex $\mathbb{T}_{Z_1} \oplus \mathbb{T}_{Z_2} \to p^*(\mathbb{T}_{X})$ has Tor amplitude in $[0,1]$. We are thus in the situation of a 'clean intersection', in the terminology of \cite[Sec.2.6.]{Chriss2010}). We deduce from \cite[Prop. 2.6.47]{Chriss2010} that
$([Z_1]* [Z_2])_{|U} =[U]$ as wanted.
\end{proof}

We may now finish the proof of Proposition~\ref{prop:BasisforStata}. Thanks to Lemma~\ref{eq: proof PBW bases 2}, we are reduced to checking the following equality
\begin{equation}\label{eq:proof prop pbw 3}
j_w^*\left([Z(w_1)]* [\overline{Z(w^{(1)})}]* \cdots * [\overline{Z(w^{(n)})}]* [Z(w_t)]\right)=[Z(w)] 
\end{equation}
where $n=\frac{t-2}{2}$. This is done in exactly the same fashion as above.

\end{proof}

\medskip

\subsubsection{} Let us next show that the elements $\{\mathbf{b}_{c,w}\;|\; w \in W(\ua,\ua'), c \in \mathbb{B}_w\}$ span a dense subspace of $A_{\ua,\ua'}$. Recall that by Proposition~\ref{prop:purity local} the stack
   \[Y^{I}_{\ua,\ua'} =Z^I_{\ua,\ua'}\times_{\widetilde{Coh}_{\ua}} \widetilde{Coh}_{\ua}^{I}=\bigsqcup_{w\in W(\ua,\ua')^I} Y^I_{\ua,\ua'}(w),\]
  is cohomologically pure (see Section \ref{sec:LocalOpenExhaustion}). Moreover, the open exhaustion of $Z_{\ua,\ua'}$ by the stacks $Y^I_{\ua,\ua'}$ yields an isomorphism
   \[\text{H}_*(Z_{\ua,\ua'})=\varprojlim_m \text{H}_*(Y_{\ua,\ua'}^{I_m}).\] 
  We will show that for any $m \geq 1$ the elements
   $$\{\iota_{m}^*(\mathbf{b}_{c,w})|~w\in W(\ua,\ua')^{I_m},~c\in \B_{w}\}$$
  linearly span $\text{H}_*(Y_{\ua,\ua'}^{I_m})$, where $\iota_{m}: Y_{\ua,\ua'}^{I_m}\to Z_{\ua,\ua'}$ is the open embedding.

  Fix $m\ge 1$. Because $W(\ua,\ua')^{I_m}$ is finite by Lemma \ref{lem:finiteness},  we can refine the partial order $\le$ on $W(\ua,\ua')^{I_m}$ to a linear order $\trianglelefteq$. By the purity of each of the strata $Y^{I_m}_{\ua,\ua'}(w)$, this defines a finite filtration on $\text{H}_*(Y^{I_m}_{\ua,\ua'})$ whose associated graded is 
  \begin{equation}\label{eq:filtration Y}
  \text{gr}_\bullet\text{H}_*(Y_{\ua,\ua'}^{I_m})
  =\bigoplus_{w\in W(\ua,\ua')^{I_m}} \text{H}_{*}(Y^{I_m}_{\ua,\ua'}(w)).
  \end{equation}
   Fix some $w_0 \in W(\ua,\ua')^{I_m}$ and consider the following commutative diagram with Cartesian squares
  $$\begin{tikzcd} 
  Z_{\ua,\ua'}(w_0)\arrow[r,hook,"j_{w_0}"] &Z_{\ua,\ua'}(\le w_0)\arrow[r,"i_{\leq w_0}"] &Z_{\ua,\ua'}\\
  Y^{I_m}_{\ua,\ua'}(w_0) \arrow[r,hook,"j_{Y,w_0}"]\arrow[u,hook,"\iota_{m,w_0}"]      &Y^{I_m}_{\ua,\ua'}(\le w_0) \arrow[u,hook,"\iota_{m,\leq w_0}"]\arrow[r,"i_{Y,\leq w_0}"] &Y^{I_m}_{\ua,\ua'}\arrow[u,hook,"\iota_{m}"],
  \end{tikzcd}$$
  in which the vertical morphisms are open embeddings.
By base change we have
$$\iota_m^*(\mathbf{b}_{c,w_0})=\iota_{m}^*(i_{\leq w_0})_*(b_{c,w_{0}})=(i_{Y,\leq w_0})_*\iota_{m,\leq w_0}^*(b_{c,w_{0}}).$$
Moreover,
$$\iota_{m,w_0}^*(j_{w_0})^*(b_{c,w_{0}})=(j_{Y,w_0})^*\iota_{m,\leq w_0}^*(b_{c,w_{0}}).$$
By purity of $Y^{I_m}_{\ua,\ua'}(w_0)$ again, the restriction morphism 
   $$\iota_{m,w_0}^*:\text{H}_*(Z_{\ua,\ua'}(w_0))\to \text{H}_*(Y^{I_m}_{\ua,\ua'}(w_0))$$
  is surjective. From the above, we see that the elements $\{\iota^*_{m,\leq w_0}({b}_{c,w_0})\;|\;c \in \B_{w_0}\}$ span a subspace of $\textbf{H}_*(Y^{I_m}_{\ua,\ua'}(\leq w_0))$ which surjects onto $\textbf{H}_*(Y^{I_m}_{\ua,\ua'}(w_0))$. The desired claim follows thanks to \eqref{eq:filtration Y}.
  
\medskip

\subsubsection{}To finish the proof of Theorem~\ref{thm:PBWbases}, it remains to show the linear independence of the set $\{\mathbf{b}_{c,w}\;|\;w \in W(\ua,\ua'), c \in \B_w\}$. Let $\mathbf{b}=\sum_{c,w} a_{c,w} \mathbf{b}_{c,w} \in \text{H}_*(Z_{\ua,\ua'})$ be a nontrivial (finite) linear combination. Let $w_0$ be a maximal element for the partial order $\leq$ on $W(\ua,\ua')$ for which there exists $c \in \B_{w_0}$ such that $a_{c,w_0} \neq 0$. Because $\{Y^{I_m}_{\ua,\ua'}(w_0)\}_m$ is an exhaustion of $Z_{\ua,\ua'}(w_0)$ by pure open substacks, there exists $m$ such that $\iota_{m,w_0}^*(\sum_c a_{c,w_0}c) \neq 0 \in \text{H}_*(Y^{I_m}_{\ua,\ua'}(w_0))$. Put 
$$X=\bigcup_{w} Z_{\ua,\ua'}(\leq w) \subset Z_{\ua,\ua'}$$
where the union ranges over the finitely many elements $w$ for which there exists $c \in \B_w$ with $a_{c,w} \neq 0$. Note that by maximality of $w_0$, $Z_{\ua,\ua'}(w_0)$ is open in $X$. Denote by $f: X \to Z_{\ua,\ua'}$ the closed embedding and $h: Z_{\ua,\ua'}(w_0) \to X$ the open embedding. By construction there exists $b'=\sum_{c,w}a_{c,w}b_{c,w} \in \text{H}_*(X)$ such that $f_*(b')=\mathbf{b}$ and $h^*(b')=\sum_c a_{c,w_0}h^*(b_{c,w_0})=\sum_c a_{c,w_0}c$. Consider the diagram with cartesian squares
$$\begin{tikzcd} 
  Z_{\ua,\ua'}(w_0)\arrow[r,hook,"h"] &X \arrow[r,"f"] &Z_{\ua,\ua'}\\
  Y^{I_m}_{\ua,\ua'}(w_0) \arrow[r,hook,"h'"]\arrow[u,hook,"\iota_{m,w_0}"]      &Y^{I_m}_{\ua,\ua'}\cap X \arrow[u,hook,"\iota'_{m}"]\arrow[r,"f'"] &Y^{I_m}_{\ua,\ua'}\arrow[u,hook,"\iota_{m}"],
  \end{tikzcd}$$
Note that $X$ being a union of cells $Z_{\ua,\ua'}(w)$, the intersection $X \cap Y^{I_m}_{\ua,\ua'}$ is still cohomologically  pure. By base change, we have
\begin{equation}\label{eq:proff inject base change}
\iota_m^*f_*(b')=f'_*(\iota_m')^*(b')
\end{equation}
and by assumption on $m$,
$$(h')^*(\iota'_m)^*(b')=\iota_{m,w_0}^*h^*(b')=\iota_{m,w_0}^*\left(\sum_c a_{c,w_0}c\right) \neq 0.$$
In particular, $(\iota'_m)^*(b') \neq 0$. By purity of $X \cap Y^{I_m}_{\ua,\ua'}$, the map $f'_*$ is injective. From \eqref{eq:proff inject base change} we deduce that $\iota_m^*(\mathbf{b})=\iota_m^*f_*(b')\neq 0$ and hence $\mathbf{b} \neq 0$ as wanted.


\bigskip

\centerline{\textbf{Acknowledgements}}

\smallskip

 We are grateful to Benjamin Hennion, Ruslan Maksimau, Sasha Minets and Eric Vasserot for numerous discussions. Part of this work was carried out during several visits of the second author to the Laboratoire de Mathématiques d'Orsay as well as to the Simion Stoilow Institute of Mathematics (IMAR), funded in part by fellowship. The second author would like to thank these institutions for their support and hospitality. The second author also gratefully acknowledges the support of Tsinghua University through its Short-Term Overseas Study Program. This work was in part supported by the PNRR grant CF 44/14.11.2022 {'COHAs of smooth surfaces and applications'}.

\medskip

\appendix

\section{Proof of Lemma~\ref{qclem}}\label{App:qclem}
  It is enough to treat the case that X is irreducible and reduced. We only need to show that $X\cap Z_i=\varnothing$ for almost all $i\in \N$. Put $\overline{Z}_{\leq i}:=\bigcup_{j\leq i}\overline{Z_j}$. Thus $(X_i:=X\cap \overline{Z}_{\leq i})_i$ is a nested sequence of closed substacks of X such that $X=\bigcup_{i}X_i$. As $X$ is quasi-compact, there exists a finite atlas $\bigsqcup_k U_k/G_k \to X$ where each $U_k$ is a finite type scheme over $\C$ and $G_k$ is an algebraic group acting on $U_k$. By pullback to $U_k$, we get a nested sequence of closed subschemes $(X_i^k)_i$ such that $\bigcup_i X_i^k=U_k$. It suffices to show that  $(X_i^k)_i$ stabilizes for each $k$. 

  \noindent
  \textit{Claim:} Let $U$ be an irreducible finite type scheme over $\C$ and let $(V_i)_{i\in \N}$ be any countable family of strict closed subschemes. Then we have $U\neq \bigcup_i V_i$.
  \begin{proof} By Noether's normalization lemma \cite[Lemma 10.115.5]{stacks-project}, there exists a finite surjection map $f:U\to \A^d_{\C}$ where $d=\dim U$. Then each $f(V_i)$ is a strict closed subset in $\A_{\C}^d$. Passing to the analytification, we see that $\A_{\C}^{d,\text{an}}\setminus f(V_i)^{\text{an}}$ is a dense open subset. By Baire category theorem, it is a Baire space since it is locally compact. Thus $\bigcap_i (\A_{\C}^{d,\text{an}}\setminus f(V_i)^{\text{an}})$ is nonempty, and then $U\setminus \bigcup_i V_i$ is nonempty.
  \end{proof}
  
  As a consequence of the above claim, the sequence $(X_i^k)_i$ stabilizes. Indeed, the inclusions $X_i^k\subseteq U_k$ cannot all be strict and then $\dim X_i^k=\dim U_k$  for some k. Since $U_k$ is irreducible, we must have $X_i^k=X_i^{k+1}=\cdots$. Note that the atlas of $X$ is finite, there exists $n$ such that $X=\overline{Z}_{\leq n}$. As $X$ is irreducible, there exists $j\leq n$ such that $Z_j$ is dense in $X$. Let $X'$ be the Zariski closure of $X\setminus Z_j$ in $X$. Thus $X'$ is a closed substack in $X$ and $\dim X'<\dim X$ and $X=Z_j\cup X'$. Repeating previous arguments with each irreducible components (finitely many) of $X'$ and noting that dimension strictly decreases each time, the process stops after finite number of iterations, as $X$ is quasi-compact and then of finite Krull dimension.


\section{Convolution in Borel-Moore homology}\label{app:convolution product}

We provide here a sheaf-theoretic description of the convolution product used to define Schur algebras $A^\nu_\alpha$. This is surely well-known to experts but we could not locate a precise reference.

Let $Y_i$ and $X$ be Artin stacks for $i=1,2,3$, with proper morphism of stacks $q_i:Y_i\to X$. Assume that $X$ is smooth, and define $Z_{ij}:=Y_i\underset{X}{\times} Y_j$. Because of the cartesian square
$$\begin{tikzcd}
    Z_{12}\underset{Y_2}\times Z_{23}\arrow[d,"p_{13}"]\arrow[r] & Y_2\arrow[d,"q_2"]\\
    Z_{13}\arrow[r] &X
\end{tikzcd}$$
the projection $p_{13}: Z_{12} \underset{Y_2}{\times}Z_{23} \to Z_{13}$ is proper. We have a commutative diagram with cartesian square
$$\begin{tikzcd} 
  Z_{12}\underset{Y_2}{\times} Z_{23}\arrow[r,"a"] \arrow[d,"q_{123}"'] &Z_{12} \times Z_{23} \arrow[d,"q_{12} \times q_{23}"'] \arrow[rd,"\pi"] &\\
  X \arrow[r,"\Delta_{X}"] & X \times X \arrow[r,"\pi'"] & pt
  \end{tikzcd}\ ,$$
  where $\Delta_{X}$ is the diagonal embedding. Since the stack $X$ is smooth,  $\Delta_X$  is a morphism between smooth stacks. Hence $\Delta_X$ is a derived lci, and there is a purity transformation: 
  $$\Delta_{X}^* \to \Delta_{X}^![-2\dim(X)],$$
see  \cite[Remark 3.8]{khan2019}. This leads to a chain of morphism
\begin{equation}\label{eq:appC chain}
\begin{split}
\pi'_*(q_{12} \times q_{23})_*(\omega_{Z_{12}} \boxtimes \omega_{Z_{23}}) & \xrightarrow{}  \pi'_*(\Delta_{X})_*\Delta_{X}^*(q_{12} \times q_{23})_*(\omega_{Z_{12}} \boxtimes \omega_{Z_{23}})\\
&\xrightarrow{} \pi'_*(\Delta_{X})_*\Delta_{X}^!(q_{12} \times q_{23})_*(\omega_{Z_{12}} \boxtimes \omega_{Z_{23}})[-2\dim(X)]\\
&\xrightarrow{\sim} \pi'_*(\Delta_{X})_*(q_{123})_*a^!(\omega_{Z_{12}} \boxtimes \omega_{Z_{23}})[-2\dim(X)]\\
&=\pi_*a_*\omega_{Z_{12} \underset{Y_2}{\times}Z_{23}}[-2\dim(X)]\ .
\end{split}
\end{equation}

Taking global sections, we obtain a morphism (intersection with supports) : 
$$ \cap:\text{H}_{*}(Z_{12},\Q) \otimes \text{H}_*(Z_{23},\Q) \to \text{H}_*(Z_{12} \underset{Y_2}{\times} Z_{23},\Q)$$
which is of homological degree $-2\dim(X)$. The convolution product is now defined as follows:
$$ \text{H}_*(Z_{12},\Q) \otimes \text{H}_*(Z_{23},\Q) \to \text{H}_*(Z_{13},\Q), \qquad c\otimes c' \mapsto (p_{13})_*(c \cap c')\ .
$$

\section{Computation of a differential}

Let $\sigma,\lambda \in (\Z^2)^+, \gamma=\sigma+\lambda$ and let
$$\xymatrix{Coh_\sigma \times Coh_\lambda & \widetilde{Coh}_{\sigma,\lambda} \ar[r]^-{p_{\sigma,\lambda}} \ar[l]_-{q_{\sigma,\lambda}} & Coh_{\gamma}}$$
be the correspondence defined in Section~\ref{sec:def correspondence}. Let us denote by $E_\sigma,E_\lambda, E_\gamma$ the tautological sheaves on $\widetilde{Coh}_{\sigma,\lambda} \times \mathbb{P}^1$ pulled back from $Coh_\sigma, Coh_\lambda$ and $Coh_\gamma$. From the short exact sequence
$$0 \to E_\lambda \to E_\gamma \to E_\sigma\to 0$$
of coherent sheaves on $\widetilde{Coh}_{\sigma,\lambda} \times \mathbb{P}^1$ we get canonical morphisms
$$f:\text{R}Hom(E_\gamma,E_\gamma) \to \text{R}Hom(E_\lambda,E_\sigma), \qquad g:\text{R}Hom(E_\sigma,E_\lambda) \to \text{R}Hom(E_\gamma,E_\gamma)\ .$$
Finally, recall that $\mathbb{T}_{Coh_\nu}=\text{R}Hom(E_\nu,E_\nu)[1]$ for any $\nu$.

The aim of this appendix is to prove the following

\smallskip

\begin{lemma}\label{lem:app differential} The differential $dp_{\sigma,\lambda}$ fits in a distinguished triangle in $D^b_{coh}(\widetilde{Coh}_{\sigma,\lambda})$
$$\mathbb{T}_{\widetilde{Coh}_{\sigma,\lambda}} \xrightarrow{dp_{\sigma,\lambda}} \text{R}Hom(E_\gamma,E_\gamma)[1] \longrightarrow \text{R}Hom(E_\lambda,E_\sigma)[1]\xrightarrow{+1} \ .$$
\end{lemma}
\begin{proof}
We have
$$\widetilde{Coh}_{\sigma,\lambda}=\text{Tot}\left(\text{R}Hom(E_\sigma,E_\lambda)[1]\right)=\text{SpecSym}_{Coh_\sigma \times Coh_\lambda}\left(\text{R}Hom(E_\sigma,E_\lambda)[1]\right)^\vee$$
and hence also a distinguished triangle
$$\text{R}Hom(E_\sigma,E_\lambda)[1] \xrightarrow{} \mathbb{T}_{\widetilde{Coh}_{\sigma,\lambda}} \xrightarrow{}\mathbb{T}_{Coh_\sigma \times Coh_\lambda} \xrightarrow{+1}\ .$$
The commutative square
$$
\begin{tikzcd}[column sep=large]
   \mathbb{T}_{\widetilde{Coh}_{\sigma,\lambda}} \arrow[r,"dp_{\sigma,\lambda}"] & \text{R}Hom(E_\gamma,E_\gamma)[1]\\
\text{R}Hom(E_\sigma,E_\lambda)[1] \arrow[u] \arrow[r,"\text{Id}"] & \text{R}Hom(E_\sigma,E_\lambda)[1] \arrow[u,"g"]
\end{tikzcd}
$$
gives rise, via the  octahedral axiom of triangulated category, to a commutative diagram with distinguished triangles:
$$
\begin{tikzcd}[column sep=large]
 ~ &~ \\
\substack{\text{R}Hom(E_\sigma,E_\sigma)[1]\\ \oplus \text{R}Hom(E_\lambda,E_\lambda)[1]} \arrow[r]\arrow[u,"+1"]& N' \arrow[r]\arrow[u,"+1"] & N\arrow[r,"+1"] &~\\
   \mathbb{T}_{\widetilde{Coh}_{\sigma,\lambda}} \arrow[u,"dq_{\sigma,\lambda}"] \arrow[r,"dp_{\sigma,\lambda}"] & \text{R}Hom(E_\gamma,E_\gamma)[1] \arrow[u] \arrow[r,"f'"] & N \arrow[u,"\text{Id}"]\arrow[r,"+1"] &~\\
\text{R}Hom(E_\sigma,E_\lambda)[1] \arrow[u] \arrow[r,"\text{Id}"] & \text{R}Hom(E_\sigma,E_\lambda)[1] \arrow[u,"g"] 
\end{tikzcd}
$$

From the middle column and top row we deduce that $N \simeq \text{R}Hom(E_\lambda,E_\sigma)[1]$, with $f'=f$ being the canonical map.
\end{proof}

\bibliographystyle{plain}

\end{document}